\documentclass[11pt]{amsart}
\usepackage[leqno]{amsmath}
\usepackage{amssymb}
\usepackage{enumerate}

\usepackage{subfigure}
\usepackage{graphicx}
\usepackage[all]{xy}
\usepackage[usenames,dvipsnames]{xcolor}
\usepackage[
colorlinks=true,linkcolor=NavyBlue,urlcolor=RoyalBlue,citecolor=PineGreen,bookmarks=true,bookmarksopen=true,bookmarksopenlevel=2,unicode=true,linktocpage]{hyperref}

\usepackage{microtype}
\usepackage{centernot}
\usepackage{stmaryrd}
\usepackage{comment}
\usepackage{bm}
\usepackage{mathrsfs}

\numberwithin{equation}{section}

\newtheorem{theorem}{Theorem}[section]
\newtheorem{remark}[theorem]{Remark}
\newtheorem{lemma}[theorem]{Lemma}
\newtheorem{proposition}[theorem]{Proposition}
\newtheorem{corollary}[theorem]{Corollary}

\newtheorem{definition}[theorem]{Definition}

\newcommand{\C}{\mathbf{C}}
\newcommand{\D}{\mathbf{D}}
\newcommand{\E}{\mathbf{E}}

\newcommand{\h}{\mathbf{H}}
\newcommand{\N}{\mathbf{N}}

\newcommand{\p}{\mathbf{P}}
\newcommand{\Q}{\mathbf{Q}}
\newcommand{\R}{\mathbf{R}}

\newcommand{\Fh}{\mathfrak {h}}

\newcommand{\CA}{\mathcal {A}}

\newcommand{\CE}{\mathcal {E}}
\newcommand{\CF}{\mathcal {F}}
\newcommand{\CI}{\mathcal {I}}

\newcommand{\CR}{\mathcal {R}}

\newcommand{\CT}{\mathcal {T}}

\newcommand{\CV}{\mathcal {V}}
\newcommand{\CW}{\mathcal {W}}

\newcommand{\CZ}{\mathcal {Z}}

\newcommand{\CG}{\mathcal {G}}
\newcommand{\CH}{\mathcal {H}}

\newcommand{\SLE}{{\rm SLE}}
\newcommand{\CLE}{{\rm CLE}}

\newcommand{\im}{\mathrm{Im}}
\newcommand{\re}{\mathrm{Re}}
\newcommand{\cov}{\mathrm{cov}}

\newcommand{\wt}{\widetilde}
\newcommand{\wh}{\widehat}
\newcommand{\ol}{\overline}
\newcommand{\ul}{\underline}

\newcommand{\RPM}{\mathsf{RPM}}

\newcommand{\BPRPM}{\mathsf{BPRPM}}

\newcommand{\BES}{\mathrm{BES}}
\newcommand{\BESQ}{\mathrm{BESQ}}

\newcommand{\strip}{{\mathscr{S}}}

\newcommand{\qmeasure}[1]{\mu_{#1}}
\newcommand{\qbmeasure}[1]{\nu_{#1}}

\newcommand{\qnt}{{\mathfrak q}}

\newcommand{\lightcone}{{\mathbf L}}

\title[$\SLE_\kappa(\rho)$ processes in the light cone regime on Liouville quantum gravity]{$\SLE_{\kappa}(\rho)$ processes in the light cone regime on\\ Liouville quantum gravity}
\author{Konstantinos Kavvadias and Jason Miller}

\begin{document}

\begin{abstract}
We study the relationship between certain $\SLE_\kappa(\rho)$ processes, which are variants of the Schramm-Loewner evolution with parameter $\kappa$ in which one keeps track of an extra marked point, and Liouville quantum gravity (LQG).  These processes are defined whenever $\rho > -2-\kappa/2$ and in this work we will focus on the light cone regime, meaning that $\kappa \in (0,4)$ and $\max(\kappa/2-4,-2-\kappa/2) < \rho < -2$.  Such processes are self-intersecting even though ordinary $\SLE_\kappa$ curves are simple for $\kappa \in (0,4)$.  We show that such a process drawn on top of an independent $\sqrt{\kappa}$-LQG surface called a weight $(\rho+4)$-quantum wedge can be represented as a gluing of a pair of trees which are described by the two coordinate functions of a correlated $\alpha$-stable L\'evy process with $\alpha = 1-2(\rho+2)/\kappa$.  Combined with another work, this shows that bipolar oriented random planar maps with large faces can be identified in the scaling limit with an $\SLE_\kappa(\kappa-4)$ curve on an independent $\sqrt{\kappa}$-LQG surface for $\kappa \in (4/3,2)$.
\end{abstract}

\date{\today}
\maketitle

\setcounter{tocdepth}{1}
\tableofcontents

\parindent 0 pt
\setlength{\parskip}{0.20cm plus1mm minus1mm}

\section{Introduction}
\label{sec:introduction}

\subsection{Overview and setting}

The purpose of this work is to study the connection between Liouville quantum gravity (LQG) surfaces and the $\SLE_\kappa(\rho)$ processes, which are variants of the Schramm-Loewner evolution ($\SLE_\kappa$) \cite{schramm2000scaling} in which one keeps track of an extra marked point \cite[Section~8.3]{lsw2003restriction}.  These processes are defined whenever $\rho > -2-\kappa/2$ and we will focus on the so-called light cone regime \cite{ms2019lightcone} which means that $\kappa \in (0,4)$ and $\max(\kappa/2-4,-2-\kappa/2) < \rho < -2$.  These processes exhibit rather different behavior from ordinary $\SLE_\kappa$.  For example, they are self-intersecting even though $\kappa \in (0,4)$.  We will begin by giving a brief overview of the objects which will be relevant for this work.

Recall that an LQG surface is a random two-dimensional Riemannian manifold which is formally described by the metric tensor
\begin{equation}
\label{eqn:lqg_form}
e^{\gamma h(z)} (dx^2 + dy^2)
\end{equation}
where $z=x+iy$, $dx^2 + dy^2$ is the Euclidean metric on a domain $D \subseteq \C$, $h$ is (some form of) the Gaussian free field (GFF) on $D$, and $\gamma \in (0,2]$ is a parameter.  Since the GFF $h$ and its variants are random variables which live in the space of distributions rather than in a space of functions, some care is required in order to make sense of~\eqref{eqn:lqg_form}.  The volume form $\qmeasure{h}$ associated with~\eqref{eqn:lqg_form} was constructed in \cite{ds2011lqgkpz} as the limit as $\epsilon \to 0$ of
\begin{equation}
\label{eqn:volume_limit}
\epsilon^{\gamma^2/2} e^{\gamma h_\epsilon(z)} dx dy	
\end{equation}
where $dx dy$ denotes Lebesgue measure and $h_\epsilon(z)$ denotes the average of $h$ on the circle $\partial B(z,\epsilon)$.  The boundary length measure $\qbmeasure{h}$ on a linear segment $L$ was also constructed in \cite{ds2011lqgkpz} as the limit as $\epsilon \to 0$ of
\begin{equation}
\label{eqn:boundary_limit}
\epsilon^{\gamma^2/4} e^{\gamma h_\epsilon(x)/2} dx	
\end{equation}
where $dx$ denotes Lebesgue measure on $L$ and $h_\epsilon(x)$ denotes the average of $h$ on the semi-circle $D \cap \partial B(x,\epsilon)$.  The measures $\qmeasure{h}$, $\qbmeasure{h}$ can also be understood as Gaussian multiplicative chaos measures in the framework developed by Kahane \cite{k1985gmc}.  The metric (two-point distance function) associated with~\eqref{eqn:lqg_form} was later constructed in the case $\gamma=\sqrt{8/3}$ in \cite{ms2020qle1,ms2021qle2,ms2021qle3} and in \cite{dddf2020tightness,gm2021metric} in the case that $\gamma \in (0,2)$ where the latter works are based on a regularization procedure analogous to~\eqref{eqn:volume_limit} while the former takes a more indirect approach.  LQG surfaces are important in statistical mechanics, string theory, and conformal field theory.

The regularization procedure~\eqref{eqn:volume_limit}, \eqref{eqn:boundary_limit} implies that the following is true.  Suppose that $h$ is an instance of (some form of) the GFF on $D$, $\varphi \colon \wt{D} \to D$ is a conformal map, and
\begin{align}
\label{eqn:change_of_coordinates}
\wt{h} = h \circ \varphi + Q \log |\varphi'| \quad\text{where}\quad Q = \frac{2}{\gamma} + \frac{\gamma}{2}.
\end{align}
Then $\qmeasure{h}(\varphi(A)) = \qmeasure{\wt{h}}(A)$ for all $A \subseteq \wt{D}$ Borel.  We say that two pairs $(D,h)$, $(\wt{D},\wt{h})$ consisting of a planar domain and a field are equivalent as quantum surfaces if they are related as in~\eqref{eqn:change_of_coordinates} and a \emph{quantum surface} is an equivalence class under the equivalence relation defined by~\eqref{eqn:change_of_coordinates}.  This definition extends to quantum surfaces which have extra marked points where one requires the conformal map $\varphi$ to take the marked points associated with $\wt{h}$ to those associated with $h$.

There has been a substantial amount of work in recent years focused on developing connections between LQG and $\SLE$ and its variants.  This started with \cite{she2016zipper} in which it was shown that it is possible to conformally weld together two independent copies $\CW_1,\CW_2$ of a certain type of quantum surface called a (weight $2$) quantum wedge and obtain a quantum wedge  $\CW$ (but with weight $4$) where the welding interface $\eta$ is given by an $\SLE_\kappa$ curve which is independent of $\CW$.  Roughly, a quantum wedge is an infinite volume surface with two marked points (the ``origin'' and the ``infinity'' point) and two boundary rays which have locally finite quantum length.  Here, it is important that the parameters $\kappa$ and $\gamma$ are matched by $\kappa = \gamma^2 \in (0,4)$.  A number of other welding results of this type were proved in \cite{dms2014mating} with different types of quantum wedges and depending on the setting one obtains as the welding interface an $\SLE_\kappa(\rho_1;\rho_2)$ process with $\rho_1, \rho_2 > -2$.  Recall that the $\SLE_\kappa(\rho)$ processes are variants of $\SLE_\kappa$ in which one keeps track of extra marked points \cite[Section~8]{lsw2003restriction} called force points and the $\rho$'s are the weights, which determine how the force points affect the behavior of the curve.  By $\SLE_\kappa(\rho_1;\rho_2)$ we typically mean the case where there is a force point of weight $\rho_1$ (resp.\ $\rho_2$) located at $0_-$ (resp.\ $0_+$).  These processes can be boundary intersecting even though ordinary $\SLE_\kappa$ for $\kappa \in (0,4)$ does not intersect the boundary.  The threshold $-2$ is important because when $\rho > -2$ an $\SLE_\kappa(\rho)$ process is absolutely continuous with respect to an ordinary $\SLE_\kappa$ process when it is away from the boundary while this is not the case for $\rho \leq -2$.

It is also shown in \cite{dms2014mating} that if one welds the two boundary rays of a quantum wedge to each other then one obtains a quantum cone, which is a surface homeomorphic to $\C$, decorated by an independent whole-plane $\SLE_\kappa(\rho)$ process.  This ultimately leads to the proof that one can explore and encode an LQG surface by an appropriate space-filling version of $\SLE_{\kappa'}$ for $\kappa' > 4$ (throughout we will use the imaginary geometry \cite{ms2016imag1} convention that $\kappa \in (0,4)$, $\kappa'=16/\kappa > 4$) which can be thought of as the Peano curve which traces the interface between a pair of continuum random trees (CRTs) $\CT_1$, $\CT_2$ drawn in the plane and whose branches are themselves whole-plane $\SLE_\kappa(\rho)$ processes (with $\rho=2-\kappa$) \cite{ms2017imag4}.  This yields the so-called ``mating of trees'' representation of LQG, whose importance is that it provides one avenue for making connections between LQG, $\SLE$, and random planar maps ($\RPM$s).  The reason for this is that a number of $\RPM$ models can be encoded using so-called tree bijections and the mating of trees representation of LQG can be thought of as a continuous analog of such a bijection.  In particular, there have been a number of scaling limit results for $\RPM$s towards $\SLE$ on LQG in which it is shown that the contour functions which encode a pair of discrete trees used to construct a planar map using a tree bijection converge in the scaling limit to the pair of contour functions which encode the continuous trees associated with space-filling $\SLE_{\kappa'}$ on LQG.  Such scaling limit results are referred to as being of ``peanosphere'' type and we will discuss this point in more detail later on.

As we mentioned earlier, the $\SLE_\kappa(\rho)$ processes with $\rho < -2$ exhibit a rather different character from the $\SLE_\kappa(\rho)$ processes with $\rho > -2$ (we exclude the case $\rho = -2$ because it is in a certain sense degenerate).  In particular, they are self-intersecting for $\kappa \in (0,4)$ even though the ordinary $\SLE_\kappa$ processes are not.  There are two regimes of $\rho$ values in which these processes are defined:
\begin{enumerate}[(i)]
\item The loop-making regime: $-2-\kappa/2 < \rho \leq \kappa/2-4$.
\item The light cone regime:	 $\max(\kappa/2-4,-2-\kappa/2) < \rho < -2$.
\end{enumerate}
The basic properties of the $\SLE_\kappa(\rho)$ processes in the loop-making regime were studied in \cite{msw2017clepercolations} and their relationship with LQG in \cite{msw2020simplecle}.  The particular case $\rho=\kappa-6$ in this regime is of special significance because it corresponds to the conformal loop ensembles ($\CLE$) \cite{s2009cle,sw2012cle}.  This case is thus motivated because it conjecturally corresponds to the scaling limit of $\RPM$ models decorated by a statistical mechanics model such as the Ising model.  In the light cone regime, the basic properties were studied in \cite{ms2019lightcone} and the purpose of the present paper is to study the relationship between these processes and LQG.  One application of this regime is that it makes it possible to connect bipolar oriented random planar maps with large faces to LQG \cite{km2022bplargefaces} using a ``peanosphere'' type scaling limit as was briefly described above.

\subsection{Main results}
\label{subsec:main_results}

We now turn to state our main results, which are stated in terms of quantum wedges and $\SLE_\kappa(\rho)$ processes in the light cone regime.  We will give the precise definition of the former in Section~\ref{subsec:lqg} and the latter in Section~\ref{subsec:slekapparho}.  The theorem statements that we will give concern an $\SLE_\kappa(\rho)$ process $\eta$ with $\max(\kappa/2-4,-2-\kappa/2) < \rho < -2$ drawn on top of an independent quantum wedge $\CW = (\h,h,0,\infty)$ and that the quantum surfaces which are parameterized by the components of $\h \setminus \eta$ together form a quantum wedge.  We emphasize that $\CW$ will be homeomorphic to $\h$ while the quantum wedge corresponding to $\h \setminus \eta$ is not homeomorphic to $\h$.  As we will describe in further detail in Section~\ref{subsec:lqg}, the starting point for the construction of this latter surface is a Bessel process $Y$ with dimension in $(0,1)$ and therefore has a well-defined local time $\ell$ for the times that it hits $0$.

\begin{theorem}
\label{thm:quantum_natural_time_cutting}
Fix $\kappa \in (0,4) , \rho \in (\kappa/2 - 4, -2) \cap (-2 -\kappa/2 , -2)$, and suppose that $\CW = (\h,h,0,\infty)$ is a quantum wedge of weight $\rho + 4$. Let $\eta$ be an $\SLE_{\kappa}(\rho)$ process in $\h$ from $0$ to $\infty$ with a single force point of weight $\rho$ located at $0_+$ and assume that $\eta$ is independent of $h$.  Then the following hold:
\begin{enumerate}[(i)]
\item The law of the beaded surface consisting of the components of $\h \setminus \eta$ which are to the right of $\eta$ is that of a quantum wedge of weight $\rho + 2$.
\item Let $\qnt_{u}$ be the first capacity time $t$ for $\eta$ that the local time at $0$ of the Bessel process $Y$ which encodes the weight $\rho + 2$ wedge is equal to $u$. For each $u > 0$, we have (as path-decorated quantum surfaces) that $(h,\eta)$ and $(h \circ f_{\qnt_{u}}^{-1} + Q\log|(f_{\qnt_{u}}^{-1})'|,f_{\qnt_{u}}(\eta))$ have the same law,  where $(f_t)$ is the forward centered Loewner flow corresponding to $\eta$.
\end{enumerate}
\end{theorem}

\begin{definition}
\label{def:quantum_natural_time}
We call $\qnt_u$ the quantum natural time parameterization of the bubbles which are cut off by $\eta$ from $\infty$.
\end{definition}

In \cite{dms2014mating}, a similar definition of quantum natural time for $\SLE_{\kappa'}$ processes with $\kappa' \in (4,8)$ is given.  In contrast to the case of $\SLE_{\kappa'}$ processes, an $\SLE_\kappa(\rho)$ process as in Theorem~\ref{thm:quantum_natural_time_cutting} parameterized by quantum natural time is not continuous, the reason being that as the process draws the boundary of a large component it does not discover any small components.

\begin{figure}[ht!]
\begin{center}
\includegraphics[scale=0.85]{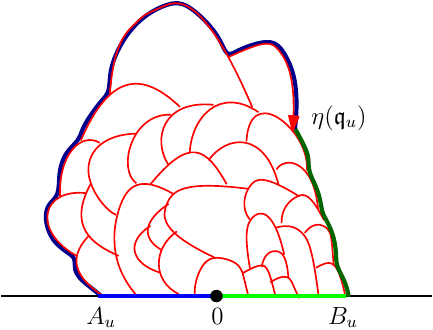}
\end{center}
\caption{\label{fig:boundary_length_evolution} Shown is an $\SLE_\kappa(\rho)$ process in $\h$ from $0$ to $\infty$ with a single force point located at $0_+$ in the light cone regime and drawn up to a typical quantum natural time $\qnt_u$.  The arc which disconnects each component from $\infty$ is drawn in its entirety before $\eta$ ``creeps'' up it disconnects further disks.  Its left boundary length $L_u$ is equal to the quantum length of the dark blue arc (counterclockwise from $\eta(\qnt_u)$ to $A_u$) minus the quantum length of the light blue arc ($[A_u,0]$).  Similarly, its right boundary length $R_u$ is equal to the quantum length of the dark green arc (clockwise from $\eta(\qnt_u)$ to $B_u$) minus the quantum length of the light green arc ($[0,B_u]$).  We prove in Theorem~\ref{thm:boundary_length_evolution} that $(L,R)$ evolves as an $\alpha$-stable L\'evy process where $\alpha$ is determined by $\rho$ as in~\eqref{eqn:alpha_value}.}
\end{figure}

Theorem~\ref{thm:quantum_natural_time_cutting} allows us to describe the law of the boundary length evolution of $\eta$ when it has the quantum natural time parameterization.  More precisely,  for each $u > 0$ we let $A_u$ (resp.\ $B_u$) be the leftmost (resp.\ rightmost) point of intersection of $\eta([0,\qnt_u])$ with $\R_-$ (resp.\ $\R_+$).  Let also $L_u$ be the difference between the quantum length of the segment of the outer boundary of $\eta([0,\qnt_u])$ connecting $\eta(\qnt_u)$ and $A_u$ with the quantum length of $[A_u,0]$.  Similarly, let $R_u$ be the difference between the quantum length of the segment of the outer boundary of $\eta([0,\qnt_u])$ connecting $\eta(\qnt_u)$ and $B_u$ with the quantum length of $[0,B_u]$.  Then the following theorem describes the evolution of $(L,R)$.  Recall that $(L,R)$ is an $\alpha$-stable L\'evy process if it has stationary and independent increments and for each $c > 0$ we have that $(L_u,R_u)_{u \geq 0}$ and $(c^{-1/\alpha}L_{c u},c^{-1/\alpha}R_{c u})_{u \geq 0}$ have the same law.

\begin{theorem}
\label{thm:boundary_length_evolution}
Suppose that we have the setup described in the statement of Theorem~\ref{thm:quantum_natural_time_cutting} and the process $(L,R)$ is as described just above.  Then $(L,R)$ is an $\alpha$-stable L\'evy process with
\begin{align}\label{eqn:alpha_value}
\alpha = 1 - \frac{2(\rho+2)}{\kappa} \in (1,2).
\end{align}
Moreover,  $L$ (resp.\ $R$) has only upward (resp.\ downward) jumps and the jump times of $L$ coincide with the jump times of $R$.  In the case that $\kappa \in (\frac{4}{3},2)$ and $\rho = \kappa - 4$ the sequence of jumps of $(L,-R)$ forms a Poisson point process (p.p.p.) $\Lambda^*$ whose law can be sampled from as follows.
\begin{itemize}
\item Firstly,  we sample a p.p.p.\  $\Lambda = \{(s_j,t_j) : j \in \N \}$ on $\R_+ \times \R_+$ with intensity measure given by $c(du \times t^{-4/\kappa}dt)$,  where $c > 0$ is a constant depending only on $\kappa$.
\item Next, we sample independently a sequence of i.i.d.\ random variables $(u_j)_{j \in \N}$ which are uniform on $[0,1]$ and consider 
\begin{align*}
\Lambda^* = \{(t_j u_j,  (1-u_j)t_j) : (s_j,t_j) \in \Lambda,  j \in \N\}.
\end{align*}
\end{itemize}
\end{theorem}

Note that by \cite{bertoin1996levy},  the law of a two-dimensional $\alpha$-stable L\'evy process is determined by the law of its jumps.  Hence when $\rho = \kappa - 4$,  Theorem~\ref{thm:boundary_length_evolution} gives a complete description of the law of the process $(L,R)$.  Moreover,  combining Theorem~\ref{thm:boundary_length_evolution} with the one-dimensional $\text{KPZ}$ formula for quantum boundary length,  we obtain the following formula for the Hausdorff dimension of $\eta \cap \R_+$.

\begin{corollary}
\label{cor:dim_of_boundary_intersection}
Fix $\kappa \in (0,4),  \rho \in (\kappa/2-4,-2) \cap (-2 - \kappa/2,-2)$ and let $\eta$ be an $\SLE_{\kappa}(\rho)$ process in $\h$ from $0$ to $\infty$ with its force point located at $0_+$.  Then, a.s.\  the Hausdorff dimension of the set $\eta \cap \R_+$ is given by 
\begin{equation}
\label{eqn:dim_of_intersection_with_R_+}
-\frac{(2+\rho)(\kappa + 8 + 2\rho)}{2\kappa}.
\end{equation}
\end{corollary}

We note that the analog of Corollary~\ref{cor:dim_of_boundary_intersection} for $\rho > -2$ was given in \cite{mw2017slepaths} and for $-2-\kappa/2 < \rho \leq \kappa/2-4$ follows from the loop-trunk decomposition of such processes given in \cite{msw2017clepercolations} combined with the $\rho > -2$ case given in \cite{mw2017slepaths}.  We note that for $\rho = \kappa/2-4$ the value from~\eqref{eqn:dim_of_intersection_with_R_+} is equal to $2-8/\kappa'$ where $\kappa'=16/\kappa$.  This, in turn, is equal to the dimension of the intersection of an $\SLE_{\kappa'}$ process in $\h$ with $\partial \h$ determined in \cite{as2011covariant}.  This is not a coincidence as it was shown in \cite{ms2019lightcone} that the law of the range of an $\SLE_\kappa(\kappa/2-4)$ process agrees with that of an $\SLE_{\kappa'}(\kappa'/2-4)$ process (but the corresponding curves visit the points in their range in a different order).

\subsection{Mating of trees interpretation}
\label{subsec:interpretation}

We will now describe how Theorems~\ref{thm:quantum_natural_time_cutting} and~\ref{thm:boundary_length_evolution} have the interpretation of giving a mating of trees representation of $\SLE_\kappa(\rho)$ processes in the light cone regime drawn on top of LQG, how it contrasts with the mating of trees representations given in \cite{dms2014mating}, and how this is related to the scaling limits of certain types of $\RPM$s.

\begin{figure}[ht!]
\begin{center}
\includegraphics[scale=0.85]{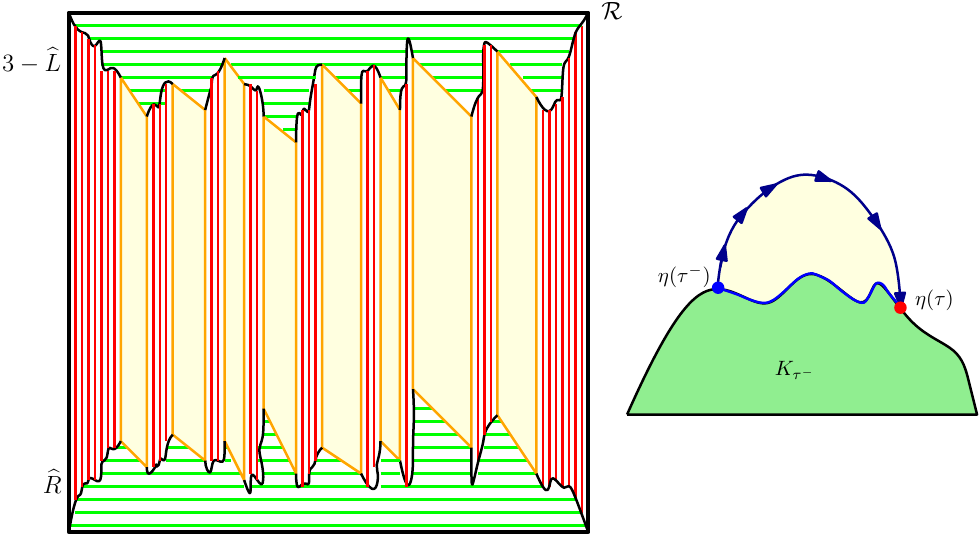}
\end{center}
\caption{\label{fig:mot_representation} Illustration of the ``mating of trees representation'', which describes the topology of an $\SLE_\kappa(\rho)$ process $\eta$ drawn on top of an independent quantum wedge of weight $\rho+4$.  {\bf Left:} Points on the graph of $\wh{R}$ are equivalent if they can be connected by a horizontal line which is below the graph of $\wh{R}$ and points on the graph of $3-\wh{L}$ are equivalent if they can be connected by a horizontal line which is above the graph of $3-\wh{L}$.  Points on the graphs of $\wh{R}$ and $3-\wh{L}$ are equivalent if they can be connected by a vertical line which does not correspond to a jump time of $(L,R)$.  The jump times $(L,R)$ are the yellow regions and in the quotient correspond to doubly marked disks which are cut out by $\eta$.  The opening (resp.\ closing) point of such a disk $D$ is given by the projection of the left (resp.\ right) vertical boundary of the corresponding yellow region $Y$ and the left (resp.\ right) boundary of $D$ is given by the projection of the part of $\partial Y$ which is on the graph of $3-\wh{L}$ (resp.\ $\wh{R}$). {\bf Right:} The time $\tau$ is a jump time for $L$, $R$, which corresponds to a time at which $\eta$ disconnects a topological disk from $\infty$ (shown in yellow); $K_{\tau-}$ is the hull cut out by $\eta$ up to just before time $\tau$.  The opening point of the disk is $\eta(\tau^-)$ (blue) and the closing point is $\eta(\tau)$ (red).  At the time $\tau$, $R$ (resp.\ $L$) makes a downward (resp.\ upward) jump of size equal to the quantum length  counterclockwise (resp.\ clockwise) arc of the disk  boundary from the opening point to the closing point.}
\end{figure}

Suppose that $(L,R)$ is an $\alpha$-stable L\'evy process with law as in Theorem~\ref{thm:boundary_length_evolution} so that $L$ (resp.\ $R$) has only upward (resp.\ downward) jumps and that the jump times of $L$ and $R$ coincide.  Let $\wt{L}$ (resp.\ $\wt{R}$) be the continuous process obtained by starting with $L$ (resp.\ $R$) and then replacing each jump with a linear segment whose length is equal to square of the jump made by $L-R$ at this time.  Since the sum of the squares of the jumps of an $\alpha$-stable L\'evy process with $\alpha \in (1,2)$ on any compact time interval is finite, it follows that $(\wt{L},\wt{R})$ is well-defined.  Let $\phi$ be a homeomorphism from $\R$ to $(-1,1)$ which takes $-\infty$ (resp.\ $\infty$) to $-1$ (resp.\ $1$) and fixes $0$.  Then the processes $\wh{L}_u = \phi(\wt{L}_{\phi^{-1}(u)})$, $\wh{R}_u = \phi(\wt{R}_{\phi^{-1}(u)})$ are defined on $[0,1)$ and take values in $(-1,1)$.  Let $\CI$ be the image under $\phi$ of those intervals where $\wt{L}$, $\wt{R}$ are linear.

Suppose that we draw the graph of $\wh{R}$ and the graph of $3-\wh{L}$ in the rectangle $\CR = [0,1] \times [-1,4]$.  Let $\CA$ consist of those vertical lines $[(u, \wh{R}_u), (u, 3-\wh{L}_u)]$ where $u \in \CI$.  We consider the finest equivalence relation $\sim$ on $\CR$ such that points in $\CR \setminus \CA$ which can be connected by:
\begin{itemize}
\item a horizontal line which lies below the graph of $\wh{R}$, or
\item a horizontal line which lies above the graph of $3-\wh{L}$, or
\item a vertical line which connects the graphs of $\wh{R}$ and $3-\wh{L}$ whose $x$-coordinate is not in $\CI$.
\end{itemize}
Arguing as in \cite{dms2014mating}, it is possible to check using Moore's theorem that the topological space $\CR / \sim$ is homeomorphic to $\h$.  Moreover, one can consider the continuous curve $\eta$ which is given by the projection of the curve $u \mapsto 3 - \wh{L}_u$ under the quotient map.  Then $\eta$ is homeomorphic to an $\SLE_\kappa(\rho)$ process where $\rho$ is determined by $\alpha$ as in~\eqref{eqn:alpha_value}.  If $I$ is a connected component of $\CI$ (i.e., $I = (a,b)$ for some $-1 < a < b < 1$), then the topological disk given by the union of the vertical segments $[(u, \wh{R}_u), (u, 3-\wh{L}_u)]$ for $u \in I$ corresponds to a component which is disconnected from $\infty$ by $\eta$.

Recall that a planar map is a graph $G = (V,E)$ together with an embedding into the plane so that no two edges cross,  considered up to orientation preserving homeomorphism.  A $\RPM$ is a planar map chosen according to some probability measure.  Examples of $\RPM$s include planar triangulations and quadrangulations chosen uniformly at random \cite{lg2013bm,m2013bm,lg2019survey,miermontstflour} and weighted by the partition function of various models from statistical mechanics \cite{s2016qginv,kmsw2019bipolar,lsw2017schnyder,gm2021saw,gm2017percolation,gkmw2018active}.  One approach to study $\RPM$s is to encode them in terms of random trees and random walks via combinatorial bijections.  An important class of such bijections are the so-called \emph{mating of trees bijections} which represent a planar map decorated by a statistical mechanics model as the gluing of a pair of discrete trees.  Since discrete trees can be represented in terms of their contour function, such an encoding is equivalent to encoding the map by a two-dimensional walk.  The construction of LQG surfaces using matings of trees \cite{dms2014mating} therefore leads to a natural topology on surfaces which is called the \emph{peanosphere topology}.  In particular, a surface decorated by a tree can be encoded in terms of a pair of continuous functions $(L,R)$ where $L$ (resp.\ $R$) is given by the contour function of the tree (resp.\ dual tree) on the surface.  Then if we have two tree-decorated surfaces with associated pairs of contour functions $(L,R)$ and $(L',R')$,  we define the distance between the two surfaces to be the $L^\infty$ distance between $(L,R)$ and $(L',R')$.  In other words, the peanosphere topology is the restriction of the $L^\infty$ metric to the space of continuous functions which arise as the contour functions related to tree-decorated surfaces.  Since applying a rescaling to a planar map corresponds to applying a rescaling to the discrete pair of trees encoding the map, we say that a sequence of rescaled $\RPM$s converges in distribution to a tree-decorated random quantum surface as the number of vertices (or faces) of the map tends to infinity, if the rescaled random walks encoding the $\RPM$s converge in distribution with respect to the~$L^\infty$ metric to the pair of continuous functions $(L,R)$ encoding the tree-decorated quantum surface. It was shown in \cite{kmsw2019bipolar} that if we pick a bipolar oriented planar map uniformly at random with fixed boundary lengths and whose faces are all triangles then the corresponding rescaled random walk converges in the limit to a correlated two-dimensional Brownian motion.  Using the construction related to the peanosphere topology, this result can be interpreted as being a scaling limit result towards the $\SLE$ with parameter $\kappa = 12$ on a $\sqrt{4/3}$-LQG surface.

The case of random bipolar oriented planar maps ($\BPRPM$s) with \emph{large faces}, meaning that the law of the face degree is in the domain of attraction of an $\alpha$-stable random variable, are considered in \cite{km2022bplargefaces}. It is shown in \cite[Theorem 1.3]{km2022bplargefaces} that the contour functions for the pair of discrete trees which encode such a $\BPRPM$ converge in the scaling limit to an $\alpha$-stable L\'evy process with the same law as in Theorem~\ref{thm:boundary_length_evolution} in the case that $\rho = \kappa-4$ and $\alpha=4/\kappa-1$. This allows us to interpret $\SLE_\kappa(\kappa-4)$ drawn on top of a weight $\rho+4$ quantum wedge as the scaling limit of $\BPRPM$s with large faces in the peanosphere sense.

\subsection*{Outline}  The remainder of this article is structured as follows.  In Section~\ref{sec:preliminaries}, we will collect a number of preliminaries.  Next, in Section~\ref{sec:slegff_couplings} we will collect some results on the $\SLE$/GFF coupling.  Finally, we will complete the proofs of our main theorems in Section~\ref{sec:main_theorems_proof}.

\subsection*{Acknowledgements} K.K.'s work was supported by the EPSRC grant EP/L016516/1 for the University of Cambridge CDT (CCA) and by ERC starting grant 804166 (SPRS).  J.M.'s work was supported by ERC starting grant 804166.

\section{Preliminaries}
\label{sec:preliminaries}

The purpose of this section is collect a number of preliminaries which will be used in the remainder of this article.  In Section~\ref{subsec:bessel}, we will give a brief review of Bessel processes.  Next, in Section~\ref{subsec:gff} we will give a review of the GFF and in Section~\ref{subsec:lqg} of LQG.  In Section~\ref{subsec:slekapparho} we will review the $\SLE_\kappa(\rho)$ processes and in Section~\ref{subsec:ig} the aspects of imaginary geometry which are relevant for this work.  Finally, in Section~\ref{subsec:light_cones} we will review the $\SLE_\kappa(\rho)$ processes in the light cone regime in the imaginary geometry framework.

\subsection{Bessel processes}
\label{subsec:bessel}

In this subsection,  we recall a few facts about Bessel processes which will play an important role in this paper.  We refer the reader to \cite[Chapter XI]{revuz2013continuous} for a more in depth overview of Bessel processes.

Fix $\delta \in \R$ and $x \geq 0$.  The squared $\delta$-dimensional Bessel process ($\BESQ^\delta$) starting from $x^2$ is given by the unique strong solution to the SDE
\begin{equation}\label{eq:square_bessel_sde}
dZ_t = \delta dt +  2\sqrt{Z_t} dB_t,\quad Z_0 = x^2
\end{equation}
where $B$ is a standard Brownian motion.  If we want to emphasize the starting point of a $\BESQ^\delta$ process we will write $\BESQ_{x^2}^\delta$.  Standard results for SDEs imply that there is a unique strong solution to~\eqref{eq:square_bessel_sde} up until the first time that $Z$ hits $0$.  When $\delta > 0$,  there exists a unique strong solution defined for all times which always remains non-negative.  Then for $\delta > 0$, the \emph{Bessel process} of dimension $\delta$ ($\BES^{\delta}$),  is the process $X_t = \sqrt{Z_t}$, where $Z$ is the unique strong solution to~\eqref{eq:square_bessel_sde}.  If we want to emphasize the starting point of a $\BES^\delta$ process we will write $\BES_x^\delta$.

If $\delta \geq 2$,  then a.s.\ $X_t > 0$ for all $t > 0$ while if $\delta \in (0,2)$,  then $X_t$ a.s.\ assumes the value zero on a non-empty random set with zero Lebesgue measure.  Also,  for all $\delta > 0$,  the process $X_t$ is invariant under Brownian scaling,  i.e.,  for every given constant $c > 0$,  the processes $(c^{-1/2}X_{ct})$ and $(X_t)$ have the same law.  Moreover,  when $\delta > 1$, the process $X_t$ is a semimartingale and a strong solution to the SDE
\begin{equation}
\label{eq:bessel_sde}
d X_t =  \frac{\delta-1}{2} \cdot \frac{1}{X_t} dt + d B_t, \quad X_0 = x
\end{equation}
and when $\delta =1$,  the process $X_t$ is equal in distribution to $|B|$ and still a semimartingale. However, when $\delta \in (0,1]$,  \eqref{eq:bessel_sde} holds up until the first time $t$ such that $X_t = 0$,  but it does not hold for larger $t$ since the integral $\int_{0}^{t} X_s^{-1}ds$ is infinite a.s.  In order to make sense of a solution to~\eqref{eq:bessel_sde} in integrated form for $\delta \in (0,1)$,  we make a principal value correction.  More precisely,  $X$ satisfies the equation
\begin{equation}
\label{eq:principal_value_correction_equation}
X_t = x + \frac{\delta-1}{2}\text{P.V.}\int_{0}^{t}\frac{1}{X_s} ds + B_t.
\end{equation}
The principal value correction can be expressed in terms of an integral of the local time processes associated with $X$ (see \cite[Chapter XI]{revuz2013continuous}).

Next,  we recall the \emph{approximate Bessel processes} for $\delta \in (0,1)$ considered in \cite[Section 6]{s2009cle}.  For fixed $\epsilon > 0$,  we define an $\epsilon$-$\BES^{\delta}$ process $X_t^{\epsilon}$ to be the Markov process beginning at $x > 0$ that solves~\eqref{eq:bessel_sde} except that each time it hits zero it immediately jumps to $\epsilon$ and continues.  Then we have that
\begin{equation}
\label{eq:approximate_bessel_sde}
X_t^{\epsilon} = x + \frac{\delta-1}{2} \int_{0}^{t}\frac{1}{X_s^{\epsilon}}ds + B_t + J_t^{\epsilon}
\end{equation}
for all times $t \geq 0$, where $J_t^{\epsilon}$ is $\epsilon$ times the number of $\epsilon$-jump discontinuities of $X_t^{\epsilon}$ up to and including time $t$.  Then if $t_i$ is the time of the $i$th jump,   the following was shown in \cite{s2009cle}:

\begin{proposition}\label{prop:epsilon_bessel}
As $\epsilon \to 0$,  the $X_{t}^{\epsilon}$ converge in law to a $\BES_{x}^{\delta}$ with respect to the $L^{\infty}$ metric on a fixed interval $[0,T]$, with $T>0$. As $\epsilon \to 0$ we have that
\begin{enumerate}[(i)]
\item  $J_{T}^{\epsilon} \to 0$ if $\delta>1$,
\item  $J_{T}^{\epsilon} \to \infty$ if  $0<\delta<1$, and
\item\label{it:jtsquared}  $J_{T}^{\epsilon^{2}} := \sum_{t_i \leq T} \epsilon^{2} \to 0$ for all $\delta>0$.
\end{enumerate}
\end{proposition}

\subsection{Gaussian free fields}
\label{subsec:gff}

Let $D \subseteq \C$ be a simply connected domain with harmonically non-trivial boundary and let $H_{0}(D)$ be the Hilbert space closure of $C_{0}^{\infty}(D)$ with respect to the Dirichlet inner product 
\[ (f,g)_{\nabla} = \frac{1}{2\pi}\int_{D}\nabla f(z) \cdot \nabla g(z) dz.\]
Then the zero boundary Gaussian free field (\text{GFF}) $h$ is the random distribution defined by 
\begin{equation}\label{eq:gff_expression}
h = \sum_{n \geq 1}\alpha_n \phi_n
\end{equation}
where $(\phi_n)_{n \geq 1}$ is an orthonormal basis of $H_{0}(D)$ with respect to $(\cdot,\cdot)_{\nabla}$ and $(\alpha_n)_{n \geq 1}$ is a sequence of i.i.d.\ random variables with distribution $N(0,1)$.  Note that a \text{GFF} on $D$ should not be considered as a function but rather as a random variable taking values in the space of distributions.  

The above construction also applies for other variants of the GFF.  In particular, a free boundary \text{GFF} is defined in the same way except that we replace $H_{0}(D)$ by the closure $H(D)$ with respect to $(\cdot,\cdot)_{\nabla}$ of the space of functions $f \in C^{\infty}(D)$ such that $\int_{D}f(z)dz = 0$.  Note that in this way,  the free boundary \text{GFF} is defined in a space of distributions modulo additive constant.  However,  we can fix the additive constant for the field by fixing its value when acting on a fixed test function which does not have mean zero.  We also note that we can define the free boundary \text{GFF} in the following way.  Let $G_{D}^{N}$ be the Green's function with Neumann boundary conditions on $\partial D$ and recall that
\begin{equation}
\label{eqn:neumann_greens}
G_{\h}^{N}(z,w) = -\log |z-w| - \log |z-\overline{w}|.
\end{equation}
Then, we can define the free boundary \text{GFF} on $\h$ as the centered Gaussian process with covariance kernel~$G_{\h}^{N}$.  We can also define the free boundary \text{GFF} in other simply connected domains by conformally mapping a free boundary \text{GFF} on~$\h$.

\subsection{Liouville quantum gravity}
\label{subsec:lqg}

Recall the definition of the \text{LQG} surfaces given in Section~\ref{sec:introduction}. We can also consider \text{LQG} surfaces with marked points $(D,h,x_1,\dots,x_n)$,  $(\wt{D},\wt{h},\wt{x}_1,\dots,\wt{x}_n)$ where we consider them to be equivalent if~\eqref{eqn:change_of_coordinates} holds and $x_j = \varphi(\wt{x}_j)$ for all $j=1,\dots,n$.

Next we give the definition of a so-called \emph{thick} quantum wedge, which is a type of quantum surface which is homeomorphic to $\h$ and has two marked points: the \emph{origin} and \emph{infinity}.  Compact neighborhoods of the former have finite quantum area while any neighborhood of the latter has infinite quantum area.  We recall that a quantum surface is an equivalence class under the equivalence relation defined by~\eqref{eqn:change_of_coordinates}.  Thus to specify the law of a quantum surface, we are free to choose which domain we will use to parameterize it.  In the case of a quantum wedge, the most convenient choice is the infinite strip $\strip = \R \times (0,\pi)$.

\begin{definition}
\label{def:thick_quantum_wedge}
Fix $\gamma \in (0,2)$, $W > \gamma^2 / 2$, and set $\alpha = \frac{\gamma}{2} + Q - \frac{1}{\gamma}W$.  A quantum wedge of weight $W$ is the quantum surface together with two marked boundary points at $-\infty$ and $+\infty$ (when parameterized by $\strip$) described by the distribution $h$ on $\strip$ whose law can be sampled from as follows.  Let $X_s$ be the process defined for $s \in \R$ such that
\begin{itemize}
\item For $s \geq 0$,  $X_s = B_{2s} + (Q - \alpha)s$ where $B$ is a standard Brownian motion with $B_0 = 0$ conditioned so that $B_{2u} + (Q - \alpha)u > 0$ for all $u > 0$ and
\item For $s < 0$,  $X_s = \wh{B}_{-2s} + (Q - \alpha)s$ where $\wh{B}$ is a standard 	Brownian motion independent of $B$ with $\wh{B}_0 = 0$.
\end{itemize}

Let $\mathcal{H}_1(\strip)$ be the subspace of $H(\strip)$ which consists of those functions which are constant on vertical lines of the form $u + [0,i\pi]$ for $u \in \R$ and let $\mathcal{H}_2(\strip)$ be the subspace of $H(\strip)$ which consists of those functions which have a common mean on all such vertical lines.  Then $h$ is the field with projection onto $\mathcal{H}_1(\strip)$ given by the function whose common value on $u + [0,i\pi]$ is $X_u$ for $u \in \R$ and its projection onto $\mathcal{H}_2(\strip)$ is given by the corresponding projection of a free boundary GFF on $\strip$.  We fix the additive constant for the projection of $h$ onto $\mathcal{H}_1(\strip)$ (resp.\ $\mathcal{H}_2(\strip)$) so that its average on $[0,i\pi]$ vanishes.
\end{definition}

\begin{remark}
The above embedding of a quantum wedge into $\strip$ is the so-called \emph{circle average embedding}.  Note that we can change the embedding $\strip$ while keeping the marked points fixed by applying~\eqref{eqn:change_of_coordinates} and using a conformal transformation $\phi : \strip \rightarrow \strip$ fixing $-\infty$ and $+\infty$, and this corresponds to translating the field $h$ by a constant $r \in \R$.
\end{remark}

Throughout this paper, we will use the notation $\langle \cdot \rangle$ in order to denote the quadratic variation.

\begin{remark}
Suppose that $X$ is as in Definition~\ref{def:thick_quantum_wedge} and let $Z_t = \exp( \gamma X_t/2)$.  By It\^o's formula,
\begin{align*}
d\langle Z \rangle_{t} = \frac{\gamma^2}{2} Z_t^2 dt.
\end{align*}
By \cite[Proposition 3.4]{dms2014mating},  we see that if we reparameterize $Z$ by its quadratic variation then it evolves as a $\BES^{\delta}$ with 
\begin{align}
\label{eqn:bessel_dimension_wrt_weight}
\delta = 2 + \frac{2(Q-a)}{\gamma} = 1 + \frac{2}{\gamma^2}W.
\end{align}
This gives another way of sampling $h$ as follows.  Start with a $\BES^{\delta}$ process $Z$ with $\delta$ as in~\eqref{eqn:bessel_dimension_wrt_weight}.  Then we sample the projection of $h$ onto $\mathcal{H}_1(\strip)$ to be given by reparameterizing $2\gamma^{-1}\log(Z)$ to have quadratic variation $2dt$ and taking the horizontal translation so that it last hits $0$ at time $0$.  The projection onto $\mathcal{H}_2(\strip)$ is sampled as in Definition~\ref{def:thick_quantum_wedge} and independently of $Z$.  This is how we define the $a = Q$,  $W = \gamma^2/2$ quantum wedge.  When $\delta \geq 2$,  we call the resulting wedge a \emph{thick} quantum wedge.
\end{remark}

We can also consider quantum surfaces parameterized by a closed set (not necessarily homeomorphic to a disk) such that each component of its interior together with its prime-end boundary is homeomorphic to the closed unit disk,  and $h$ is defined as a distribution on each of these components.

\begin{definition}
\label{def:thin_quantum_wedge}
Fix $\gamma \in (0,2)$ and $W \in (-\gamma^2 / 2,  \gamma^2 / 2)$.  A quantum wedge of weight $W$ on $\strip$ (with marked points at $-\infty$ and $+\infty$) is the random distribution $h$ on $\strip$ that can be sampled as follows.  Let $Y$ be a $\BES^{\delta}$ process where $\delta$ is as in~\eqref{eqn:bessel_dimension_wrt_weight}.  Let $\mathcal{H}_1(\strip),  \mathcal{H}_2(\strip)$ be as before.  We sample a countable collection of distribution valued random variables $h_{e}$ on $\strip$ indexed by the excursions~$e$ of~$Y$ from~$0$ and take the projection of $h_{e}$ onto $\mathcal{H}_1(\strip)$ to be given by reparameterizing $2\gamma^{-1}\log(e)$ to have quadratic variation $2dt$ and the projection of $h_{e}$ onto $\mathcal{H}_2(\strip)$ to be sampled according to the law of the corresponding projection of a free boundary GFF on $\strip$ (with additive constant fixed so that its average on $[0,i\pi]$ vanishes) taken to be independent of the corresponding projection of the other excursions from $0$ of $Y$.  Note that $\delta \in (0,2)$ and in that case we call the quantum wedge \emph{thin} quantum wedge.
\end{definition}

In \cite{dms2014mating}, the thin quantum wedges are defined only when $W \in (0,\gamma^2/2)$.  The definition given in Definition~\ref{def:thin_quantum_wedge} is the same as in \cite{dms2014mating} except it is for $W \in (-\gamma^2/2,\gamma^2/2)$.  The difference between the regime $W \in (-\gamma^2/2,0]$ and $W \in (0,\gamma^2/2)$ is that for the latter the sum of the boundary lengths of the bubbles which make up the wedge is locally finite while in the former this is not the case.

\begin{remark}
Let $D \subseteq \C$ be a simply connected domain.  Then we can consider a quantum wedge on $D$ with distinct marked points $x,y \in \partial D$ by starting with a quantum wedge on $\strip$ and then applying~\eqref{eqn:change_of_coordinates} for some conformal transformation $\phi : \strip \rightarrow D$ such that $\phi(-\infty) = x$ and $\phi(+\infty) = y$.
\end{remark}

\subsection{$\SLE_{\kappa}(\rho)$ processes with $\kappa > 0$ and $\rho > -2 - \kappa/2$}
\label{subsec:slekapparho}

We will now review the definition of the $\SLE_\kappa(\rho)$ processes, first focusing on the forward case and then on the reverse case.  In order to differentiate between the former and the latter, we will use tildes for quantities associated with the reverse case.

\subsubsection{Forward $\SLE_\kappa(\rho)$}
Fix $\kappa > 0$,  $\rho > -2 - \kappa/2$, and let $\delta$ be such that
\begin{equation}
\label{eqn:delta_and_rho}
\delta = 1 + \frac{2(\rho + 2)}{\kappa}.
\end{equation}
Let $X$ be a $\BES^{\delta}$ process and let 
\begin{align*}
V_t = \frac{2}{\sqrt{\kappa}}\text{P.V.}\int_{0}^{t} \frac{1}{X_s} ds\quad \text{and} \quad W = V - \sqrt{\kappa}X.
\end{align*}
Then an $\SLE_{\kappa}(\rho)$ process is described by the solution to the \text{ODE}
\begin{align*}
\partial_{t}g_t(z) = \frac{2}{g_t(z) - W_t},\quad g_0(z) = z.
\end{align*}
For each $t \geq 0$ we let $\h_t$ be the domain of $g_t$ and $K_t = \h \setminus \h_t$ be the hull at time $t$ associated with the $\SLE_\kappa(\rho)$ process.  The location of the force point at time $0$ is given by $V_0$ and at time $t$ is given by $g_t^{-1}(V_t)$. 

When $\rho > -2$, we have that $\delta > 1$ and so the force point at time $t$ is located at the rightmost intersection of the hull at time $t$ with $\R$.  The continuity of the process in that case was shown in \cite{ms2016imag1}.  When $\rho \in (-2-\kappa/2,-2)$,  the force point does not stay in $\R$ because of the principal value correction and since $\delta \in (0,1)$.  The continuity of the process in the case that $\rho \in (-2-\kappa/2,\kappa/2-4]$ and $\kappa \in (2,4)$ was shown in \cite{msw2017clepercolations} while the continuity when $\rho \in (-2-\kappa/2,-2) \cap (\kappa/2-4,-2)$ and $\kappa \in (0,4)$ was shown in \cite{ms2019lightcone}. 

We can more generally consider the $\SLE_\kappa(\rho)$ processes with multiple force points, which we will denote by $\SLE_\kappa(\ul{\rho})$.  More specially, suppose that $\ul{x} = (x_1,\ldots,x_n)$ are distinct points in $\partial \h$ (where we consider $0_-$ and $0_+$ to be distinct) and that $\ul{\rho} = (\rho_1,\ldots,\rho_n) \in \R^n$.  Suppose that $B$ is a standard Brownian motion and let $W$ be the solution to the SDE
\begin{align}
\label{eqn:sle_kappa_rho_sde}
dW_t = \sqrt{\kappa} dB_t + \sum_{i=1}^n \frac{\rho_i}{W_t - V_t^i} dt, \quad dV_t^i = \frac{2}{V_t^i-W_t} dt, \quad V_0^i = x_i.
\end{align}
It is shown in \cite{ms2016imag1} that there is a unique solution~\eqref{eqn:sle_kappa_rho_sde} up until the so-called \emph{continuation threshold}, which is the first time $t$ that $\sum_{i : V_t^i = W_t} \rho_i = -2$ and that the corresponding Loewner flow is generated by a continuous curve up until this time.

\subsubsection{Reverse $\SLE_\kappa(\wt{\rho})$}
\label{subsubsec:reverse_sle_kappa_rho}

Fix $\kappa > 0$, $\wt{\rho} > 2-\kappa/2$, and let
\begin{equation}
\label{eqn:delta_and_rho_reverse}
\wt{\delta} = 1 + \frac{2(\wt{\rho}-2)}{\kappa}.
\end{equation}
Let
\[ \wt{V}_{t} = -\frac{2}{\sqrt{\kappa}}\text{P.V.} \int_{0}^{t}\frac{1}{\wt{X}_{s}}ds\quad\text{and}\quad \wt{W} = \wt{V} - \sqrt{\kappa}\wt{X}.\]
A reverse $\SLE_{\kappa}(\wt{\rho})$ (with the force point located at $0^{+}$) centered Loewner flow is defined to be  the  family of conformal maps ($\wt{f_{t}}$) given by $\wt{g_{t}}(z) - \wt{W_{t}}$ where ($\wt{g_{t}}$) solve the reverse Loewner equation
\begin{align}
\label{eqn:rev_loewner}
    \partial_{t}{\wt{g}_{t}}(z) = -\frac{2}{\wt{g}_{t}(z) - \wt{W}_{t}},\quad \wt{g}_0(z) = z.
\end{align}
We note that~\eqref{eqn:rev_loewner} has a unique solution for every fixed $z$ in $\h$ defined for all times $t$.

\subsection{Imaginary geometry}
\label{subsec:ig}

Fix $\kappa \in (0,4)$, $\kappa'=16/\kappa > 4$, and let
\begin{equation}
\label{eqn:chi_value}
\lambda = \frac{\pi}{\sqrt{\kappa}},\quad  \lambda' = \frac{\pi}{\sqrt{\kappa'}}, \quad\text{and}\quad \chi = \frac{2}{\sqrt{\kappa}} - \frac{\sqrt{\kappa}}{2}.
\end{equation}
It was shown in \cite{she2016zipper,ms2016imag1} (see also \cite{dub2009partition}) that the $\SLE_{\kappa}(\ul{\rho})$ processes with $\kappa \in (0,4)$ and $\rho > -2$ can be considered as the flow lines of the vector field $e^{ih / \chi}$,  where $h$ is a GFF with fixed boundary data.

Let us first explain how this works in the case of the $\SLE_\kappa$ processes.  Suppose that~$h$ is a GFF on~$\h$ with boundary conditions given by $-\lambda$ on $\R_-$ and $\lambda$ on $\R_+$.  Then it is shown in \cite{ms2016imag1} that there exists a unique coupling of $h$ with an $\SLE_\kappa$ process $\eta$ in $\h$ from $0$ to $\infty$ so that the following is true.  Let $(g_t)$ be the Loewner flow for $\eta$, $W$ its driving function, and let $f_t = g_t - W_t$ be its centered Loewner flow.  Then for every stopping time $\tau$ for $\eta$ we have that
\[ h \circ f_\tau^{-1} - \chi \arg (f_\tau^{-1})' \stackrel{d}{=} h.\]
In this coupling, we moreover have that $\eta$ is determined by $h$.

More generally, suppose that $\eta$ is an $\SLE_\kappa(\ul{\rho})$ process where $\ul{\rho} = (\ul{\rho}^L; \ul{\rho}^R)$, $\ul{\rho}^L = (\rho_1^L,\ldots,\rho_\ell^L)$, $\ul{\rho}^R = (\rho_1^R,\ldots,\rho_k^R)$ with force points located at $\ul{x} = (\ul{x}^L;\ul{x}^R)$ with $-\infty = x_{\ell+1}^L < x_\ell^L < \cdots < x_1^L \leq x_0^L = 0_-$ and $0_+ = x_0^R \leq x_1^R < \cdots < x_k^R < x_{k+1}^R = +\infty$.  Suppose that $h$ is a GFF on $\h$ with boundary conditions given by
\begin{align*}
 -\lambda\left(1+ \sum_{i=1}^j \rho_i^L \right) \quad&\text{in}\quad (x_{i+1}^L,x_i^L] \quad\text{for}\quad  0 \leq i \leq \ell \quad \quad\text{and}\\
 \lambda\left( 1+ \sum_{i=1}^j \rho_i^R \right) \quad&\text{in}\quad (x_i^R,x_{i+1}^R] \quad\text{for}\quad 0 \leq i \leq k.
\end{align*}
Let $(f_t)$ be the centered Loewner flow for $\eta$.  Then there exists a unique coupling of $\eta$ with $h$ so that for every stopping time $\tau$ for $\eta$ we have that $h \circ f_\tau^{-1} - \chi \arg (f_\tau^{-1})'$ is a GFF on $\h$ with boundary conditions given by
\begin{align*}
 -\lambda\left(1+ \sum_{i=1}^j \rho_i^L \right) \quad&\text{in}\quad (f_\tau(x_{i+1}^L),f_\tau(x_i^L)] \quad\text{for}\quad  0 \leq i \leq \ell \quad \quad\text{and}\\
 \lambda\left( 1+ \sum_{i=1}^j \rho_i^R \right) \quad&\text{in}\quad (f_\tau(x_i^R),f_\tau(x_{i+1}^R)] \quad\text{for}\quad 0 \leq i \leq k
\end{align*}
where we take the convention that $f_\tau(x_0^L) = 0_-$ and $f_\tau(x_0^R) = 0_+$.  In this coupling, we moreover have that $\eta$ is determined by $h$.

In both of the above situations, we refer to $\eta$ as the flow line of $h$ from $0$ to $\infty$.  We can more generally define the flow line of $h$ from $x$ to $\infty$ by translating $h$.  Also, if $\theta \in \R$ then we define the flow line of $h$ with angle $\theta$ to be the flow line of $h + \theta \chi$.

The manner in which the flow lines interact was determined in \cite{ms2016imag1}.  In particular, suppose that $\eta_{x_i}^{\theta_i}$ for $i=1,2$ are the the flow lines of $h$ starting from $x_1 < x_2$ with angles $\theta_1,\theta_2$.  If $\theta_1 > \theta_2$, then $\eta_{x_1}^{\theta_1}$ stays to the left of $\eta_{x_2}^{\theta_2}$.  If $\theta_1 = \theta_2$, then $\eta_{x_1}^{\theta_1}$ merges with $\eta_{x_2}^{\theta_2}$ and does not subsequently separate.  Finally, if $\theta_1 \in (\theta_2-\pi,\theta_2)$, then $\eta_{x_1}^{\theta_1}$ crosses $\eta_{x_2}^{\theta_2}$ upon intersecting and does not subsequently cross back.

The $\SLE_{\kappa'}$ curves are also coupled with the GFF in a similar manner but the interpretation of the coupling is different.  Suppose that $h$ is a GFF on $\h$ with boundary conditions given by $\lambda'$ on $\R_-$ and $-\lambda'$ on $\R_+$.  Then it is shown in \cite{ms2016imag1} that there exists a unique coupling of $h$ with an $\SLE_{\kappa'}$ process $\eta'$ in $\h$ from $0$ to $\infty$ so that the following is true.  Let $(g_t)$ be the Loewner flow for $\eta'$, $W$ its driving function, and let $f_t = g_t - W_t$ be its centered Loewner flow.  Then for every stopping time $\tau$ for $\eta'$ we have that
\[ h \circ f_\tau^{-1} - \chi \arg (f_\tau^{-1})' \stackrel{d}{=} h.\]
In this coupling, we moreover have that $\eta'$ is determined by $h$.

More generally, suppose that $\eta$ is an $\SLE_\kappa(\ul{\rho})$ process where $\ul{\rho} = (\ul{\rho}^L; \ul{\rho}^R)$, $\ul{\rho}^L = (\rho_1^L,\ldots,\rho_\ell^L)$, $\ul{\rho}^R = (\rho_1^R,\ldots,\rho_k^R)$ with force points located at $\ul{x} = (\ul{x}^L;\ul{x}^R)$ with $-\infty = x_{\ell+1}^L < x_\ell^L < \cdots < x_1^L \leq x_0^L = 0_-$ and $0_+ = x_0^R \leq x_1^R < \cdots < x_k^R < x_{k+1}^R = +\infty$.  Suppose that $h$ is a GFF on $\h$ with boundary conditions given by
\begin{align*}
 \lambda' \left(1+ \sum_{i=1}^j \rho_i^L \right) \quad&\text{in}\quad (x_{i+1}^L,x_i^L] \quad\text{for}\quad  0 \leq i \leq \ell \quad \quad\text{and}\\
 -\lambda' \left( 1+ \sum_{i=1}^j \rho_i^R \right) \quad&\text{in}\quad (x_i^R,x_{i+1}^R] \quad\text{for}\quad 0 \leq i \leq k.
\end{align*}
Let $(f_t)$ be the centered Loewner flow for $\eta'$.  Then there exists a unique coupling of $\eta'$ with $h$ so that for every stopping time $\tau$ for $\eta'$ we have that $h \circ f_\tau^{-1} - \chi \arg (f_\tau^{-1})'$ is a GFF on $\h$ with boundary conditions given by
\begin{align*}
 \lambda'\left(1+ \sum_{i=1}^j \rho_i^L \right) \quad&\text{in}\quad (f_\tau(x_{i+1}^L),f_\tau(x_i^L)] \quad\text{for}\quad  0 \leq i \leq \ell \quad \quad\text{and}\\
 -\lambda'\left( 1+ \sum_{i=1}^j \rho_i^R \right) \quad&\text{in}\quad (f_\tau(x_i^R),f_\tau(x_{i+1}^R)] \quad\text{for}\quad 0 \leq i \leq k
\end{align*}
where we take the convention that $f_\tau(x_0^L) = 0_-$ and $f_\tau(x_0^R) = 0_+$.  In this coupling, we moreover have that $\eta'$ is determined by $h$.

In both of the above situations, we refer to $\eta'$ as the counterflow line of $h$ from $0$ to $\infty$.

In \cite{ms2016imag1} the manner in which the flow and counterflow lines of the GFF interact is also described.  In particular, suppose that $h$ is a GFF on $\h$ and $\eta'$ is the counterflow line of $h$ from $\infty$ to $0$.  Then the left (resp.\ right) boundary of $\eta'$ is equal to the flow line of $h$ from $0$ to $\infty$ with angle $\pi/2$ (resp.\ $-\pi/2$).

It is also possible to consider flow lines starting from interior points \cite{ms2017imag4} but we will not need to consider this in the present work.

\subsection{Light cones}
\label{subsec:light_cones}

Now, we review some basic properties of \emph{light cones} and their relation with the $\SLE_{\kappa}(\rho)$ processes.

We will first review the definition of the light cone.  We will focus on the case that it is defined on~$\h$; the definition on another simply connected domain is given by conformal mapping.  Suppose that $h$ is a GFF on $\h$ and $x \in \partial \h$ with piecewise constant boundary conditions that change only finitely many times.  Fix angles $\theta_1 \leq \theta_2 \leq \theta_1 + \pi$.  The $\SLE_{\kappa}$ \emph{light cone} $\lightcone_x(\theta_1,\theta_2)$ of $h$ starting from $x$ with angle range $[\theta_1,\theta_2]$ is given by the closure of the set of points accessible by the flow lines of $h$ starting from $x$ with angles which are either rational and contained in $[\theta_1,\theta_2]$ or equal to $\theta_1$ or $\theta_2$ and which change angles a finite number of times and only at positive rational times.  More generally,  if $A$ is a segment of $\partial D$,  we let $\lightcone_A(\theta_1,\theta_2)$ be the set of points accessible by flow lines of $h$ starting from a countable dense subset of $A$ with angles which are either rational and contained in $[\theta_1,\theta_2]$ or equal to $\theta_1$ or $\theta_2$ which change angles only a finite number of times and only at positive rational times.  It was shown in \cite[Theorem 1.2]{ms2019lightcone} that the range of an $\SLE_{\kappa}(\rho)$ process for $\kappa \in (0,4)$,  $\rho \in [\kappa/2-4,-2) \cap (-2 - \kappa/2,-2)$ is equal to $\lightcone_{\R_-}(0,\theta)$ when the boundary of $h$ and~$\theta$ are chosen appropriately.  More precisely, the following was shown in \cite{ms2019lightcone}:
\begin{theorem}
\label{thm:ligthcone_and_sle}
Fix $\kappa \in (0,4)$,  $\rho \in [\kappa/2-4,-2)$ and $\rho > -2 - \kappa/2$,  and suppose that $h$ is a \text{GFF} on $\h$ whose boundary data is given by $-\lambda$ on $\R_-$ and $\lambda (1+\rho)$ on $\R_+$.  Let $\eta$ be an $\SLE_{\kappa}(\rho)$ process on $\h$ from $0$ to $\infty$ where its force point is located at $0_+$.  For each $t \geq 0$,  let $K_t$ denote the closure of the complement of the unbounded connected component of $\h \setminus \eta([0,t])$,  let $g_t : \h \setminus K_t \to \h$ be the unique conformal transformation with $g_t(z)-z \to 0$ as $z \to \infty$,  and let $(W,V)$ be the Loewner driving pair for $\eta$.  There exists a unique coupling of $h$ and $\eta$ such that the following is true.  For each $\eta$-stopping time $\tau$,  the conditional law of 
\begin{align*}
h \circ g_{\tau}^{-1} - \chi \arg(g_{\tau}^{-1})'
\end{align*}
given $\eta|_{[0,\tau]}$ is that of a \text{GFF} on $\h$ with boundary conditions given by 
\begin{align*}
h|_{(-\infty,W_{\tau}]} \equiv -\lambda,\quad h|_{(	W_{\tau},V_{\tau}]} \equiv \lambda,\quad \text{and} \quad h|_{(V_{\tau},\infty)}\equiv \lambda (1+\rho),
\end{align*}
where $\lambda = \pi / \sqrt{\kappa}$.

Moreover,  in the coupling $(h,\eta)$,  $\eta$ is a.s.\ determined by $h$.  Finally,  let 
\begin{equation}\label{eqn:theta_and_rho}
\theta = \theta_{\rho} = \pi \!\left ( \frac{\rho + 2}{\kappa / 2 - 2} \right ).
\end{equation}
Then the range of $\eta$ is a.s.\ equal to $\lightcone_{\R_-}(0,\theta)$.
\end{theorem}
Theorem~\ref{thm:ligthcone_and_sle} allows to interpret an $\SLE_{\kappa}(\rho)$ process with the above range of values of $\rho$ as an ordered light cone of flow lines of $e^{ih / \chi}$.

\section{SLE/GFF couplings}
\label{sec:slegff_couplings}

The purpose of this section is to prove two results related to couplings of $\SLE$ with the GFF.  In Section~\ref{subsec:reverse_coupling} we will focus on the so-called reverse coupling and will establish a version of it which holds for reverse $\SLE_\kappa(\wt{\rho})$ processes with $\wt{\rho} < 2$.  In Section~\ref{subsec:law_of_sle_excursion} we use the forward coupling of $\SLE$ with the GFF in order to describe the law of an excursion of an $\SLE_\kappa(\rho)$ process in the light cone regime.

\subsection{Reverse coupling with $\wt{\rho} < 2$}
\label{subsec:reverse_coupling}

In this subsection we will state and prove a version of the reverse coupling of $\SLE_\kappa(\wt{\rho})$ with the GFF when $\wt{\rho} < 2$.  In the case of ordinary reverse $\SLE_\kappa$, this was first proved in \cite{she2016zipper} and it was extended to the case of the reverse $\SLE_\kappa(\wt{\rho})$ processes with $\wt{\rho} > 2$ in \cite{dms2014mating}.  In both \cite{she2016zipper} and \cite{dms2014mating}, the process which drives the reverse $\SLE$ is a semimartingale which allows one to use the tools of ordinary stochastic calculus.  In contrast, when $\wt{\rho} < 2$ the driving process for a reverse $\SLE_\kappa(\wt{\rho})$ process is not a semimartingale and this introduces some extra complications in the version that we will prove.

Throughout, we fix $\kappa >0$, $\wt{\rho} \in (2-\kappa/2 ,2)$, and let
\[ \gamma=\min\left( \sqrt{\kappa},\frac{4}{\sqrt{\kappa}} \right) \quad\text{and}\quad Q=\frac{2}{\gamma}+\frac{\gamma}{2}.\]
(Recall from Section~\ref{subsubsec:reverse_sle_kappa_rho} that reverse $\SLE_\kappa(\wt{\rho})$ is only defined for $\wt{\rho} > 2-\kappa/2$.)  Let $G_\h^N$ be the Neumann Green's function on $\h$ (recall~\eqref{eqn:neumann_greens}).  Let $(\wt{f}_t)$ be the centered Loewner flow associated with a reverse $\SLE_\kappa(\wt{\rho})$ process and let
\begin{align*}
\wt{G}_{t}(y,z) = G_\h^N(\wt{f}_{t}(y),\wt{f}_{t}(z)),\quad
\wh{\Fh}_{0}(z) = \frac{2}{\sqrt{\kappa}}\log|z|,\quad \text{and}\quad \wh{\Fh}_{t}(z) = \wh{\Fh}_{0}(\wt{f}_{t}(z)) + Q\log|\wt{f}_{t}'(z)|.
\end{align*}

The following is the main theorem of this subsection and serves to extend \cite[Theorem~1.2]{she2016zipper} and \cite[Theorem~5.1]{dms2014mating}.

\begin{theorem}
\label{thm:reverse_coupling}
Let $\kappa$, $\wt{\rho}$ be as above, let $\wt{W}$ be the driving function associated with a reverse $\SLE_{\kappa}(\wt{\rho})$ with a single force point located at $0_+$ of weight $\wt{\rho}$, and let $(\wt{f}_{t})$ be the centered reverse Loewner evolution driven by $\wt{W}$. For each $t \geq 0$, let 
\begin{align*}
\wt{\Fh}_{t}(z) = \wh{\Fh}_{t}(z) + \frac{\wt{\rho}}{2\sqrt{\kappa}}\wt{G}_{t}(0_+,z)
\end{align*}
and let $h$ be an instance of the free boundary $\text{GFF}$ on $\h$ independent of $\wt{W}$. Suppose that $\wt{\tau}$ is an a.s.\ finite stopping time for the filtration generated by the Brownian motion which drives $\wt{W}$. Then
\begin{align*}
\wt{\Fh}_{0} + h \stackrel{d}{=} \wt{\Fh}_{\wt{\tau}} + h \circ \wt{f}_{\wt{\tau}}
\end{align*}
where the left and right sides are viewed as distributions modulo additive constant.
\end{theorem}
Recall from Section~\ref{subsubsec:reverse_sle_kappa_rho} that $\wt{f}_t(0_+)/\sqrt{\kappa}$ is a $\BES^{\wt{\delta}}$ process with
\[ \wt{\delta} = 1+ \frac{2(\wt{\rho}-2)}{\kappa}.\]
The proofs of \cite[Theorem~1.2]{she2016zipper} and \cite[Theorem~5.1]{dms2014mating} make use of stochastic calculus and require that $\int_0^t \wt{f}_s(0_+)^{-1}ds$ is finite for all $t\geq 0$ a.s.  However, when $\wt{\rho} \in (2-\kappa/2, 2)$ we have that $\wt{\delta} \in (0,1)$ so that this integral is infinite a.s.\ and this fact is the key difference in the proof of Theorem~\ref{thm:reverse_coupling} in comparison to that given in \cite{she2016zipper,dms2014mating}.  In order to overcome this difficulty,  we are going to use Proposition~\ref{prop:epsilon_bessel} to approximate $\wt{f}_t(0_+)/\sqrt{\kappa}$ using an $\epsilon$-$\BES^{\wt{\delta}}$ process and then take a limit as $\epsilon \to 0$.  Moreover,  we will also make use of the following generalization of It\^o's formula for semimartingales with discontinuities.

\begin{theorem}
\label{thm:general_Ito}
Suppose that $Z_{t} = X_{t} + iY_{t}$ is a semimartingale (not necessarily continuous) with $X_{t} = \re(Z_{t})$ and $Y_{t} = \im(Z_{t})$ and $Z_{t} \in U$ for all $t$ for $U \subseteq \C$ open. Then if $f$ is an analytic function on $U$, it holds that
\begin{align*}
    f(Z_{t}) - f(Z_{0}) &= \int_{0}^{t}f'(Z_{s})dZ_{s} + \frac{1}{2}\int_{0}^{t}f''(Z_{s^{-}})d\langle Z^{c}\rangle_{s}\\
    &+\sum_{0 \leq s \leq t}\left( f(Z_{s}) - f(Z_{s^{-}}) - f'(Z_{s^{-}})\Delta Z_{s}\right)
\end{align*}
\end{theorem}

Let us now perform some initial calculations which will be used in the proof of Theorem~\ref{thm:reverse_coupling}.  Let~$B$ be the Brownian motion which drives $\wt{X}$ and let $\wt{X}^{\epsilon}$ be the solution of~\eqref{eq:approximate_bessel_sde}. Let also $\wt{J}_t^\epsilon$ be as in~\eqref{eq:approximate_bessel_sde} and $\wt{J}_t^{\epsilon^2}$ be as in part~\eqref{it:jtsquared} of Proposition~\ref{prop:epsilon_bessel} (so that $\wt{J}_t^{\epsilon^2}$ is $\epsilon^2$ times the number of $\epsilon$-jump discontinuities of $X_t^\epsilon$ up to and including time $t$). We consider the processes 
\begin{align*}
    \wt{Y}^\epsilon = \frac{2}{\wt{\delta} - 1}(\wt{X}^{\epsilon} - B) \quad\text{and}\quad
    \wt{W}^{\epsilon} = -\frac{2}{\sqrt{\kappa}}\wt{Y}^{\epsilon} - \sqrt{\kappa}\wt{X}^{\epsilon}
\end{align*}
By Proposition~\ref{prop:epsilon_bessel}, we can couple $\wt{X}^{\epsilon}$ and $\wt{X}$ such that $\wt{X}^{\epsilon}$ converges to $\wt{X}$ as $\epsilon \to 0$ uniformly on compact intervals a.s. Fix $z \in \h$ and let
\begin{align*}
    \wt{Z}_{t}^{\epsilon} = \wt{g}_{t}(z) + \frac{2}{\sqrt{\kappa}}\wt{Y}_{t}^{\epsilon} + \sqrt{\kappa}\wt{X}_{t}^{\epsilon}.
\end{align*}
Note that $\wt{Z}^{\epsilon}$ converges to $\wt{f}_{t}(z)$ as $\epsilon \to 0$ uniformly on compact intervals a.s.  Since
\[ \langle \wt{Z}^{\epsilon}\rangle_{t} = \kappa t + \left(\sqrt{\kappa} + \frac{4}{\sqrt{\kappa}(\wt{\delta} - 1)}\right)^2 \wt{J}_{t}^{\epsilon^{2}},\]
by applying Theorem~\ref{thm:general_Ito} we obtain that 
\begin{align*}
\frac{2}{\sqrt{\kappa}}\left( \log \wt{Z}_{t}^{\epsilon} - \log\wt{Z}_{0}^{\epsilon}\right) = \wt{\alpha}_{t}^{\epsilon} + \frac{2}{\sqrt{\kappa}}\int_{0}^{t}\frac{1}{\wt{Z}_{s^{-}}^{\epsilon}}d \wt{Z}_{s}^{\epsilon} - \sqrt{\kappa}\int_{0}^{t}\frac{1}{(\wt{Z}_{s^{-}}^{\epsilon})^{2}}ds
\end{align*}
 where 
\begin{align*}
\wt{\alpha}_{t}^{\epsilon} = \frac{2}{\sqrt{\kappa}}\sum_{0 \leq s \leq t}\left( \log \wt{Z}_{s}^{\epsilon} - \log \wt{Z}_{s^-}^{\epsilon} - \frac{\Delta \wt{Z}_{s}^{\epsilon}}{\wt{Z}_{s^{-}}^{\epsilon}}\right).
\end{align*}
Similarly we have that
\[ \langle \wt{A}^{\epsilon}\rangle_{t} = \frac{16 \wt{J}_{t}^{\epsilon^{2}}}{\kappa (\wt{\delta} -1)^{2}} \quad\text{where}\quad \wt{A}_{t}^{\epsilon} = \wt{g}_t(z) + \frac{2}{\sqrt{\kappa}} \wt{Y}_{t}^{\epsilon}\]
and hence by applying Theorem~\ref{thm:general_Ito} again we obtain that 
\begin{align*}
\frac{\wt{\rho}}{\sqrt{\kappa}}\left( \log \wt{A}_{t}^{\epsilon} - \log\wt{A}_{0}^{\epsilon}\right) = \wt{\beta}_{t}^{\epsilon} + \frac{\wt{\rho}}{\sqrt{\kappa}}\int_{0}^{t}\frac{1}{\wt{A}_{s^-}^{\epsilon}}d\wt{A}_{s}^{\epsilon}
\end{align*}
where 
\begin{align*}
\wt{\beta}_{t}^{\epsilon} = \frac{\wt{\rho}}{\sqrt{\kappa}}\sum_{0 \leq s \leq t}\left(\log \wt{A}_{s}^{\epsilon} - \log \wt{A}_{s^-}^{\epsilon} - \frac{\Delta \wt{A}_{s}^{\epsilon}}{\wt{A}_{s^-}^{\epsilon}}\right).
\end{align*}
Let $(t_i)$ be as in~\eqref{eq:approximate_bessel_sde}.  Since $\wt{X}_{t_{i}^-}^{\epsilon} = 0$ we have that $\wt{A}_{t_{i}^{-}}^{\epsilon} = \wt{Z}_{t_{i}^-}^{\epsilon}$ and so 
\begin{align}
        &\frac{2}{\sqrt{\kappa}}\left( \log \wt{Z}_{t}^{\epsilon} - \log \wt{Z}_{0}^{\epsilon} \right) - \frac{\wt{\rho}}{\sqrt{\kappa}} \left( \log\wt{A}_{t}^{\epsilon} - \log \wt{A}_{0}^{\epsilon} \right) \nonumber\\
        =&(\wt{\alpha}_{t}^{\epsilon} - \wt{\beta}_{t}^{\epsilon})  +\int_{0}^{t}\frac{2}{\wt{Z}_{s^-}^{\epsilon}}dB_{s} - \frac{2}{\sqrt{\kappa}}\int_{0}^{t}\frac{2}{\wt{Z}_{s^-}^{\epsilon}\wt{f}_{s}(z)}ds\nonumber\\
        &-\sqrt{\kappa}\int_{0}^{t}\frac{1}{(\wt{Z}_{s^-}^{\epsilon})^{2}}ds + \frac{2\wt{\rho}}{\sqrt{\kappa}}\int_{0}^{t}\frac{\wt{Z}_{s^-}^{\epsilon} - \wt{f}_{s}(z)}{\wt{A}_{s^-}^{\epsilon}\wt{f}_{s}(z) \wt{Z}_{s^-}^{\epsilon}}ds.\label{eqn:basic_equation}
\end{align}

Fix $T > 0$. It is easy to see that
\[ \log \wt{Z}_{t_i}^{\epsilon} - \log \wt{Z}_{t_{i}^-}^{\epsilon} - \frac{\Delta \wt{Z}_{t_i}^{\epsilon}}{\wt{Z}_{t_{i}^-}^{\epsilon}} = O((\Delta \wt{Z}_{t_i}^{\epsilon})^2) \quad\text{for all}\quad  t_i \leq T,\]
where the implicit constant is random but independent of $\epsilon$. Since
\[ \Delta \wt{Z}_{t_i}^{\epsilon} = \left( \sqrt{\kappa} + \frac{4}{\sqrt{\kappa}(\wt{\delta} - 1)} \right) \epsilon\]
we obtain that $\wt{\alpha}_{T}^{\epsilon} = O(\wt{J}_{T}^{\epsilon^2})$, where the implicit constant is random and independent of $\epsilon$. Therefore by applying Proposition~\ref{prop:epsilon_bessel} we obtain that $\wt{\alpha}_{t}^{\epsilon}$ converges to zero as $\epsilon \to 0$ uniformly on compact intervals a.s. A similar analysis holds for $\wt{\beta}_{t}^{\epsilon}$.  Moreover,  it is easy to see that
\[ \int_{0}^{t}\frac{1}{\wt{Z}_{s^-}^{\epsilon}}dB_{s} \to \int_{0}^{t}\frac{1}{\wt{f}_{s}(z)}dB_{s} \quad\text{in}\quad L^{2} \quad\text{as}\quad \epsilon \to 0.\]
In particular, the convergence holds in probability. Therefore the above observations imply that the right hand side of~\eqref{eqn:basic_equation} converges in probability as $\epsilon \to 0$ to 
\begin{align}
\label{eq:RHS_convergence}
    \int_{0}^{t} \frac{2}{\wt{f}_{s}(z)}dB_{s} - Q\int_{0}^{t} \frac{2}{(\wt{f}_{s}(z))^2}ds
\end{align}
while the left hand side of~\eqref{eqn:basic_equation} converges a.s.\ as $\epsilon \to 0$ to 
\begin{align}
\label{eq:LHS_convergence}
    \frac{2}{\sqrt{\kappa}}\left(\log \wt{f}_{t}(z)  - \log z\right) - \frac{\wt{\rho}}{\sqrt{\kappa}}\left(\log(\wt{f}_{t}(z) - \wt{f}_{t}(0^{+})) - \log z \right)
\end{align}
Hence the quantities in~\eqref{eq:RHS_convergence} and~\eqref{eq:LHS_convergence} are equal for all $t$  a.s. 

Next, we follow the same strategy as in \cite[Section~5.1]{dms2014mating} in order to prove Theorem~\ref{thm:reverse_coupling}.  We will restate the lemmas used in order to make clear how the equality of the expressions in~\eqref{eq:RHS_convergence} and~\eqref{eq:LHS_convergence} is used.

\begin{lemma}[$\text{\cite[Lemma~5.2]{dms2014mating}}$]
\label{lem:helping_lemma}
For each $z \in \h$, we have that 
\begin{align*}
d\wt{\Fh}_{t}(z) = \re\frac{2}{\wt{f}_{t}(z)}dB_{t}
\end{align*}
For each $y,z \in \h$ distinct, we have that 
\begin{align*}
d\wt{G}_{t}(y,z) = -\re\left(\frac{2}{\wt{f}_{t}(y)}\right)\re\left(\frac{2}{\wt{f}_{t}(z)}\right)dt
\end{align*}
\end{lemma}
\begin{proof}
By~\eqref{eqn:rev_loewner} we obtain that 
\begin{align*}
d\wt{f}_{t}'(z) = \frac{2\wt{f}_{t}'(z)}{\wt{f}_{t}^{2}(z)}dt
\end{align*}
and thus 
\begin{align*}
\log \wt{f}_{t}'(z) = \int_{0}^{t}\frac{2}{(\wt{f}_{t}(z))^{2}}ds.
\end{align*}
The first claim then follows by taking real parts in~\eqref{eq:RHS_convergence} and~\eqref{eq:LHS_convergence} and then adding $Q\log|\wt{f}_{t}'(z)|$ to both of them.
\newline
As for the second claim, we apply the same argument as in \cite[Lemma~5.2]{dms2014mating}.
\end{proof}
\begin{lemma}[$\text{\cite[Lemma~5.3]{dms2014mating}}$]\label{lem:square_integrable_martingale}
For each $ t \geq 0$, $\wt{\Fh}_{t} + h \circ\wt{f}_{t}$ a.s.\ takes values in the space of modulo additive constant distributions. Suppose that $\phi \in C_{0}^{\infty}(\h)$ with $\int_{\h}\phi(z)dz = 0$. Then both $(\wt{\Fh}_{t} + h \circ \wt{f}_{t},\phi)$ and $(\wt{\Fh}_{t},\phi)$ are a.s.\ continuous in $t$ and the latter is a square-integrable martingale.
\end{lemma}
\begin{proof}
It follows from the proof of \cite[Lemma~5.3]{dms2014mating}.
\end{proof}
For each $\phi \in C_{0}^{\infty}(\h)$ with $\int_{\h}\phi (z)dz = 0$ and $t \geq 0$ we let 
\begin{align*}
\wt{E}_{t}(\phi) = \int_{\h}\int_{\h} \phi(y)\wt{G}_{t}(y,z)\phi(z)dydz
\end{align*}
be the conditional variance of $(h \circ\wt{f}_{t},\phi)$ given $\wt{f}_{t}$.

\begin{lemma}[$\text{\cite[Lemma~5.4]{dms2014mating}}$]
\label{lem:variation_formula}
For each $\phi \in C_{0}^{\infty}(\h)$ with $\int_{\h} \phi (z)dz = 0$ we have that 
\begin{align*}
d\langle (\wt{\Fh}_{\cdot},\phi)\rangle_t  = -d\wt{E}_{t}(\phi)
\end{align*}
\end{lemma}
\begin{proof}
It follows from the proof of \cite[Lemma~5.4]{dms2014mating}.
\end{proof}
We are now ready to prove Theorem~\ref{thm:reverse_coupling}.
\begin{proof}
It follows by combining Lemmas~\ref{lem:helping_lemma},   \ref{lem:square_integrable_martingale}, and~\ref{lem:variation_formula} as in the proof of \cite[Theorem~5.1]{dms2014mating}.
\end{proof}

\subsection{Law of an $\SLE_{\kappa}(\rho)$ excursion}\label{subsec:law_of_sle_excursion}

In this section, we will give the conditional law of the excursion of an $\SLE_{\kappa}(\rho)$ process in $\h$ with $\kappa \in (0,4)$ and $\rho \in (\kappa/2-4,-2) \cap (-2 -\kappa/2,-2)$ with the force point at $0_+$ which separates a given boundary point $x > 0$ from $\infty$ given the realization of the path before the start of the excursion.

\begin{theorem}
\label{thm:law_of_sle_excursion}
Fix $\kappa \in (0,4)$, $\rho \in (\kappa/2 - 4, -2)$, $\rho > -\kappa/2 - 2$.  Let $\eta$ be an $\SLE_{\kappa}(\rho)$ process in $\h$ from $0$ to $\infty$ with the force point located at $0_+$.  Fix $x > 0$ and let $T_x$ be the first time that~$\eta$ separates~$x$ from $\infty$.  Let $S_x = \sup\{ t < T_x : W_t = V_t\}$ be the most recent time before $T_x$ that~$\eta$ hits its force point. Then the conditional law of $\eta|_{[S_x,T_{x}]}$ given $\eta|_{[0,S_x]}$ and $\eta(T_x)$ is that of an $\SLE_{\kappa}(\rho+2; \kappa-4-\rho)$ process from $\eta(S_x)$ to $\eta(T_x)$ in $\h_{S_x}$ with the weight $\rho+2$ force point located at $\infty$ and the weight $\kappa-4-\rho$ force point located at $(\eta(S_x))^+$.
\end{theorem}

In order to give the proof of Theorem~\ref{thm:law_of_sle_excursion}, we first need to collect the following lemma.

\begin{lemma}
\label{lem:sle_separates_points}
Fix $\kappa \in (0,4)$, $\rho \in (\kappa/2 - 4, -2)$, $\rho > -\kappa/2 - 2$.  Let $\eta$ be an $\SLE_{\kappa}(\rho)$ process in~$\h$ from~$0$ to~$\infty$ with the force point located at $0_+$.  Then for each fixed $x>0$, we have that $\eta$ does not hit $x$ and separates $x$ from $\infty$ a.s.
\end{lemma}
\begin{proof} 
\noindent{\it Overview.} The proof will be carried out in three steps.  First of all, by the scale invariance of the law of $\eta$,  it suffices to prove the claim of the lemma when $x=1$.  In Step $1$,  we show that $\eta$ hits $(1,\infty)$ with positive probability.  In Step 2, we use the conformal Markov property of $\SLE_\kappa(\rho)$ in order to show that $\eta$ in fact hits $(1,\infty)$ a.s.  Finally, in Step 3 we prove that $\eta$ a.s.\ does not hit $1$.

\noindent{\it Step 1.  $\eta$ hits $(1,\infty)$ with positive probability.} Let $(V,W)$  be the driving pair encoding $\eta$ as in Section~\ref{subsec:slekapparho} and let $\tau$ be the first time that $\eta$ hits $(1,\infty)$.  Let also $(g_t)$ be the family of conformal maps solving the Loewner equation used to define $\eta$.  Since $\eta$ is continuous,  there exists $\epsilon_1>0$ so that $\p[\eta([0,\epsilon_1]) \subseteq B(0,\frac{1}{2})] > 0$.  Moreover,  since $\kappa^{-1/2} (V-W)$ is a $\BES^{\delta}$ process starting from $0$ with $\delta$ as in~\eqref{eqn:delta_and_rho},  we obtain that $V_{\epsilon_1} \neq W_{\epsilon_1}$ a.s.\  and so there exist $p,\epsilon_2 \in (0,1)$ such that $\p[A] \geq p$,  where $A$ is the event that $\eta([0,\epsilon_1]) \subseteq B(0,\frac{1}{2})$ and $\epsilon_2 \leq V_{\epsilon_1} - W_{\epsilon_1} \leq \epsilon_2^{-1}$.  The Markov property of $\eta$ implies that if $\mathcal{F}_t = \sigma(\eta|_{[0,t]})$,  then conditioned on $\mathcal{F}_{\epsilon_1}$ and restricted to $A$,  the curve $\wt{\eta}(t) = g_{\epsilon_1}(\eta(\epsilon_1 + t)) - W_{\epsilon_1}$ has the law of an $\SLE_{\kappa}(\rho)$ process in $\h$ from $0$ to $\infty$ with the force point located at $V_{\epsilon_1} - W_{\epsilon_1}$.  Let $\gamma$ be the concatenation of the segments $[0,i]$,  $[i,i+\epsilon_2^{-1}]$ and $[i+\epsilon_2^{-1},\epsilon_2^{-1}]$.  By \cite[Lemma~2.5]{mw2017slepaths},  we obtain that there exists a constant $q \in (0,1)$ such that conditionally on $\mathcal{F}_{\epsilon_1}$ and on the event $A$,  we have that with probability at least $q$,  the curve $\wt{\eta}$ hits $(\epsilon_2^{-1}-\epsilon_2,\epsilon_2^{-1}+\epsilon_2)$ before exiting the $\epsilon_2 / 2$-neighborhood of $\gamma$.  The claim then follows possibly by taking $\epsilon_1$ to be smaller since standard estimates for conformal maps imply that $|g_{\epsilon_1}(z)-z|\leq 6 \epsilon_1$ for all $z \in \h_{\epsilon_1}$,  where~$\h_{\epsilon_1}$ is the unbounded connected component of~$\h \setminus \eta([0,\epsilon_1])$.

\noindent{\it Step 2.  $\eta$ hits $(1,\infty)$ a.s.} By Step 1, we have that $\eta$ hits $(1,\infty)$ with positive probability.  Let $p$ be the probability that $\eta$ hits $(1,\infty)$ and assume that $p \in (0,1)$.  Fix $\epsilon \in (0,p)$.  Then there exists $T > 0$ so that the probability that $\eta|_{[0,T]}$ hits $(1,\infty)$ is at least $p-\epsilon$.  Let $\tau = \inf\{t \geq T : \eta(t) \in \R_+\}$.  On the event that $\eta|_{[0,\tau]}$ has not hit $(1,\infty)$, we let $\wt{\eta} = (g_\tau(\eta) - W_\tau)/(g_\tau(1) - W_\tau)$.  Then we have that~$\wt{\eta}$ hits $(1,\infty)$ with probability $p$.  Overall, this implies that the probability that $\eta$ hits $(1,\infty)$ is at least $p - \epsilon + (1-p+\epsilon)p$.  By taking $\epsilon > 0$ sufficiently small we see that $p - \epsilon + (1-p+\epsilon)p > p$, which is a contradiction.  Therefore $p=1$.

\noindent{\it Step 3.  $\eta$ does not hit $1$ a.s.}  Let $p$ be the probability that $\eta$ hits $1$.  By the argument given in Step~1, we have that $p < 1$.  Assume that $p > 0$ and fix $q \in (0,1)$ so that $p < q < 1$.  Fix $\delta > 0$ sufficiently small so that the probability that $\eta$ hits $B(1,\delta)$ is at most $q$.  Let $\tau = \inf\{t \geq 0 : \eta(t) \in B(1,\delta)\}$.  On the event that $\tau < \infty$, let $\phi_\tau = (g_\tau - W_\tau)/(V_\tau - W_\tau)$ so that $\wt{\eta} = \phi_\tau(\eta)$ is an $\SLE_\kappa(\rho)$ process in $\h$ from $0$ to $\infty$ with its force point located at $1$.  Let $\wt{\tau} = \inf\{t \geq 0 : \wt{\eta}(t) \in (1,\infty)\}$.  Let $p_1(\delta)$ be the probability that $\wt{\eta}(\wt{\tau}) \neq \phi_{\tau}(1)$ conditionally on the event that $\tau < \infty$.  Then,  \cite[Lemma~2.5]{mw2017slepaths} implies that there exists $p_2 > 0$ independent of $\delta$ so that $p_1(\delta) \geq p_2$.  Let $p_{1,L}(\delta)$ be the probability that $\wt{\eta}(\wt{\tau}) < \phi_\tau(1)$ and $p_{1,R}(\delta)$ be the probability that $\wt{\eta}(\wt{\tau}) > \phi_\tau(1)$,  conditionally on the event that $\tau<\infty$,  so that $p_1(\delta) = p_{1,L}(\delta) + p_{1,R}(\delta)$.  On the event that $\wt{\eta}(\wt{\tau}) < \phi_\tau(1)$, by the conformal Markov property of $\SLE_\kappa(\rho)$ the probability that it subsequently hits $\phi_\tau(1)$ is $p$.  Altogether, we have that
\[ p \leq q ( (1-p_1(\delta)) + p_{1,L}(\delta) p) \leq q( (1-p_1(\delta)) + p_1(\delta) p) \leq q(1-p_2(1-p)).\]
Since $ 1-p_2(1-p)< 1$, this is a contradiction unless $p = 0$ since by taking $\delta > 0$ small we can assume that $p < q < 1$ is arbitrarily close to $p$.
\end{proof}

\begin{figure}[ht!]
\begin{center}
\includegraphics[scale=0.85]{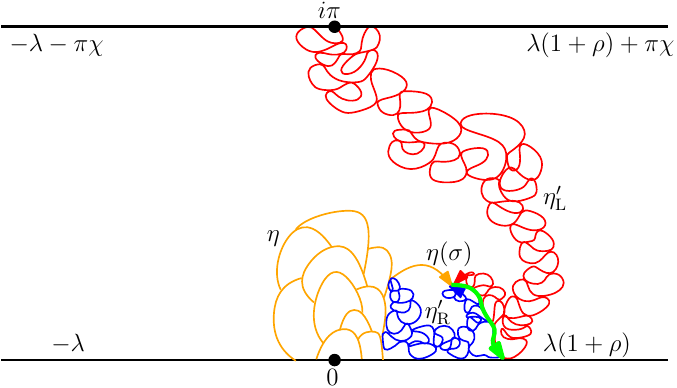}
\end{center}
\caption{\label{fig:sle_excursion_strip} Illustration of the setup to prove Theorem~\ref{thm:law_of_sle_excursion}.  Shown in orange is the $\SLE_\kappa(\rho)$ process $\eta$, which we view as the light cone path from $0$ to $i\pi$ coupled with the GFF on $\strip$ with the indicated boundary data.  The counterflow lines $\eta_{\mathrm{L}}'$ and $\eta_{\mathrm{R}}'$ are shown in red and blue, respectively, and their common outer boundary gives the completion of the excursion that $\eta$ is drawing at the time $\sigma$ (shown in green).}
\end{figure}

\begin{proof}[Proof of Theorem~\ref{thm:law_of_sle_excursion}]
\noindent{\it Step 1.  Setup and outline of the proof.} First we note that Lemma~\ref{lem:sle_separates_points} implies that $T_x < \infty$ a.s.  We are going to apply arguments which are similar to those presented in the proofs of \cite[Theorem~5.6]{dms2014mating} and \cite[Proposition~5.12]{ms2016imag2}.  We will in particular make use of the coupling of $\eta$ with the $\text{GFF}$ light cone introduced in \cite{ms2019lightcone} (see Section~\ref{subsec:light_cones} for a review of this).  First, we apply a conformal transformation $\h \to \strip$ taking $0$ to $0$ and $\infty$ to $i \pi$.  Then we can view $\eta$ as the light cone exploration path from $0$ to $i \pi$ associated with a $\text{GFF}$ $h$ on $\strip$ and $h$ has boundary values given by $-\lambda$ on $\R_-$, $\lambda (1+\rho)$ on $\R_+$, $-\lambda - \pi \chi$ on $\R_- + i \pi$, and $\lambda (1+\rho) + \pi \chi$ on $\R_+ + i\pi$.  (The reason for making this change of coordinates is that it will be easier to read off the law of different objects associated with the GFF on $\strip$ versus $\h$.)  Let $\eta_{\text{L}}'$ be the counterflow line of $h+ \pi \chi/2$ starting from $i \pi$ and targeted at $0$.  As usual, we let $\kappa'=16/\kappa$.  Then we note that $\eta_{\text{L}}'$ has the law of an $\SLE_{\kappa}(\rho_{\text{L}}' ; \rho_{\text{R}}')$ process where $\rho_{\text{L}}' = \kappa'(1+\rho/4)-4$ and $\rho_{\text{R}}' = \kappa'/2-2$ where the force points are located at $(i\pi)_-$ and $(i\pi)_+$.  Indeed, the reason for this is that $\rho_{{\mathrm L}}'$, $\rho_{{\mathrm R}}'$ are determined by the equations
\[ \lambda'(1+\rho_{{\mathrm L}}') = \lambda(1+\rho)+ \frac{3\pi\chi}{2} \quad\text{and}\quad  -\lambda'(1+\rho_{{\mathrm R}}') = -\lambda -\frac{\pi \chi}{2}.\]
Hence,  a.s.\  $\eta_{\text{L}}'$ does not intersect either $\R_-$ or $\R_- + i \pi$ since $\kappa'/2-2$ is the critical value below which such processes can hit the boundary.  The proof is then divided into three subsequent steps.  In Step~2,  we will describe the interaction between $\eta_{\text{L}}'$ and the pockets formed by $\eta$.  In Step~3,  we are going to prove that the joint law of $\eta|_{[S_x,T_x]}$ together with certain flow lines satisfies the same resampling property as stated in \cite[Proposition~5.12]{ms2016imag2}.  Finally,  in Step~4,  we complete the proof by observing that if $\eta|_{[S_x,T_x]}$ had the law of an $\SLE_{\kappa}(\rho+2 ; \kappa-4-\rho)$ process,  then it would satisfy the same resampling property.  \cite[Proposition~5.10]{ms2016imag2} implies that the conditional law of $\eta|_{[S_x,T_x]}$ is uniquely determined by this resampling property and so this completes the proof of the lemma.

\noindent{\it Step 2.  Interaction between $\eta_{\text{L}}'$ and the pockets of $\eta$.} We now describe the interaction between $\eta_{\text{L}}'$ and the pockets formed by $\eta$ as described in \cite{ms2019lightcone}.  Let $P$ be a pocket of formed by $\eta$ with $u$ (resp.\ $v$) its opening (resp.\ closing) point.  Let $\eta_1(P)$ be the $0$-angle flow line part of $\partial P$ and $\eta_2(P)$ be the $\theta_{\rho}$-angle flow line part of $\partial P$ with $\theta_{\rho}$ as in~\eqref{eqn:theta_and_rho}.  For $z \in \strip$,  we let $\eta_z'$ be the counterflow line of $h+\pi \chi/2$ starting from $i \pi$ and targeted at $z$.  Then $\eta_{\text{L}}'$ and $\eta_z'$ agree up until they separate $z$ from $0$ a.s., and the right side of $\eta_z'$ stopped upon hitting $z$ is equal to the flow line of $h$ starting from $z$ with angle $0$.  It in particular follows that $\eta_{\text{L}}'$ hits the pocket of $\eta$ containing $z$ provided the flow line of $h$ starting from $z$ exits $\strip$ in $\R_+$ or $\R_+ + i \pi$.  Assume that we are on this event.  Combining with the flow line interaction rules \cite[Theorem~1.7]{ms2017imag4} with \cite[Section~3]{ms2019lightcone},  we obtain that $\eta_{\text{L}}'$ interacts with $P$ in the following way.
\begin{enumerate}[(i)]
\item $\eta_{\text{L}}'$ enters the interior of $P$ at $u$ after filling the left side of $\eta_1(P)$ in reverse chronological order,  where we view $\eta_1(P)$ as an oriented path from $u$ to $v$.
\item Upon intersecting $P$,  $\eta_{\text{L}}'$ visits some points on the right side of $\eta_1(P)$ as it travels from $u$ to $v$.  It does not touch $\eta_2(P)$ up until hitting $v$.  
\item Upon hitting $v$,  it visits all the points of $\eta_2(P)$ in reverse chronological order and, while doing so,  $\eta_{\text{L}}'$ makes excursions both into and out of $P$.
\end{enumerate}
The above interaction rules imply that $\eta_{\text{L}}'$ contains the pockets $P$ of $\eta$ such that the the flow line corresponding to $\eta_1(P)$ exits $\strip$ in $\R_+$ or $\R_+ + i \pi$ a.s.  In particular, if $\eta$ makes an excursion away from where it hits its force point and the excursion terminates in $\R_+$ or $\R_+ + i \pi$, then $\eta_{\text{L}}'$ has to hit it at its tip and then fill the path in reverse chronological order a.s.  Hence, the excursion that $\eta$ completes at time $T_x$ is contained in $\eta_{\text{L}}'$ a.s.\ and corresponds to an excursion that $\eta_{\text{L}}'$ makes away from $\R_+ \cup (\R_+ + i \pi)$.

\noindent{\it Step 3.  Resampling argument.} Suppose that $\sigma$ is a stopping time for $\eta$.  We assume that we are working on the event that $\eta$ does not hit its force point at time $\sigma$ and that $S_x \leq \sigma \leq T_x$.  We then let~$\eta_{\text{R}}'$ be the counterflow line of the restriction of $h-\frac{\pi \chi}{2}$ to the connected component $D$ of $\strip \setminus \eta([0,\sigma])$ whose boundary contains $i\pi$, which starts from the point corresponding to where~$\eta$ has most recently before time $\sigma$ hit its force point, and it is targeted at $\eta(\sigma)$.  Then $\eta_{\text{R}}'$ has the law of an $\SLE_{\kappa}(\rho_{1,\text{L}}';\rho_{1,\text{R}}',\rho_{2,\text{R}}')$ process in $D$ with $\rho_{1,\text{L}}' = \kappa'/2-2$, $\rho_{1,\text{R}}' = -(\kappa'/4)(2+\rho)$,  and $\rho_{1,\text{R}}'+\rho_{2,\text{R}}'=\kappa'-4$,  and the force points are located immediately to the left and right of its starting point and at $i\pi$.  Indeed, the reason for this is that $\rho_{1,\text{L}}'$, $\rho_{1,\text{R}}'$, and $\rho_{2,\text{R}}'$ are determined by the equations
\[ \lambda'(1+\rho_{1,\text{L}}') = \lambda + \frac{\pi\chi}{2}, \quad -\lambda'(1+\rho_{1,{\mathrm R}}') = \lambda(1+\rho) + \frac{\pi\chi}{2} ,\quad\text{and}\quad -\lambda'(1+\rho_{1,{\mathrm R}}'+\rho_{2,{\mathrm R}}') = -\lambda- \frac{3\pi \chi}{2}.\]

Let $\ol{T}$ be a stopping time for the filtration generated by $\eta|_{[0,\sigma]}$ and the time-reversal of $\eta|_{[\sigma,T_x]}$.  Note that the segment of $\eta$ from $\eta(\sigma)$ to $\eta(T_x)$ in the clockwise order has boundary conditions $\lambda'+\chi \cdot \text{winding}$ (resp.  $-\lambda'+\chi \cdot \text{winding}$) on its right (resp.  left) side.  Also,  the flow line interaction rules imply that the right outer boundary of $\eta_{\text{R}}'$ is given by the flow line of $h|_D$,  starting from $\eta(\sigma)$.  Conditionally on $\eta|_{[0,\sigma]}$,  we let $T_q'$ be the first time that $\eta_q'$ hits $\eta(\ol{T})$ for $q \in \{L,R\}$.  Then,  $\eta([\ol{T},T_x])$ is equal to the common part of the outer boundaries of $\eta_{\text{L}}'([0,T_L'])$ and $\eta_{\text{R}}'([0,T_R'])$.  Let $\eta_{\text{L}}$ (resp.\ $\eta_{\text{R}}$) be the right (resp.\ left) outer boundary of $\eta_{\text{L}}'([0,T_{\text{L}}'])$ (resp.\ $\eta_{\text{R}}'([0,T_{\text{R}}'])$).  Then,  arguing as in the proof of \cite[Lemma~5.13]{ms2016imag2} gives that the conditional law of $\eta|_{[\sigma,\ol{T}]}$ given $\eta|_{[0,\sigma]}$, $\eta_{\text{L}}$, $\eta_{\text{R}}$,  and $\eta|_{[\ol{T},T_x]}$,  is that of an $\SLE_{\kappa}(\kappa/2-2;\kappa/2-2)$ process in $D$ starting from $\eta(\sigma)$ and targeted at $\eta(\ol{T})$ with the force points located at $i \pi$ and at the left of the starting point of $\eta_{\text{R}}'$.

\noindent{\it Step 4.  Conclusion of the proof.} 
Now we are going to identify the law of $\eta|_{[S_x,T_x]}$ given $\eta|_{[0,S_x]}$ and $\eta(T_x)$.  Conditionally on $\eta|_{[0,T_{x}]}$, we sample an $\SLE_{\kappa}(\rho+2 ; \kappa-4-\rho)$ process $\wt{\eta}$ in $D$ from~$\eta(\sigma)$ to~$\eta(T_x)$ where the left (resp.\ right) force point is located at $i \pi$ (resp.\ $(\eta(S_x))^+$).  Let~$\wt{T}$ be a stopping time for $\wt{\eta}$ and $\wt{T}'$ be a reverse stopping time for $\wt{\eta}$ given $\wt{\eta}|_{[0,\wt{T}]}$ and assume that we are working on the positive probability event that $\wt{\eta}|_{[0,\wt{T}]}$ and $\wt{\eta}|_{[\wt{T}',\infty)}$ are disjoint. Let $A = \eta([0,\sigma]) \cup \wt{\eta}$ and let $\wt{h}$ be $\text{GFF}$ on $\strip \setminus A$ with the same boundary conditions as $h$ on $\partial{(\strip \setminus A)} \setminus \wt{\eta}$ and with boundary conditions along $\wt{\eta}$ as if $\wt{\eta}$ were the flow line of $\wt{h}$ starting from~$\eta(\sigma)$ and targeted at~$\eta(T_x)$. Let $\wt{\eta}_{L}$ (resp.\ $\wt{\eta}_{R}$) be the flow line of $\wt{h}$ starting from $\wt{\eta}(\wt{T}')$ with angle $\pi$ (resp.\ $-\pi$) and targeted at $i\pi$ (resp.  the point corresponding to where $\eta$ has most recently before time $\sigma$ hit its force point). Then by construction we have that the conditional law of $\wt{\eta}_{L}$ and $\wt{\eta}_{R}$ given $A$ is the same as the conditional law of $\eta_{L}$ and $\eta_{R}$ given $\eta|_{[0,T_{u}]}$ in the analogous set up for $\eta$. Moreover,  the proof of \cite[Proposition~5.12]{ms2016imag2} implies that the conditional law of $\wt{\eta}|_{[\wt{T},\wt{T}']}$ given $\wt{\eta}_{L}, \wt{\eta}_{R}, \eta|_{[0,\sigma]}, \wt{\eta}|_{[0,\wt{T}]}$ and $\wt{\eta}|_{[\wt{T}',\infty)}$ is that of an $\SLE_{\kappa}(\kappa/2 - 2; \kappa/2 - 2)$ process in the connected component of $\strip \setminus (\eta([0,\sigma]) \cup \wt{\eta}_{L} \cup \wt{\eta}_{R} \cup \wt{\eta}([0,\wt{T}]))$ with $\wt{\eta}(\wt{T})$ on its boundary from $\wt{\eta}(\wt{T})$ to $\wt{\eta}(\wt{T}')$.  Since the stopping time $\sigma$ was arbitrary, the proof will be complete if we manage to prove that this resampling property determines the law of the triple $(\eta, \eta_{L}, \eta_{R})$ conditional on $\eta|_{[0,T]}$ and $\eta|_{[T',T_x]}$ where $T,T'$ are the corresponding stopping times for $\eta$.
To see this, let $T_{1}, T_{2}$ be two stopping times for $\eta_{L},\eta_{R}$ respectively. Then the law of the triple $(\eta_{1}, \eta_{L}$, $\eta_{R})$ with $\eta_{1}$ = $\eta|_{[T,T']}$ conditional on $\eta|_{[0,T]}$, $\eta|_{[T',T_x]}$ , $\eta_{L}|_{[0,T_{1}]}$ and $\eta_{R}|_{[0,T_{2}]}$ is the same as that of ordinary flow lines of a $\text{GFF}$ conditioned on the positive probability event that $\eta_{L}$ and $\eta_{R}$ do not intersect since this conditioned law satisfies the same resampling properties. In particular, the marginal law of $(\eta_{L},\eta_{R})$ satisfies the hypothesis of \cite[Proposition 5.10]{ms2016imag2} with $\theta_{2} = -\pi$ and $\theta_{1} = \pi$, and so its marginal law is uniquely determined. Since we know the conditional law of $\eta_{1}$ given $(\eta_{L},\eta_{R})$, the law of the triple $(\eta_{1}, \eta_{L}, \eta_{R})$ is uniquely determined and this completes the proof.
\end{proof}

\section{Poissonian structure}
\label{sec:main_theorems_proof}

The purpose of this section is to prove Theorems~\ref{thm:quantum_natural_time_cutting} and~\ref{thm:boundary_length_evolution} as well as Corollary~\ref{cor:dim_of_boundary_intersection}.  Fix $\kappa \in (0,4)$ and $\rho \in (\kappa/2-4,-2) \cap (-2-\kappa/2,-2)$,  and set $\gamma = \sqrt{\kappa} \in (0,2)$.  The main step in establishing all of these results is to describe the Poissonian structure of the quantum surfaces which are cut out of a quantum wedge of weight $\rho+4$ by an independent $\SLE_{\kappa}(\rho)$ process.  Our strategy will be similar to the one followed in \cite[Section~6]{dms2014mating} for the $\rho \in (-2,\kappa/2-2)$ case.  However, as discussed in Section~\ref{subsec:reverse_coupling},  there will be some key differences in the proofs since $\SLE_{\kappa}(\rho)$ processes are not simple for the range of $\rho$ that we consider.

As in \cite{dms2014mating},  associated with the quantum surfaces which are cut out by the $\SLE_{\kappa}(\rho)$ process is a local time which comes from the encoding of such surfaces using Bessel processes (see \cite[Section~4]{dms2014mating}).  By \cite[Proposition~19.12]{dms2014mating},  we have that this local time should be considered as counting the number of small bubbles which have been cut out by the $\SLE_{\kappa}(\rho)$ process.  More precisely,  it is the limit as $\epsilon \to 0$ of the number of excised bubbles of quantum mass in $[\epsilon,2\epsilon]$,  times an appropriate power of $\epsilon$.  If we consider the right-continuous inverse of this local time,  then we obtain a different time-parameterization for the $\SLE_{\kappa}(\rho)$ process which is intrinsic to the bubbles that it cuts out viewed as quantum surfaces.  Following the terminology of \cite{dms2014mating}, we refer to this notion of time as ``quantum natural time''.  We will show that when the $\SLE_{\kappa}(\rho)$ process is drawn on top of a certain independent quantum wedge and it is parameterized by quantum natural time,  then the law of the surfaces is invariant under zipping according to a fixed amount of quantum natural time,  and hence completing the proof of Theorem~\ref{thm:quantum_natural_time_cutting}.  Moreover,  by combining the previous results with standard properties of stable processes,  we will identify the law of the process associated with the evolution of the changes in boundary lengths, and hence complete the proof of Theorem~\ref{thm:boundary_length_evolution}.  Finally,  we will give the proof of Corollary~\ref{cor:dim_of_boundary_intersection} as a consequence of Theorem~\ref{thm:boundary_length_evolution}.  

As before, we let $\strip = \R \times (0,\pi)$ be the the strip.  We also let $\strip_+ = \R_+ \times (0,\pi)$ be the positive part of the strip and $\strip_- = \R_- \times (0,\pi)$ be the negative part of the strip.

\subsection{Strategy}
\label{subsec:strategy}

We will give an outline of the strategy used to prove Theorem~\ref{thm:quantum_natural_time_cutting}, as it will take quite a bit of work to complete.

Let $h$ be a free boundary GFF on~$\h$ with a log singularity at $0$ and which is coupled with a reverse $\SLE_{\kappa}(\wt{\rho})$ process with $\wt{\rho} = \rho + 4$ as in Theorem~\ref{thm:reverse_coupling}.  More precisely,  we have that $h = \wh{h} - \frac{\rho + 2}{\gamma}\log |\cdot|$,  where $\wh{h}$ is a free boundary $\text{GFF}$ on $\h$.  Next,  we let $\eta$ be an $\SLE_{\kappa}(\rho)$ process in $\h$ from $0$ to $\infty$ with a single boundary force point at $0_+$ and which is independent of $h$.  We will fix the additive constant for $h$ so that the average of $h \circ f_{T_u}^{-1} + Q \log |(f_{T_u}^{-1})'|$ on $\h \cap \partial \D$ is equal to $0$,  where $T_u$ is the right-continuous inverse of the local time of $V-W$ at $0$ with $(V,W)$ being the driving pair for~$\eta$,  and $(f_t)$ the centered forward Loewner evolution.  We will also take $u>0$ to be large but fixed.
\begin{enumerate}
\item[Step 1.] We will describe the Poissonian structure of the quantum surfaces which are cut out by $\eta|_{[0,T_u]}$ from $\infty$ ordered reverse chronologically in comparison to how they are first visited by $\eta|_{[0,T_u]}$.  We will in particular show they have the same structure as that of the bubbles arising from a weight $\rho + 2$ quantum wedge.  This step will be carried out in Section~\ref{subsec:lightcone_on_free_boundary_gff} just below (see Theorem~\ref{thm:wedge_cutting}).
\item[Step 2.] The weight $\rho+2$ quantum wedge structure of these bubbles leads to a quantum measure which is supported on the union of $\eta([0,T_u]) \cap \partial \h$ and the set of self-intersection points of $\eta$.  This measure corresponds to the local time of the Bessel process encoding the weight $\rho+2$ wedge. Using this, we will see that the bubbles cut out by $\eta|_{[0,T_u]}$ (at least near the typical point) ordered according to when they are first visited by $\eta$ have the structure of a weight $\rho+2$ wedge.  This will be described in Section~\ref{subsec:natural_time} below.
\item[Step 3.]  Combining Steps 1 and 2, we will complete the proof of the first part of Theorem~\ref{thm:quantum_natural_time_cutting} by identifying the law of the bubbles which are cut out by $\eta$ when drawn on top of an independent weight $\rho + 4$ quantum wedge.  The weight $\rho + 4$ appears because of the combination of the initial singularity at the origin given by $-(2+\rho)/\gamma \log |\cdot|$ and an additional $-\gamma (2-d) \log |\cdot-w|$ singularity with $d = 1 + (2/\gamma^2)(\rho + 2)$ which arises from observing the field at a quantum typical point $w$ (using the measure from Step 2). This will be described in Section~\ref{subsec:natural_time}.
\item[Step 4.]  As for the second part of Theorem~\ref{thm:quantum_natural_time_cutting},  it will follow from the fact that if we shift the bubbles near the typical point when viewed as quantum surfaces by a fixed amount of local time,  then their Poissonian structure remains the same.  This will be described in Section~\ref{subsec:natural_time}.
\end{enumerate}

Once we have proved Theorem~\ref{thm:quantum_natural_time_cutting}, we will use the scaling properties to determine the boundary length evolution and complete the proof of Theorem~\ref{thm:boundary_length_evolution} in Section~\ref{subsec:boundary_length_evolution}.  We will also give the proof of Corollary~\ref{cor:dim_of_boundary_intersection} here.

\subsection{Poissonian structure of light cone path on a free boundary GFF}
\label{subsec:lightcone_on_free_boundary_gff}

We will now carry out Step 1 as described in Section~\ref{subsec:strategy}.  That is, we will determine the law of the quantum surfaces which are cut out by an $\SLE_{\kappa}(\rho)$ process for $\kappa \in (0,4)$ and $\rho \in (\kappa/2 - 4, - 2) \cap (- 2 - \kappa/2, - 2)$ drawn on top of a free boundary $\text{GFF}$ with an appropriate log singularity and with the additive constant fixed in a certain way.

\subsubsection{Statement}
\begin{theorem}
\label{thm:wedge_cutting}
Fix $\kappa \in (0,4)$, $\rho \in (\kappa/2 - 4 , - 2) \cap ( - 2 - \kappa/2 , - 2)$, $\wt{\rho} = \rho + 4$ and let $\gamma = \sqrt{\kappa}$. Let $\wh{h}$ be a free boundary $\text{GFF}$ on $\h$ and let 
\begin{align*}
h = \wh{h} + \frac{(2 - \wt{\rho})}{\gamma}\log|\cdot | = \wh{h} - \frac{2 + \rho}{\gamma}\log| \cdot|
\end{align*}
Let $\eta$ be an $\SLE_{\kappa}(\rho)$ process starting from $0$ with a single boundary force point of weight $\rho$ located at $0_+$ and assume that $\eta$ is independent of $h$. Let $(f_{t})$ be the centered Loewner flow for $\eta$, $(W , V)$ be its driving process, $L_{t}$ be the local time of $V - W$ at $0$, and let $T_{u} = \inf \{s >0 : L_{s} > u\}$ be the right-continuous inverse of $L_{t}$. Fix $u > 0$ large and assume that the additive constant for $h$ is fixed such that the average of $h \circ f_{T_{u}}^{-1} + Q\log|(f_{T{u}}^{-1})'|$ on $\h \cap \partial \D$ is equal to $0$. Then the law of the quantum surfaces parameterized by the bounded components of $\h \setminus  \eta([0,T_{u}])$ ordered in reverse from when they are first visited by $\eta$, is equal to those of a weight $\rho + 2$ wedge up to a time $s_{u}$ which tends to $\infty$ in probability as $u \to \infty$.
\end{theorem}

Note that in our setup there are some subtleties in the proof of Theorem~\ref{thm:wedge_cutting} since the force point of an $\SLE_{\kappa}(\rho)$ process for the range of the values of $\rho$ that we consider might not lie on the boundary of the domain.  However the main idea of the proof of \cite[Theorem~6.1]{dms2014mating} still works in our case.

\subsubsection{Setup}
\label{sec:setup}

First, we recall from~\eqref{eqn:bessel_dimension_wrt_weight} that the dimension of the Bessel process $Y$ associated with a wedge of weight $\rho + 2$ is given by:
\begin{align}
\label{eq:bessel_dimension}
    \delta = \delta (\kappa , \rho) = 1 +  \frac{2\rho + 4}{\kappa} = 1 + \frac{2\rho + 4}{\gamma^{2}}
\end{align}
Note that the value of $\delta(\kappa , \rho)$ varies in $(0,1)$ as $\rho$ varies in $(-\kappa/2 -2 , -2)$.

Let $(g_{t})$  be the family of conformal maps solving the Loewner equation driven by $W$ and for all $t \geq 0$ let $P_{t}$ be the point of last intersection of $\eta$ either with itself or $\partial \h$ by time $t$. Then we know that $X =\kappa^{-\frac{1}{2}}(V - W)$ has the law of a Bessel process starting from zero of dimension $\delta(\kappa,\rho)$ where $V_{t} = g_{t}(P_{t}^+)$.

Let $(T_{u})$ be the family of the right-continuous inverses for the local time of $X$ at $0$.  We also consider $\wt{\eta}^0$ to be an $\SLE_{\kappa}(\rho)$ in $\h$ from $0$ to $\infty$ which is independent of $W$ and $h$ a free boundary $\text{GFF}$ on $\h$ independent of $(W , \wt{\eta}^0)$. Fix $u > 0$ and note that $T_{u} \to \infty$ a.s.\ as $u \to \infty$. We consider the process $\wt{X} = (\wt{X}_{t})$ with $\wt{X}_{t} = X_{T_{u} -t}$ for $0 \leq t \leq T_{u}$ and $\wt{X}_{t} = X_{t}$ for $t \geq T_{u}$, and with corresponding family of local times $(\wt{L}_{t}^x)$. We also consider the process $\wt{V} = (\wt{V}_{t})$ with $\wt{V}_{t} = V_{T_{u} - t} - V_{T_{u}}$ for $0 \leq t \leq T_{u}$ and $\wt{V}_{t} = - V_{t}$ for $t \geq T_{u}$. Note that the time-reversal symmetries of Bessel processes (see \cite[Proposition~3.1]{dms2014mating}) imply that $X$ and $\wt{X}$ have the same law. It is also easy to see that
\[ \wt{V}_{t} = -\frac{2}{\sqrt{\kappa}}\text{P.V.} \int_{0}^{t}\frac{1}{\wt{X}_{s}}ds.\]
Thus $\wt{W}$ has the law of a reverse Loewner driving function, where $\wt{W}_t = W_{T_u - t} - W_{T_u}$ for $0 \leq t \leq T_u$ and $\wt{W}_t = -V_t - \sqrt{\kappa}X_t$ for $ t \geq T_u$.

Let $(\wt{g}_t)$ be the family of conformal maps solving the reverse Loewner equation driven by $\wt{W}$ with $\wt{g}_0(z) = z$ and $\wt{f}_t(z) = \wt{g}_t(z)-\wt{W}_t$.  Let also $\wt{\eta}^{T_{u}}$ be the random curve such that $\wt{\eta}^{T_{u}}|_{[0,T_{u}]}$ has driving function $W^{T_{u}}$ given by $t \mapsto \wt{W}_{T_{u} - t} - \wt{W}_{T_{u}}$ and $\wt{\eta}^{T_{u}}(s) = \wt{f}_{T_u}(\wt{\eta}^0(s - T_{u}))$ for $s \geq T_{u}$.  For $0 \leq s \leq t$ we set $\wt{f}_{s,t} = \wt{f}_{t} \circ (\wt{f}_{s})^{-1}$ and for $0 \leq t \leq T_{u}$ we consider the continuous curve $\wt{\eta}^t$ defined by $\wt{\eta}^t(s) = (\wt{f}_{t,T_{u}})^{-1}(\wt{\eta}^{T_{u}}(T_{u} - t + s))$ for $0 \leq s \leq t$ and $\wt{\eta}^t(s) = \wt{f}_{t}(\wt{\eta}^0(s - t))$ for $s \geq t$. It is easy to check that $W_{s}^t = \wt{W}_{t - s} - \wt{W}_{t}$ and $g_{s}^t(z) = \wt{g}_{t - s}((\wt{f}_{t})^{-1}(z)) - \wt{W}_{t}$ where $W^t$ is the driving function of $\wt{\eta}^t|_{[0,t]}$ and $(g_{s}^t)$ the conformal maps solving the Loewner equation driven by $W^t$.

For all $0 \leq s \leq t \leq T_{u}$, let $P_{s}^t$ be the last point of intersection of $\wt{\eta}^t$ before time $s$ either with itself or with $\partial \h$, and set $V_{s}^t = g_{s}^t(P_{s}^t)$. It is again easy to see that $V_{s}^t = V_{T_{u} - t + s}^{T_{u}} - W_{T_{u} - t}^{T_{u}}$ by applying certain uniqueness arguments for the Loewner equation.  Also, the above imply that $\wt{U}_{t} = f_{T_{u} - t}^{T_{u}}(\wt{\eta}^{T_{u}}(s))$ where $s = \sup\{0 \leq r \leq T_{u} - t : \wt{\eta}^{T_{u}}(r) \in \partial \h  \cup \wt{\eta}^{T_{u}}([0,r))\}$ and $\wt{U}_{t} = \sqrt{\kappa}\wt{X}_{t}$.

By construction, $\wt{\eta}^{T_{u}}$ has the law of an $\SLE_{\kappa}(\rho)$ and so if $P$ is a pocket drawn by $\wt{\eta}^{T_{u}}$ and $S_{1}(P)$ is the left side of $P$ with zero angle boundary conditions, then every point on the right side of $S_{1}(P)$ is hit exactly once. Moreover, since the set of times during which $\wt{\eta}^{T_{u}}$ draws $S_{1}(Q)$ for some pocket $Q$, is dense in $\R_+$, we obtain that $\wt{\eta}^{T_{u}}((T_{u} - t,\infty))$ does not intersect the right side of $\wt{\eta}^{T_{u}}([s,T_{u} - t])$. This shows that $\wt{U}_{t}$ lies on the boundary of the connected component of $\h \setminus f_{T_{u} - t}^{T_{u}}(\wt{\eta}^{T_{u}}([T_{u} - t,\infty)))$ whose boundary contains $0_+$. Also, $\wt{\eta}^t([0,t]) = f_{T_{u} - t}^{T_{u}}(\wt{\eta}^{T_{u}}([T_{u} - t,T_{u}]))$ and $\wt{\eta}^t([t,\infty)) = \wt{f}_{t}(\wt{\eta}^0([0,\infty))) = f_{T_{u} - t}^{T_{u}}(\wt{\eta}^{T_{u}}([T_{u},\infty)))$ and hence $\wt{\eta}^t([0,\infty)) = f_{T_{u} - t}^{T_{u}}(\Tilde{\eta}^{T_{u}}([T_{u}-t,\infty)))$. Therefore $\wt{U}_{t}$ lies on the connected component of $\h \setminus \wt{\eta}^t([0,\infty))$ whose boundary contains the origin. We also observe that $\wt{\eta}^t([t,t + r]) = \wt{f}_{t}(\wt{\eta}^0([0,r]))$ for all $0 \leq t \leq T_{u}$ and for all $r > 0$, and that the hull of $\wt{\eta}^t([0,t])$ is equal to the hull of $\wt{f}_{t}(f_{T_{u}}^{T_{u}}(\wt{\eta}^{T_{u}}([0,T_{u}])))$. But $f_{T_{u}}^{T_{u}}(\wt{\eta}^{T_{u}}([0,T_{u}])) \subseteq \partial \h$ and so the hull of $\wt{\eta}^t([0,t + r])$ is equal to the hull of $\wt{f}_{t}(\partial \h \cup \wt{\eta}^0([0,r]))$.

Now set 
\begin{align*}
\wt{h}^{T_{u}} = h + \frac{2}{\gamma}\log|\cdot| - \frac{\wt{\rho}}{\gamma}\log|\cdot- \wt{U}_{T_u}| = h + \frac{2}{\gamma}\log|\cdot| - \frac{\wt{\rho}}{\gamma}\log|\cdot|
\end{align*}

By Theorem~\ref{thm:reverse_coupling}, we obtain that the random distributions
\begin{align}\label{eq:law_equivalence}
      &h + \frac{2 - \wt{\rho}}{\gamma}\log|\cdot|, \nonumber \\
     & h \circ \wt{f}_{T_{u}} + \frac{2-\wt{\rho}}{\gamma}\log|\wt{f}_{T_{u}}| +Q\log|(\wt{f}_{T_{u}})'|, \quad\text{and} \nonumber  \\
     & h \circ \wt{f}_{t} + \frac{2}{\gamma}\log|\wt{f}_{t}| - \frac{\wt{\rho}}{\gamma}\log|\wt{f}_{t} - \wt{U}_{t}| +Q\log|(\wt{f}_{t})'|
\end{align}
have the same law for all $0 \leq t \leq T$. Set 
\begin{align*}
\wt{h}^t = \wt{h}^{T_{u}} \circ \wt{f}_{t,T_{u}} + Q\log|(\wt{f}_{t,T_{u}})'|
\end{align*}
Then, the equality in law in~\eqref{eq:law_equivalence} implies that $\wt{h}^t$ and $h + \frac{2}{\gamma}\log|z| - \frac{\wt{\rho}}{\gamma}\log|z - \wt{U}_{t}|$ have the same law for all $0 \leq t \leq T_u$.

The above observations show that for every fixed $u \in (0,\infty)$, we can find a random family $(\wt{h}^t,\wt{\eta}^t)_{0\leq t \leq T_u}$, where $\wt{\eta}^t$ is a random continuous curve in $\h$ from $0$ to $\infty$ and $\wt{h}^t$ is a random distribution on $\h$ such that the following hold;
\begin{enumerate}[(i)]
\item $(\wt{\eta}^t)_{0 \leq t \leq T_u}$ is independent of $\wt{h}^0$ and $\wt{\eta}^0$ has the law of an $\SLE_{\kappa}(\rho)$.
\item $(W_{T_u - t}^{T_u} - W_{T_u}^{T_u})_{0 \leq t \leq T_u}$ has the law of a reverse Loewner driving flow restricted to $[0,T_u]$ with dimension $\delta(\kappa,\rho)$ and $(\wt{U}_{t}^t)_{0 \leq t \leq T_u}$ has the law of a Bessel process with dimension $\delta(\kappa,\rho)$ and multiplied by $\sqrt{\kappa}$, restricted to $[0,T_u]$, where $\wt{U}_{t}^{T_u} = V_{T_u - t}^{T_u} - W_{T_u - t}^{T_u} + V_{T_u}^{T_u} - W_{T_u}^{T_u}$. Here, $W^{T_u} |_{[0,T_u]}$ is the Loewner driving function of $\wt{\eta}^{T_u} |_{[0,T_u]}$.
\item  If $(\wt{f}_{t}^{T_u})_{0 \leq t \leq T_u}$ is the reverse centered Loewner flow associated with $(W_{T_u - t}^{T_u} - W_{T_u}^{T_u})_{0 \leq t \leq T_u}$, then for all $0 \leq t \leq T_u$ and $r >0 $, the hull of $\Tilde{\eta}^{t}([0,t + r])$ is equal to the hull of $\wt{f}_{t}(\partial \h \cup \wt{\eta}^0([0,r]))$.
\item $\wt{h}^s = \wt{h}^t \circ \wt{f}_{s,t}^{T_u} + Q\log|(\wt{f}_{s,t}^{T_u})'|$ and $\wt{h}^t$ can be expressed as the sum of a free boundary $\text{GFF}$ on $\h$ plus $\frac{2}{\gamma}\log|z| - \frac{\wt{\rho}}{\gamma}\log|z - \wt{U}_{t}|
$ for all $0 \leq s \leq t \leq T_u$.
\item $\wt{U}_{t}^{T_u}$ lies on the boundary of the component of $\h \setminus \wt{\eta}^t([0,\infty))$ whose boundary contains 0 for all $t \in [0,T_u]$.
\end{enumerate}

The following lemma gives us a way to identify the laws of the fields which arise when performing the zipping and the unzipping operations.  It is the analogue of \cite[Lemma~6.2]{dms2014mating}.

\begin{lemma}[$\text{\cite[Lemma~6.2]{dms2014mating}}$]
\label{lem:extension_lemma}
There exists a random family $(\wt{h}^t,\wt{\eta}^t)_{t \geq 0}$ where $\wt{h}^t$ is a distribution on $\h$ and $\wt{\eta}^t$ is a continuous curve on $\overline{\h}$ from $0$ to $\infty$ such that the following hold:
\begin{enumerate}[(i)]
\item $(\wt{\eta}^t)_{t \geq 0}$ is independent of $\wt{h}^0$ and $\wt{\eta}^0$ has the law of an $\SLE_{\kappa}(\rho)$.
\item $\wt{W} = (\wt{W}_{t})_{t \geq 0}$ with $\wt{W}_{t} = -W_{t}^t$ has the law of a reverse Loewner driving function with dimension $\delta$ and let $(\wt{f}_{t})_{t \geq 0}$ be the corresponding centered reverse Loewner flow. Also, $\wt{U} = (\wt{U}_{t})_{t \geq 0}$ with $\wt{U}_{t} = \wt{U}_{t}^t$ has the law of a Bessel process with dimension $\delta$ multiplied by $\sqrt{\kappa}$.
\item The hull of $\wt{\eta}^t([0,t + r])$ is equal to the hull of $\wt{f}_{t}(\partial \h  \cup \wt{\eta}^0([0,r]))$ for all $t \geq 0$ and $r > 0$.
\item  $\wt{U}_{t}$ lies on the boundary of the component of $\h \setminus \wt{\eta}^t([0,\infty))$ whose boundary contains the origin.
\item $\wt{h}^s = \wt{h}^t \circ \wt{f}_{s,t} + Q\log|(\wt{f}_{s,t})'|$ and $\wt{h}^t$ can be expressed as the sum of a free boundary $\text{GFF }$ on $\h$ plus $\frac{2}{\gamma}\log|z| - \frac{\wt{\rho}}{\gamma}\log|z - \wt{U}_{t}|$ for all $0 \leq s \leq t$.
\item $\wt{W} = (\wt{W}_{t})_{t \geq 0}$ is the reverse Loewner driving function corresponding to $\wt{U}$.
\end{enumerate}
\end{lemma}

\begin{proof}
Noting that Theorem~\ref{thm:reverse_coupling} is an analogue of \cite[Theorem~5.1]{dms2014mating},  the proof follows from the same argument used to prove \cite[Lemma~6.2]{dms2014mating}.
\end{proof}
We will assume throughout that $(\wt{h}^t)$ and $(\wt{\eta}^t)$ are as in the statement of Lemma~\ref{lem:extension_lemma}. We fix the additive constant for $\wt{h}^0$, hence $\wt{h}^t$ for all $t \geq 0$, by taking its average on $\h \cap \partial \D$ to be equal to $0$.  When shifting the above collection of pairs of fields and curves by a fixed amount of local time,  it will be important to identify the law of the resulting pair as seen by the description of our strategy for the proof of Theorem~\ref{thm:quantum_natural_time_cutting}.  It turns out that the laws of the pairs are invariant when zipping by a fixed amount of local time.  However,  we note that for general $t \geq 0$,  the curve $\wt{\eta}^t$ does not necessarily have the law of a forward $\SLE_{\kappa}(\rho)$ process.

\begin{lemma}[$\text{\cite[Lemma~6.3]{dms2014mating}}$]                                        \label{lem:equality_in_law}
Suppose that we have the setup as described above.  Let $\wt{L}_{t}$ denote the local time of $\wt{U}_{t}$ at $0$ and, for each $u > 0$, we let $\wt{T}_{u}$ be the right-continuous inverse of $\wt{L}_{t}$. For each $u > 0$, we have that $\wt{\eta}^{\wt{T}_{u}}$ has the law of a forward $\SLE_{\kappa}(\rho)$ process. In particular,  the laws of the fields $(\wt{h}^{\wt{T}_{u}},\wt{\eta}^{\wt{T}_{u}})$ and  $ (\wt{h}^0 , \wt{\eta}^0)$ are the same,
where $\wt{h}^{\wt{T}_{u}}$ is viewed as a distribution modulo additive constant.
\end{lemma}
\begin{proof}
It follows from the proof of \cite[Lemma~6.3]{dms2014mating}.
\end{proof}

We are now going to define the objects which are used in the proof of Theorem~\ref{thm:wedge_cutting}.  We will follow the same terminology as in \cite[Section~6.2]{dms2014mating}.  For each $t \geq 0$ such that $\wt{U}_{t} \neq 0$, we let $\wt{B}_{t}$ denote the component of $\h \setminus \wt{\eta}^t([0,\infty))$ with $\wt{U}_{t}$ on its boundary. If $\wt{U}_{t} = 0$, we take $\wt{B}_{t} = \emptyset$. $\wt{B}_{t}$ is a Jordan domain as its boundary consists of part of the curve $\wt{\eta}^t$ and an interval in $\partial \h$.  For $t\geq 0$ such that $\wt{B}_{t} \neq \emptyset$, we let
\begin{enumerate}[(i)]
\item $\wt{\zeta}_{t}$ be the closing point of $\partial{\wt{B}_{t}}$ (the opening point is $0$). 
\item $\wt{\varphi}_{t}: \wt{B}_{t}\rightarrow \strip$ be the unique conformal transformation which takes $\wt{U}_{t}$ to $+\infty$, $0$ to $0$, and $\wt{\zeta}_{t}$ to $-\infty$. 
\item $\wh{h}^t = \wt{h}^t \circ \wt{\varphi}_{t}^{-1} + Q\log|(\wt{\varphi}_{t}^{-1})'|$.
\item $\wh{X}^t$ be the projection of $\wh{h}^t$ onto $\CH_{1}(\strip)$
\end{enumerate}

We note that in the setup of Lemma~\ref{lem:extension_lemma},  the transformation from $(\wt{h}^s,\wt{\eta}^s)$ to $(\wt{h}^t,\wt{\eta}^t)$ for $t>s$ corresponds to zipping up for $t-s$ units of capacity time and the transformation from $(\wt{h}^s,\wt{\eta}^s)$ to $(\wt{h}^t,\wt{\eta}^t)$ for $t<s$ corresponds to unzipping for $s-t$ units of capacity time.  The reason that we consider the quantum surfaces parameterized by the domains $\wt{B}_t$ is the following.  When we perform the forward Loewner flow at times when the $\SLE_{\kappa}(\rho)$ process completes an entire excursion either away from $\partial \h$ or away from itself,  we have that a region with finite and positive quantum area parameterized by some $\wt{B}_t$ is cut out from $\infty$.  If we apply the reverse Loewner flow,  then a region with positive quantum area parameterized by some $\wt{B}_t$ is added all at once. Moreover,  we have that every quantum surface parameterized by a bubble lying to the right of the $\SLE_{\kappa}(\rho)$ process $\eta$ will correspond to a newly zipped in quantum surface parameterized by some $\wt{B}_t$ (modulo coordinate change).  Furthermore,  the latter is determined by the values of the field in a very small region and the conditional law of the field given its values in that region is very close to the unconditioned law.  The latter fact will lead to the quantum surfaces which are cut out by $\eta$ to be independent.

\subsubsection{Filtration and stopping times}
First,  we start by constructing the filtration and stopping times that we are going to consider.  The definitions and notation are similar to those in \cite[Section~6.2.3]{dms2014mating}.  Suppose that we have the setup of Lemma~\ref{lem:extension_lemma}.  When performing the zipping and the unzipping operations,  we need to keep track of the path and the field as well.  However,  as we mentioned at the end of Section~\ref{sec:setup},  we will only need to know the values of the field outside of the region that is currently zipped in.  Also,  it will be convenient to use $\strip$ as the underlying domain.  We also note that the dimension of the Bessel process encoding the bubbles which are zipped in is determined by the form of the projection of the field which describes such a bubble onto $\mathcal{H}_1(\strip)$ with fixed additive constant and it does not depend on its projection onto $\mathcal{H}_2(\strip)$.

The above observations motivate to consider the following filtration.  Consider the process $\wt{V} = (\wt{V}_{s})_{s \geq 0}$ with $\wt{V}_{s} = \wt{U}_{s} + \wt{W}_{s}$ and for $t \geq 0$ we let $\wt{\CF_{t}}$ be the $\sigma$-algebra with respect to which the following are measurable:
\begin{enumerate}[(i)]
\item Both $\eta = \wt{\eta}^0$ and the driving pair $(\wt{W}_{s},\wt{V}_{s})$ of the reverse Loewner flow for $s \leq t$.
\item The field $\wt{h}^t$ restricted to $\h \setminus \wt{B}_{t}$ (i.e., $\wt{h}^t$ restricted to any component of $\h \setminus \wt{\eta}^t$ other than the one currently been generated).
\item The restriction of $\wh{h}^t$ to $\strip_-$.
\end{enumerate}

Note that $\wt{\eta}^t([0,t]) = f_{T_u - t}^{T_u}(\wt{\eta}^{T_u}([T_u - t ,  T_u]))$ for $t < T_u$ and so $\wt{\eta}^t([0,t])$ is determined by $W_{s+T_u -t}^{T_u}-W_{T_u - t}^{T_u},  s \in [0,t]$.  But $\wt{\eta}^{T_u}$ hits either itself or $\R_+$ at time $T_u$ and so $\wt{B}_t$ is a connected component of $\h \setminus \wt{\eta}^t([0,t])$.  Also,  $W_{s+T_u - t}^{T_u} - W_{T_u}^{T_u} = \wt{W}_{t-s}-\wt{W}_t$ for all $s \in [0,t]$ and so $\wt{B}_t$ is determined by $(\wt{W}_s,\wt{V}_s),  0 \leq s \leq t$.  It follows that $\wt{\zeta}_t,\wt{\varphi}_t$,  and $\wh{X}^t|_{\R_-}$ are determined by $\wt{\CF}_t$.  Moreover,  by arguing as in \cite[Section~6.2.3]{dms2014mating}, we obtain that $(\wt{\CF}_{t})$ is a filtration.

Next we fix $t_{1},t_{2} \in (0,\infty)$ such that $\wt{U}_{t_{1}} = \wt{U}_{t_{2}} = 0$ and $\wt{U}_{r} > 0 $ for all $r \in (t_{1},t_{2})$. Fix $t_{0} \in (t_{1},t_{2})$. Then the quantum length of $\partial{\wt{B}_{t}}$ with respect to $\wt{h}^t$ remains the same for all $t \in (t_{1},t_{2})$, since $\wt{h}^t|_{\wt{B}_{t}}$ and $\wt{h}^s|_{\wt{B}_{s}}$ remain equivalent as surfaces for all $t,s \in (t_{1},t_{2}) $. Also the same holds for the quantum length of $[\wt{U}_{t},\wt{\zeta}_{t}]$ with respect to $\wt{h}^t$ for $t \in (t_{1},t_{2})$. This shows that the quantum length of $\partial{\wt{B}_{t}} \setminus \partial{\h}$ tends to $0$ as $t \to t_{1}$ and it tends to the quantum length of the clockwise part of $\partial{\wt{B}_{t_{0}}}$ from $\wt{U}_{t_{0}}$ to $\wt{\zeta}_{t_{0}}$ as $t \to t_{2}$. Hence if $\alpha(t)$ is such that $\wh{h}^t$ comes by translating $\wh{h}^{t_{0}}$ horizontally by $\alpha(t)$ units, then $\alpha(t) \to -\infty$ as $t \to t_{1}$ and $\alpha(t) \to +\infty$ as $t \to t_{2}$, and $\alpha$ is continuous in $t$. Therefore if $\sup_{u \in \R}(\wh{X}^{t_{0}}_{u}) \geq r$, then there exists $t \in (t_{1},t_{2})$ such that $\inf\{u \in \R : \wh{X}^{t}_{u} = r\} =0$.

Now we are ready to define the stopping times that we are going to consider.  We note that our goal is to analyze the structure of the bubbles which are zipped in when viewed as quantum surfaces modulo coordinate change.  It will be convenient for the proof of Theorem~\ref{thm:quantum_natural_time_cutting} to analyze the structure of a given bubble which has only been partially zipped in.  Also,  we will first analyze the collection of ``large bubbles'' in the sense of quantum area and then obtain an asymptotic coupling of these bubbles with the bubbles cut out from the $\SLE_{\kappa}(\rho)$ process as their quantum areas tend to zero.  Following again \cite[Section~6.2.3]{dms2014mating},  we define the stopping times as follows.

Fix $\epsilon > 0, \overline{\epsilon} > 0$ and $r \in \R$. We define stopping times inductively as follows. We let $\wt{\tau_{1}}^{\epsilon,\overline{\epsilon},r}$ be the first time $t \geq 0$ such that the following hold simultaneously:
\begin{enumerate}[(i)]
\item The quantum length of $\partial{\wt{B}_{t}} \setminus \partial{\h}$ is at least $\epsilon$ with respect to the $\gamma$-$\text{LQG}$ length measure induced by $\wt{h}^t$,
\item $\inf\{u \in \R : \wh{X}^{t}_{u} = r\} =0$, and 
\item The $\gamma$-$\text{LQG}$ mass of $\strip_-$ associated with the field $\wh{h}^t$ is at least $\overline{\epsilon}$.
\end{enumerate}
We note that the time $\wt{\tau}_{1}^{\epsilon,\overline{\epsilon},r}$ is a.s.\ finite since our previous observations imply that if $\sup_{u \in \R}\wh{X}^{t}_{u}$ is at least $r$, there will be a time $s$ such that $\wh{X}^s$ first hits $r$ at $u = 0$.

We then let $\wt{\sigma}_{1}^{\epsilon,\overline{\epsilon},r}$ be the first time $t$ after time $\wt{\tau}_{1}^{\epsilon,\overline{\epsilon},r}$ that $\wt{U}_{t} = 0$. Given that $\wt{\tau}_{j}^{\epsilon,\overline{\epsilon},r},\wt{\sigma}_{j}^{\epsilon,\overline{\epsilon},r}$ for $1 \leq j \leq k$ have been defined for some $k \in \N$, we let $\wt{\tau}_{k + 1}^{\epsilon,\overline{\epsilon},r}$ be defined in exactly the same way as $\wt{\tau}_{1}^{\epsilon,\overline{\epsilon},r}$ except with $t \geq \wt{\sigma}_{k}^{\epsilon,\overline{\epsilon},r}$. We then let $\wt{\sigma}_{k}^{\epsilon,\overline{\epsilon},r}$ be the first time $t$ after time $\wt{\tau}_{k + 1}^{\epsilon,\overline{\epsilon},r}$ that $\wt{U}_{t} = 0$.

For each $j$, we let $\wh{X}^{j,\epsilon,\overline{\epsilon},r}$ be given by $\wh{X}^{\wt{\tau}_{j}^{\epsilon,\overline{\epsilon},r}}$ and $\wh{h}_{j}^{\epsilon,\overline{\epsilon},r}$ be given by $\wh{h}^{\wt{\tau}_{j}^{\epsilon,\overline{\epsilon},r}}$. We emphasize that the process $ \wh{X}_{u}^{j,\epsilon,\overline{\epsilon},r}$ hits $r$ for the first time at time $0$.

As we mentioned earlier,  the law of the sequences $(\wh{X}^{j,\epsilon,\overline{\epsilon},r})$ and $(\wh{X}^{j,\overline{\epsilon},r}) = (\wh{X}^{j,0,\overline{\epsilon},r})$ will determine the dimension of the Bessel process encoding the quantum surfaces parameterized by the bubbles which are zipped in.  In particular,  we have that $(\wh{X}^{j,\epsilon,\overline{\epsilon},r})$ and $(\wh{X}^{j,\overline{\epsilon},r})$ are both given by independent Brownian motions with a common downward drift.  Following \cite[Section~4]{dms2014mating},  we have that each such Brownian motion corresponds to an excursion of the Bessel process away from $0$.  We record the result about the form of the drift in the following proposition.

\begin{proposition}[$\text{\cite[Proposition~6.4]{dms2014mating}}$]\label{prop:form_of_drift}
Let 
\begin{align*}
\alpha = \frac{\rho + 4}{\gamma} - Q = \frac{\rho + 2}{\gamma} - \frac{\gamma}{2}.
\end{align*}
Let $\wh{B}_{2t}^{j,\epsilon,\wt{\epsilon},r} = \alpha t - \wh{X}_{t}^{j,\epsilon,\wt{\epsilon},r}$ for $t \geq 0$. Given $\wt{\CF}_{\tau_{j}^{r,\epsilon,\overline{\epsilon}}}$, the process $\wh{B}^{j,\epsilon,\overline{\epsilon},r}$ is a standard Brownian motion with $\wh{B}_{0}^{j,\epsilon,\overline{\epsilon},r} = -r$.
\end{proposition}
\begin{proof}
It follows from the proof of \cite[Proposition~6.4]{dms2014mating}.
\end{proof}

\subsubsection{Proof of Theorem~\ref{thm:wedge_cutting}}

Proposition~\ref{prop:form_of_drift} implies that the projections of the fields $(\wh{h}_j^{\epsilon,\overline{\epsilon},r})_{j \in \N}$ onto $\mathcal{H}_1(\strip)$ and restricted to $\strip_+$ are i.i.d.  and they are related to the excursions of a Bessel process from $0$.  However,  in order to prove Theorem~\ref{thm:wedge_cutting},  we will need to know the dependency between the projections of the fields onto $\mathcal{H}_2(\strip)$ and restricted to $\strip_+$.  Note that the Markov property of the $\text{GFF}$ implies that conditionally on $\wt{\CF}_{\wt{\tau}_j^{\epsilon,\overline{\epsilon},r}}$,  the field $\wh{h}_j^{\epsilon,\overline{\epsilon},r}|_{\strip_+}$ can be expressed as the sum of a $\text{GFF}$ in $\strip_+$ with Dirichlet boundary conditions on $[0,\pi i]$ and free boundary conditions on $\partial \strip_+ \setminus [0,\pi i]$.  Hence,  it follows that the way the field $\wh{h}_j^{\epsilon,\overline{\epsilon},r}|_{\strip_+}$ is correlated with the fields $\wh{h}_1^{\epsilon,\overline{\epsilon},r},\cdots,\wh{h}_{j-1}^{\epsilon,\overline{\epsilon},r}$ is encoded by the function which is harmonic in $\strip_+$ with boundary conditions given by the values of $\wh{h}_j^{\epsilon,\overline{\epsilon},r}$ on $[0,\pi i]$ and Neumann boundary conditions on $\partial \strip_+ \setminus [0,\pi i]$.

Next we follow the notation and the strategy of \cite[Section~6.2.5]{dms2014mating}. First,  we note that the quantum length of $(-\infty,0]$ with respect to $\wh{h}_{j}^{\epsilon,\overline{\epsilon},r}$ is at least $\epsilon > 0$ by the definition of $\wt{\tau}_{j}^{\epsilon,\overline{\epsilon},r}$, and so there exists a unique $\wh{\omega}_{j}^{\epsilon,\overline{\epsilon},r} \in (-\infty,0]$ such that the quantum length of $(-\infty,\wh{\omega}_{j}^{\epsilon,\overline{\epsilon},r}]$ with respect to $\wh{h}_{j}^{\epsilon,\overline{\epsilon},r}$ is equal to $\epsilon > 0$. Also, a.s.\ there exist  unique $t_{j,1}^{\epsilon,\overline{\epsilon},r},t_{j,2}^{\epsilon,\overline{\epsilon},r} > 0$ such that $\wt{\tau}_{j}^{\epsilon,\overline{\epsilon},r} \in (t_{j,1}^{\epsilon,\overline{\epsilon},r},t_{j,2}^{\epsilon,\overline{\epsilon},r})$ and $\wt{U}_{t_{j,1}^{\epsilon,\overline{\epsilon},r}} = \wt{U}_{t_{j,2}^{\epsilon,\overline{\epsilon},r}} = 0$. We let $t_{j}^{\epsilon,\overline{\epsilon},r}$ be the first time $t$ in $(t_{j,1}^{\epsilon,\overline{\epsilon},r},\wt{\tau}_{j}^{\epsilon,\overline{\epsilon},r}]$ such that the quantum length of $\partial{\wt{B}_{t}} \setminus \partial{\h}$ with respect to $\wt{h}^t$ is equal to $\epsilon > 0$. Equivalently, $t_{j}^{\epsilon,\overline{\epsilon},r}$ is the first time $t$ in $(t_{j,1}^{\epsilon,\overline{\epsilon},r},\wt{\tau}_{j}^{\epsilon,\overline{\epsilon},r}]$ such that the quantum length of $(-\infty,0]$ with respect to $\wh{h}^t$ is equal to $\epsilon > 0$. Note that $\wh{h}^{t_{j}^{\epsilon,\overline{\epsilon},r}}$ can be derived by translating $\wh{h}_{j}^{\epsilon,\overline{\epsilon},r}$ horizontally by $\wh{\omega}_{j}^{\epsilon,\overline{\epsilon},r}$ units. Let $\wh{\psi}_{j}^{\epsilon,\overline{\epsilon},r}$ be the function which is harmonic in $[\wh{\omega}_{j}^{\epsilon,\overline{\epsilon},r},+\infty) \times [0,i\pi]$ with Neumann boundary conditions on the horizontal parts of the strip boundary and at $+\infty$, and Dirichlet boundary conditions on $\wh{\omega}_{j}^{\epsilon,\overline{\epsilon},r} + [0,i\pi]$ with boundary values given by those of $\wh{h}_{j}^{\epsilon,\overline{\epsilon},r}$. Then $\wh{\psi}_{j}^{\epsilon,\overline{\epsilon},r}$ determines the harmonic part of $\wh{h}_{j}^{\epsilon,\overline{\epsilon},r}$ on $[\wh{\omega}_{j}^{\epsilon,\overline{\epsilon},r},+\infty) \times [0,i\pi]$. In particular, $\wh{\psi}_{j}^{\epsilon,\overline{\epsilon},r}$ determines the harmonic part of $\wh{h}_{j}^{t_{j}^{\epsilon,\overline{\epsilon},r}}$ on $\strip_+$. By applying the Markov property to $\wh{h}_{j}^{t_{j}^{\epsilon,\overline{\epsilon},r}}$, we obtain that conditional on $\wh{\psi}_{j}^{\epsilon,\overline{\epsilon},r}$, the restriction of $\wh{h}_{j}^{t_{j}^{\epsilon,\overline{\epsilon},r}}$ to $\strip_-$ has the law of a zero boundary $\text{GFF}$ on $\strip_-$ plus the harmonic extension to $\strip_-$ of the function whose boundary values coincide with those of $\wh{h}_{j}^{t_{j}^{\epsilon,\overline{\epsilon},r}}$. Since the zero boundary part is independent of of $\wh{h}_{1}^{\epsilon,\overline{\epsilon},r},\cdots,\wh{h}_{j - 1}^{\epsilon,\overline{\epsilon},r}$ and the harmonic part is determined by $\wh{\psi}_{j}^{\epsilon,\overline{\epsilon},r}$, we obtain that conditional on $\wh{\psi}_{j}^{\epsilon,\overline{\epsilon},r}, \wh{h}_{j}^{\epsilon,\overline{\epsilon},r}$ is independent of $\wh{h}_{1}^{\epsilon,\overline{\epsilon},r},\cdots,\wh{h}_{j -1 }^{\epsilon,\overline{\epsilon},r}$.
As we explained above,  if we show that the functions $\wh{\psi}_j^{\epsilon,\overline{\epsilon},r}$ become trivial as $\epsilon \to 0$,  then this will imply that the bubbles which are successively zipped in are independent quantum surfaces.  This is the purpose of the next lemma.

\begin{lemma}[$\text{\cite[Lemma~6.6]{dms2014mating}}$]
Fix $\overline{\epsilon} > 0, j \in \N$ and $r \in \R$. There exists a sequence $(\epsilon_{k})$ of positive numbers decreasing to $0$ such that $\wh{\psi}_{j}^{\epsilon_{k},\overline{\epsilon},r}$ a.s.\ converges to the $0$ function on $\strip_+$ as $k \to +\infty$ with respect to the topology of local uniform convergence modulo a global additive constant.
\end{lemma}
\begin{proof}
The proof follows essentially from the same argument used in the proof of \cite[Lemma~6.6]{dms2014mating}.  Nevertheless,  we are going to emphasize the parts of the proof where Lemma~\ref{lem:sle_separates_points} and Theorem~\ref{thm:law_of_sle_excursion} are used.

Fix $u > 0$ large. For $x > 0$ rational,  we let $\CV_{x}$ be the connected component of $\h \setminus \wt{\eta}^{\wt{T}_{u}}$ with $x$ on its boundary. Note that $\wt{\eta}^{\wt{T}_{u}}$ has the law of an $\SLE_{\kappa}(\rho)$ process and so by Lemma~\ref{lem:sle_separates_points} we obtain that a.s.\ such connected component exists for all $x \in \Q_+$.  We fix $\epsilon >0$ and we assume that we are working on the event that the quantum length of $\partial{\CV}_x \setminus \partial{\h}$ with respect to $\wt{h}^{\wt{T}_{u}}$ is at least $\epsilon > 0$.  Note that the probability of this event tends to 1 as $\epsilon \to 0$ for $x \in \Q_+$ fixed.  We let $\xi$(resp.\ $\zeta$) be the first (resp.\ last) time that $\wt{\eta}^{\wt{T}_{u}}$ hits a point on $\partial{\CV_x}$ and let $\varphi^{\epsilon} : \CV_x \rightarrow \strip$ be a conformal transformation with $\varphi^{\epsilon}(\wt{\eta}^{\wt{T}_{u}}(\xi)) = +\infty$ and $\varphi^{\epsilon}(\wt{\eta}^{\wt{T}_{u}}(\zeta)) = -\infty$. Note that $\varphi^{\epsilon}$ is determined up to a horizontal translation, so we can determine $\varphi^{\epsilon}$ by requiring that the $\gamma$-$\text{LQG}$ boundary measure of $(-\infty,0]$ associated with $\wt{h}^{\wt{T}_{u}} \circ (\varphi^{\epsilon})^{-1} + Q\log|((\varphi^{\epsilon})^{-1})'|$ is equal to $\epsilon$.

Let $\psi^{\epsilon}$ be the function which is harmonic in $\strip_+$ with Neumann boundary conditions on $\partial{\strip_+} \setminus [0,i\pi]$ and Dirichlet boundary conditions on $[0,i\pi]$ with boundary values given by the values of $\wt{h}^{\wt{T}_{u}} \circ (\varphi^{\epsilon})^{-1} + Q\log|((\varphi^{\epsilon})^{-1})'|$ on $[0,i\pi]$.  Then,  arguing as in the proof of \cite[Lemma~6.6]{dms2014mating},  we obtain that if we show that as $\epsilon \to 0$,  the law of $\psi^{\epsilon}$ converges weakly with respect to the topology of local uniform convergence in $\strip_+$ modulo a global additive constant to a function which is harmonic in $\strip_+$,  then we have that $\wh{\psi}_j^{\epsilon,\overline{\epsilon},r}$ will converge to a constant. Therefore,  it suffices to show the convergence of $\psi^{\epsilon}$ as $\epsilon \to 0$.

We are going to deduce this  from Theorem~\ref{thm:law_of_sle_excursion} and the independence of $\wt{\eta}^{\wt{T}_{u}}$ and $\wt{h}^{\wt{T}_{u}}$ (viewed modulo additive constant). For each $\delta > 0$ we let $\eta^{\delta}(t) = \delta^{-1} (\wt{\eta}^{\wt{T}_{u}}(\zeta - t) - \wt{\eta}^{\wt{T}_{u}}(\zeta))$. Then Theorem~\ref{thm:law_of_sle_excursion} combined with the time-reversal symmetries of $\SLE_{\kappa}(\rho_1 ; \rho_2)$ processes (\cite[Theorem 1.1]{ms2016imag2}) imply that the law of $\eta^{\delta}$ converges as $\delta \to 0$ to the law of a continuous curve in $\overline{\h}$. We note that $(\wt{h}^{\wt{T}_{u}},\wt{\eta}^{\wt{T}_{u}})$ and $(\wt{h}^0,\wt{\eta}^0)$ have the same law, where we view the distributions modulo additive constants. Let $h^{\delta}$ be the field which arises  by precomposing $\wt{h}^{\wt{T}_{u}}$ with $z \mapsto \delta (z + \wt{\eta}^{\wt{T}_{u}}(\zeta))$. Then $h^{\delta}$ (modulo additive constants) converges as $\delta \to 0$ to a free boundary $\text{GFF}$ on $\h$ by the scale and translation invariance of free boundary $\text{GFF}$s and the independence of $\wt{h}^{\wt{T}_{u}}$ (modulo additive constant) and $\wt{\eta}^{\wt{T}_{u}}$. Therefore $(\eta^{\delta},h^{\delta})$ converges in law as $\delta \to 0$ to the law of a pair consisting of a continuous curve in $\overline{\h}$ and an independent free boundary $\text{GFF}$ on $\h$.

For each $t \in (0,\zeta - \xi)$,  we let $\wt{\varphi}_t^{\epsilon}$ be the conformal transformation mapping $\h \setminus \wt{\eta}^{\wt{T}_u}([\zeta-t,\zeta])$ onto $\strip$ with $\wt{\eta}^{\wt{T}_u}(\zeta)$ mapped to $-\infty$ and $\wt{\eta}^{\wt{T}_u}(\zeta-t)$ mapped to $+\infty$,  and the horizontal translation fixed so that the $\gamma-\text{LQG}$ boundary length of $(-\infty,0]$ with respect to $\wt{h}^{\wt{T}_u} \circ (\wt{\varphi}_t^{\epsilon})^{-1} + Q \log |((\wt{\varphi}_t^{\epsilon})^{-1})'|$ is equal to $\epsilon$.  Then,  arguing as in the proof of \cite[Lemma~6.6]{dms2014mating},  we obtain that as $t,\epsilon \to 0$ with $\epsilon \to 0$ at a sufficiently fast rate relative to the rate at which $t \to 0$,  the law of the function which is harmonic in $\strip_+$ with Neumann boundary conditions on $\partial \strip_+ \setminus [0,\pi i]$ and with Dirichlet boundary conditions on $[0,\pi i]$ given by the values of $\wt{h}^{\wt{T}_u} \circ (\wt{\varphi}_t^{\epsilon})^{-1} + Q \log |((\wt{\varphi}_t^{\epsilon})^{-1})'|$ on $[0,\pi i]$ converges weakly with respect to the topology of local uniform convergence modulo global additive constant.

Finally,  arguing as in the last paragraph of the proof of \cite[Lemma~6.6]{dms2014mating} gives that the total variation distance between the fields $\wt{h}^{\wt{T}_u} \circ (\wt{\varphi}_t^{\epsilon})^{-1} + Q \log |((\wt{\varphi}_t^{\epsilon})^{-1})'|$ and $\wt{h}^{\wt{T}_u} \circ (\varphi^{\epsilon})^{-1} + Q \log |((\varphi^{\epsilon})^{-1})'|$ tends to $0$ as $t \to 0$ and $\epsilon$ tends to $0$ much faster than $t$.  This proves the convergence of $\varphi^{\epsilon}$ and hence the convergence of $\psi^{\epsilon}$.  This completes the proof of the lemma.
\end{proof}

Now we are ready to prove Theorem~\ref{thm:wedge_cutting}.

\begin{proof}[Proof of Theorem~\ref{thm:wedge_cutting}]
We will be brief since the proof is essentially the same as that of \cite[Theorem~6.1]{dms2014mating}.  Fix $u>0$ large and let $Y$ be a Bessel process of dimension $\delta(\kappa,\rho)$ as in~\eqref{eq:bessel_dimension}.  As in the proof of \cite[Theorem~6.1]{dms2014mating},  we can use the previous results to obtain an asymptotic coupling where the excursions of $Y$ away from $0$ up until some time $s_u$ each correspond to a connected component of $\h \setminus \wt{\eta}^{\wt{T}_u}$ which is to the left of $X_u$ and this correspondence is bijective,  where $X_u$ is the last intersection point of $\wt{\eta}^{\wt{T}_u}$ with $\R_+$ before time $\wt{T}_u$.  Note also that $Y$ encodes these bubbles as quantum surfaces from right to left.  Moreover,  by the construction of the coupling,  we have that $\wt{U}_t = V_{T_u-t}^{T_u}-W_{T_u-t}^{T_u}$ for all $t \in [0,T_u]$,  where $T_u$ is the right continuous inverse of the local time at $0$ of $V^{T_u} - W^{T_u}$ at $u$.  Hence $T_u = \wt{T}_u$ a.s.\  and $\wt{h}^{\wt{T}_u}$ is independent of $\wt{\eta}^{\wt{T}_u}$ with the law of the restriction of $\wt{h}^{\wt{T}_u}$ to the bubbles of $\h \setminus \wt{\eta}^{\wt{T}_u}([0,\wt{T}_u])$ ordered from right to left to be determined by the law of $Y$.  This completes the proof of the theorem.
\end{proof}

\subsection{Zipping according to quantum natural time}
\label{subsec:natural_time}

Suppose that we have the setup of Theorem~\ref{thm:wedge_cutting}.  Then Theorem~\ref{thm:wedge_cutting} implies that if $u>0$ is sufficiently large but fixed,  we have that the quantum surfaces which are cut out by $\eta|_{[0,T_u]}$ from $\infty$ and they are to the right of $\eta$ have the same Poissonian structure as those arising from a weight $\rho + 2$ quantum wedge.  Let $Y$ be the Bessel process encoding the law of the above bubbles when viewed as quantum surfaces modulo coordinate change.  Recall that $Y$ induces a quantum measure $m$ on the union of $\eta([0,T_u]) \cap \R_+$ with the set of self-intersection points of $\eta|_{[0,T_u]}$.  In Lemma~\ref{lem:main_lemma},  we will identify the conditional law of the pair $(h,\eta)$ when we sample $w$ from $m$ and condition on it,  and then show that when we zoom in near $w$ as in \cite[Proposition~4.7]{dms2014mating},  we have that the beaded quantum surface which consists of the connected components of $\h \setminus \eta$ lying to the right of $w$ converges to that of a quantum wedge of weight $\rho + 2$.  However,  when we condition on $w$,  we add to the field an extra log singularity at $w$ and the law of the curve is weighted by a certain Radon-Nikodym derivative.  This leads to the weight $\rho + 4$ that we eventually obtain when we zoom in near $w$.  Combining,  we obtain part (i) of Theorem~\ref{thm:quantum_natural_time_cutting}.  Finally,  part (ii) of Theorem~\ref{thm:quantum_natural_time_cutting} will follow by observing that the total variation distance between the laws of the pairs of the fields and curves tends to $0$ as $u \to \infty$ when we zoom in near $w$ and when we zoom in near the point obtained by shifting $w$ according to a fixed amount of quantum measure.

As we mentioned in the previous paragraph,  we will have to deal with fields perturbed by adding certain smooth functions.  Therefore,  it will be useful to know how the quantum measure described above changes when we add such a function.  This is the content of the next lemma.

\begin{lemma}[$\text{\cite[Lemma~6.19]{dms2014mating}}$]
\label{lem:smooth_perturbation}
Assume that we have the same setup as in Theorem~\ref{thm:wedge_cutting}(with the additive constant for $h$ fixed in the same way) and let $L$  be the local time at $0$ associated with the Bessel process $Y$ which encodes the weight $\rho + 2$ wedge corresponding to the bounded components of $\h \setminus \eta([0,T_{u}])$. For each smooth function $\phi$, let $Y^{\phi}$ be the corresponding process with $h$ replaced by $h + \phi$ (which we assume to be parameterized according to its quadratic variation). Let $\overline{\nu}$ (resp.\ $\overline{\nu}^{\phi}$) denote the empirical measure of the excursions made by $Y$ (resp.\ $Y^{\phi}$) from $0$ which correspond to the bubbles cut off $\infty$ by $\eta([T_{r},T_{u}])$. Let $d$ be the Bessel dimension of $Y$. We a.s.\ have (off a common set of measure $0$) for all $r \in [0,u]$ and all such smooth functions $\phi$ that the limit $m^{\phi}([r,u])$ of $\frac{\overline{\nu}^{\phi}(\CE(\epsilon))}{\nu_{d}^{\BES}(\CE(\epsilon))}$ as $\epsilon \to 0$ exists and (with $m = m^0$)
\begin{align*}
m^{\phi}([r,u]) = \int_{r}^{u}\exp{ \left(\frac{\gamma(2 - d)}{2}\phi(\eta(T_{v}))\right)dm(v)}.
\end{align*}
\end{lemma}
\begin{proof}
It follows from the proof of \cite[Lemma~6.19]{dms2014mating}.
\end{proof}

We note that Lemma~\ref{lem:smooth_perturbation} implies that the law of $Y^{\phi}$ is mutually absolutely continuous with respect to the law of $Y$ and moreover the local time of $Y^{\phi}$ at $0$ is defined for all smooth $\phi$ simultaneously.  Now,  we are ready to state and prove Lemma~\ref{lem:main_lemma}.

\begin{lemma}[$\text{\cite[Lemma~6.21]{dms2014mating}}$]
\label{lem:main_lemma}
Assume that we have the setup of Theorem~\ref{thm:wedge_cutting} (with the additive constant for $h$ fixed in the same way) and recall that 
\begin{align}
\label{eq:law_of_the _field}
    h = \wh{h} - \frac{2 + \rho}{\gamma}\log|\cdot|
\end{align}
where $\wh{h}$ is a free boundary $\text{GFF}$ on $\h$.
For each $0 \leq q < r \leq u$, we let $m_{q,r}$ be the restriction of $m$ as in Lemma~\ref{lem:smooth_perturbation} to subsets of $[q,r]$. Consider the law on $(w,h,\eta)$ triples given by $\CZ^{-1}_{q,r}dm_{q,r}dhd\eta$ where $dh$ denotes the law as in~\eqref{eq:law_of_the _field} and $\CZ_{q,r}^{-1}$ is a normalization constant. (Note that $m_{q,r}$ depends on $h$ and $\eta$.)
\begin{enumerate}[(i)]
\item Given $w$ and $\eta$, the conditional law of $h$ is equal to the law of $$\wh{h} - \frac{2 + \rho}{\gamma}\log|z| + \frac{\gamma(2 - d)}{2}G(z,\eta(T_w)) + \psi \circ f_{T_{u}}$$ where $\wh{h}$ is a free boundary $\text{GFF}$ on $\h$, $d$ is the dimension of the Bessel process $Y$, $\psi$ is a function which is harmonic outside of $\h \cap \partial{\D}$, and the additive constant is fixed in the same manner as for $h$.
\item Given $w$ and $\eta|_{[0,T_{w}]}$, the conditional law of $\eta|_{[T_{w},\infty)}$ is that of an $\SLE_{\kappa}(\rho)$ process in the unbounded component of $\h \setminus \eta([0,T_{w}])$ from $\eta(T_{w})$ to $\infty$ with a single boundary force point of weight $\rho$ located at $(\eta(T_{w}))^{+}$ weighted by the Radon-Nikodym derivative
\begin{align*}
\CZ^{-1}\exp\!\left(\frac{\gamma^{2}(2-d)^2}{8}\left(\int \int G(f_{T_u}^{-1}(x),f_{T_u}^{-1}(y))dxdy - 2 \int G(\eta(T_w),f_{T_u}^{-1}(x))dx\right)\right)
\end{align*}
where $G$ denotes the Neumann Green's function on $\h$ and $\CZ^{-1}$ is a normalizing constant and the integrals are all over $\partial \D \cap \partial \h$.
\item If one zooms in near $\eta(T_{w})$ as in part (ii) of \cite[Proposition 4.7]{dms2014mating}, then the law of the beaded surface which consists of the components of $\h \setminus \eta$ which are to the right of $\eta|_{[T_{w},\infty)}$ converges to that of a quantum wedge of weight $\rho + 2$.
\end{enumerate}
\end{lemma}

\begin{proof}
The proofs of parts (i) and (ii) follow from the proofs of parts (i) and (ii) respectively of \cite[Lemma~6.21]{dms2014mating}.  We note that the harmonic function $\psi \circ f_{T_u}$ obtained in part (i) can be defined on $\h$ since the Lebesgue measure of $\eta \cap \h$ is equal to $0$ a.s.  The latter follows since the law of $\eta$ on the intervals that it is not colliding with its force point is mutually absolutely continuous with respect to the law of an $\SLE_{\kappa}$ restricted to the corresponding intervals. So the above harmonic function extends to a function defined on $\h$.
 
Now we turn to the proof of part (iii). First of all, we recall that the beaded surface consisting of the bubbles parameterized by the components of $\h \setminus \eta([0,T_r])$ from right to left is given by that of a quantum wedge $\CW$ of weight $\rho + 2$, up to a given amount of local time $M$, which is a random time.  Without loss of generality, we can assume that $\CW$ is a bi-infinite wedge (which is just a concatenation of two independent weight $\rho + 2$-wedges). More precisely,  let $\CW_1,\CW_2$ be two independent weight $\rho + 2$-wedges with corresponding encoding Bessel processes $(Y_{t}^1)_{t \geq 0}$ and $(Y_{t}^2)_{t \geq 0}$ respectively. Then, their concatenation is the beaded surface $\CW$ encoded by the process $Y = (Y_t)_{t \in \R}$ defined by $Y_t = Y_{t}^1$ for $t \geq 0$ and $Y_t = Y_{-t}^2$ for $t \leq 0$.   For $\epsilon \in (0,1)$ we consider the conformal transformation $\varphi_{\epsilon} \colon \h \to \h$ with $\varphi_{\epsilon}(z) = z/\epsilon$ and the $\text{GFF}$ on $\h$, 
\begin{align*}
h^{\epsilon,C} = h \circ f_{T_w}^{-1} \circ \varphi_{\epsilon}^{-1} + Q\log  |(f_{T_w}^{-1})' \circ \varphi_{\epsilon}^{-1}| + Q\log (\epsilon) + \frac{C}{\gamma}
\end{align*}
and set $\epsilon(C) = \sup\{\epsilon \in [0,1] : h_{1}^{\epsilon,C}(0) = 0\}$ and $h^{C} = h^{\epsilon(C),C}$. Here, $h_{1}^{\epsilon,C}(0)$ is the average of the field $h^{\epsilon,C}$ on $\h \cap \partial  \D$. We want to show that the beaded surface with respect to $h^{C}$ which consists of the components of $\h \setminus \varphi_{\epsilon}(\eta^w)$ which are to the right of $\varphi_{\epsilon}(\eta^w)$ converges to that of a quantum wedge of weight $\rho + 2$ as $C \to \infty$, where $\eta^w(t) = f_{T_w}(\eta(t+T_w))$ for $t \geq 0$.

Next, we fix a large number $N$ and we assume that $X \in [0,N]$ is sampled uniformly at random. If we take $\CW$ and recenter it after shifting by $X$ units of local time, i.e., by considering the wedge encoded by the process $Y^T = (Y_{t}^T)_{t \in \R}$ with $Y_{t}^{T} = Y_{T - t}$ for $t \in \R$ and $T = T_X$, then the bubbles of $\CW$ going from left to right have the law of the bubbles in a weight $\rho + 2$-wedge because the law of the bubbles is invariant under the time-reversal and recentering (simply because the Poisson law has this property). Therefore the law of the beaded surface with respect to $h^C$ which consists of the bubbles of $\h \setminus \varphi_{\epsilon}(\eta^w)$ lying on $B(0,\epsilon(C)^{-1})$ can be sampled as follows. Let $\wt{\CW}$ be a bi-infinite wedge of weight as above coupled with $\eta$ such that the bubbles to the right of $\eta$ (starting from $\eta(T_r)$ and going from right to left) agree with the bubbles of the wedge encoded by $(\wt Y_t)_{t \geq 0}$, where $\wt  Y $ is the encoding process for $\wt{\CW}$, and up until $M$ units of local time. We fix $N \in \N$ large and we pick $X \in [0,N]$ uniformly at random. Then we condition on the event that $X \in [0,M]$ and we consider the beaded surface which is part of $\wt{\CW}$ and it is parameterized by the bubbles which are images under $f_{T_X}^{-1} \circ \varphi_{\epsilon}^{-1}$ of the bubbles of $\h \setminus \varphi_{\epsilon}(\eta^X)$ lying on the right of $\varphi_{\epsilon}(\eta^X)$. Then by letting $N \to \infty$, we have that the law of the above surface converges to the law of the surface parameterized by the bubbles of $\h \setminus \varphi_{\epsilon}(\eta^X)$ lying to the right of $\varphi_{\epsilon}(\eta^X)$.

Fix $K>0$ and let $\CF_{C,N}$ be the $\sigma$-algebra generated by the bubbles of the previous paragraph. Then the law of these bubbles when conditioning on $\CF_{C,N}$ has Radon-Nikodym derivative with respect to the law without conditioning equal to 
\begin{align*}
\frac{\p[ X \in [0,M]\,|\,\CF_{C,N}]}{\p[ X \in [0,M] ]}.
\end{align*}
Suppose that we have shown that the $\sigma$-algebra $\CF_{C,N}$ becomes trivial when $C \to \infty$, for every fixed $N \in \N$. Then, by the backwards martingale convergence theorem, we have that the Radon-Nikodym derivative converges to $1$ as $C \to \infty$ a.s. Also, the unconditional law of the bubbles is the same as if we replaced $h^C$ by a wedge of weight $\rho + 2$. By letting $N \to \infty$ we obtain the claim of part (iii).

Hence, in order to complete the proof of part (iii), we need to show that $\cap_{C>0} \CF_{C,N}$ is trivial for all $N \in \N$.  Let $\CF_{C,N}^1$ (resp.\ $\CF_{C,N}^2$) be the $\sigma$-algebra generated by the restriction of $h_X = h \circ f_{T_X}^{-1} + Q \log |(f_{T_X}^{-1})'|$ to $\h \cap B(0,K \epsilon(C))$ (resp.\ the bubbles cut off $\infty$ by $\eta^X$ which lie on $B(0,K \epsilon(C))$).  Set $\CG_{N,C} = \sigma(\CF_{C,N}^1,\CF_{C,N}^2)$ and $\CG_{N} = \cap_{C>0}\CG_{N,C}$. Then it suffices to show that $\CG_N$ is trivial for every fixed $K$.  For the latter, it suffices to prove that the law of the pair $((h_X,\phi),\eta^X)$ given $\CG_{C,N}$ converges as $C \to \infty$ to the unconditional law of $((h_X, \phi),\eta^X)$, for every fixed $\phi \in C_{0}^{\infty}(\h)$ with $\int_{\h}\phi(z)dz = 0$.  To show this,  first we note that the conditional law of the pair consisting of $h_X$ and $\eta^X$ given $X$ is equal to that of an $\SLE_{\kappa}(\rho)$ process $\eta$ in $\h$ with a single force point at $0_+$ weighted by the Radon-Nikodym derivative of part (ii) and $h$ can be expressed as $\wh{h} + \left(\frac{\rho +2 -\gamma^2}{\gamma}\right) \log |\cdot| + \psi \circ f_{T_{N-X}}$,  where $\wh{h}$ is a free boundary $\text{GFF}$ and the additive constant is taken so that the average of the field after applying the coordinate change with $f_{T_{N-X}}$ on $\h \cap \partial \D$ is equal to $0$.  This form of the conditional law follows by combining part (ii) with the proof of \cite[Theorem~6.16]{dms2014mating}.

Set $\wh{\psi} = \left( \frac{\rho + 2 - \gamma^2}{\gamma}\right) \log |\cdot| + \psi \circ f_{T_{N-X}}$ and note that conditional on $X$,  $\wh{\psi}$ is determined by $\eta^X$.  Note also that the Markov property of the $\text{GFF}$ implies that $\wh{h}$ can be decomposed as $\wh{h} = \wh{h}_C +\wh{ \mathcal{H}}_C$,  where $\wh{h}_C$ is a $\text{GFF}$ on $\h \setminus B(0,K \epsilon(C))$ with free boundary conditions on $\h \cap \partial B(0,K\epsilon(C))$ and $\wh{\mathcal{H}}_C$ is a harmonic function on $\h$ with boundary conditions given by those of $\wh{h}$ on $\h \cap \partial B(0,K \epsilon(C))$ and free boundary conditions on $\partial \h \setminus B(0,K\epsilon(C))$. Hence the conditional law of $(h_X,\phi)$ given $X$ and $\eta^X$ is that of a Gaussian random variable with mean $(\wh{\mathcal{H}}_C,\phi)$ and variance $\int_{\h}\int_{\h} \phi(y)G_C(y,z)\phi(z)dydz$,  where $G_C$ is the Green's function on $\h \setminus B(0,K \epsilon(C))$ with Dirichlet (resp.\ Neumann) boundary conditions on $\h \cap \partial B(0,K \epsilon(C))$ (resp.\ $\partial \h \setminus B(0,K \epsilon(C)))$.  Note that
\begin{equation}\label{eq:green_convergence}
\int_{\h}\int_{\h} \phi(y) G_C(y,z) \phi(z) dydz \to \int_{\h}\int_{\h} \phi(y) G(y,z) \phi(z) dydz \,\,\,\text{as}\,\,\,C \to \infty,
\end{equation}
since $G_C \to G$ as $C \to \infty$ locally uniformly.  Also,  given $X$,  $\eta^X$ and $\CG_{C,N}$,  the random variable $(\wh{\mathcal{H}}_C,\phi)$ is a Gaussian with mean zero and covariance given by $\int_{\h}\int_{\h} \phi(y) \cov(\wh{\mathcal{H}}_C(y),\wh{\mathcal{H}}_C(z)) \phi(z) dydz$.  Then,  using the explicit form of the Poisson kernel in $\h \setminus B(0,\epsilon)$,  we have that the latter covariance tends to $0$ as $C \to \infty$.  Therefore,  we obtain that the conditional law of $(h_X,\phi)$ given $(X,\eta^X)$ and $\mathcal{G}_{C,N}$ converges to the conditional law of $(h_X,\phi)$ given $(X,\eta^X)$ as $C \to \infty$.  Moreover,  $X$ is independent of $\mathcal{G}_{C,N}$ and the conditional law of $\eta^X$ given $X$ and $\mathcal{G}_{C,N}$ converges to the conditional law of $\eta^X$ given $X$ as $C \to \infty$.  The claim then follows since $(h_X,\phi)$ and $\eta^X$ are independent given $X$.  This completes the proof of part (iii).
\end{proof}

\begin{proof}[Proof of Theorem~\ref{thm:quantum_natural_time_cutting}.]
Now that we have proved Theorem~\ref{thm:wedge_cutting} and Lemma~\ref{lem:main_lemma},  the proof follows the same argument as that of the proof of \cite[Theorem~6.16]{dms2014mating}.
\end{proof}

\subsection{Quantum boundary length evolution}
\label{subsec:boundary_length_evolution}

Suppose that we have the setup of Theorem~\ref{thm:quantum_natural_time_cutting}.  In this subsection we are going to describe the law of the boundary length evolution of $\eta$ when it has the quantum natural time parameterization as in Theorem~\ref{thm:quantum_natural_time_cutting},  and hence proving Theorem~\ref{thm:boundary_length_evolution}.  As an intermediate step in the proof of Theorem~\ref{thm:boundary_length_evolution},  we are going to describe the Poissonian structure of the quantum lengths of the beads of a surface corresponding to a quantum wedge of weight $\rho + 2$ with $\kappa \in (0,4)$ and $\rho \in (\kappa/2-4,-2) \cap (-2-\kappa/2,-2)$.  Finally,  we are going to prove Corollary~\ref{cor:dim_of_boundary_intersection} as a consequence of Theorem~\ref{thm:boundary_length_evolution}.

We start by analyzing the Poissonian structure of the quantum lengths of the beads of a weight $\rho + 2$ quantum wedge.  Our strategy will be similar to that used to prove \cite[Proposition~4.18]{dms2014mating}.  In particular,  we prove the following.

\begin{theorem}
\label{thm:law_of_quantum_lengths}
Fix $\kappa$ and $\rho$ as above.  Let $\CW$ be a quantum wedge of weight $\rho + 2$ and let $h$ be the corresponding field. Let $Y$ be the Bessel process encoding $\CW$ and $(e_{j})_{j \in \N}$ be its excursions away from zero. Let $(h_{j})_{j \in \N}$ be the fields encoding the surfaces determined by the above excursions. Let also $t_{j}$ be the boundary quantum length associated with $h_{j}$ for all $j \in \N$. Then $(t_{j})$ has the law of a $\text{p.p.p.}$ with intensity measure given by $c( du \times t^{\alpha}dt)$ where $\alpha = -2 + \frac{2(\rho + 2)}{\kappa}$ and $c > 0$ is a constant depending only on $\alpha$ and $\rho$.
\end{theorem}
\begin{proof}
Note that we can sample $Y$ by first sampling a $\text{p.p.p.}$  $\Lambda^*$ according to $du \times \nu_{\delta}^*(dt)$ with $\nu_{\delta}^*(dt) = c_{\delta}^* t^{\delta - 3}dt$ for some constant $c_{\delta}^*>0$  depending only on $\delta$ and then sampling $Y$ by associating with each $(u,e*) \in \Lambda^*$ a $\BES^{\delta}$ excursion away from zero whose maximum value is given by $e^*$.  Here $\delta$ satisfies~\eqref{eq:bessel_dimension}. Note that for fixed $e^*$ the law of the latter can be sampled by joining back to back two independent Bessel processes with dimension $\wt{\delta} = 4 - \delta$ starting from $0$ and up until the first time they hit $e^*$ (see \cite[Remark~3.7]{dms2014mating}).  Let $\mu_{\delta}^{x}$ be the law on such paths for $x = e^* \in (0,+\infty)$.

Suppose that $X = (X_{t})_{0 \leq t \leq T}$ is sampled from $\mu_{\delta}^c$, where $T$ is the length of the excursion. Since Bessel processes satisfy Brownian scaling, we obtain that the process $\wt{X} = (c^{-1}X_{c^{2}t})_{0 \leq t \leq \wt{T}}$ with $\wt{T} = c^{-2}T$ is sampled from $\mu_{\delta}^{1}$. For $\epsilon \in (0,1)$ we set $T_{\epsilon} = \inf \{t \geq 0 : X_{t} \geq \epsilon \} $, $\wt{T}_{\epsilon} = \inf \{t \geq 0 : \wt{X}_{t} \geq \epsilon \}$ and 
\begin{align*}
        \sigma^{\epsilon}(t) = \inf \left\{s \geq 0 : \frac{4}{\gamma^{2}}\int_{0}^{s}\frac{1}{(X^{\epsilon}_{r})^{2}}dr \geq 2t \right \} \\
        \wt{\sigma}^{\epsilon}(t) = \inf \left\{s \geq 0 : \frac{4}{\gamma^{2}}\int_{0}^{s}\frac{1}{(\wt{X}^{\epsilon}_{r})^{2}}dr \geq 2t \right \},
\end{align*}
where $X^{\epsilon}_{t} = X_{T_{\epsilon} + t}$ for $0 \leq t \leq  T^{\epsilon}$ and $\wt{X}^{\epsilon}_{t} = \wt{X}_{\wt{T}_{\epsilon} + t }$ for $0 \leq t \leq \wt{T}^{\epsilon}$,  and $T^{\epsilon} = T - T_{\epsilon}$,  $\wt{T}^{\epsilon} = \wt{T} - \wt{T}_{\epsilon}$.  It follows from \cite[Proposition~3.4]{dms2014mating} that $Z^{\epsilon}$ (resp.\ $\wt{Z}^{\epsilon}$) evolves as a Brownian motion starting from $\frac{2}{\gamma}\log(\epsilon)$ and run twice the speed with drift $\frac{(2-\delta)}{2}\gamma$ up until the first time it hits $\frac{2}{\gamma}\log(c)$ (resp.\ $0$) and then it evolves independently as a Brownian motion starting from $\frac{2}{\gamma}\log(c)$ (resp.\ $0$) and run twice the speed with drift $\frac{ (\delta - 2)}{2}\gamma$, where $Z^{\epsilon} = \frac{2}{\gamma}\log(X_{\sigma^{\epsilon}(t)}^{\epsilon})$ and $\wt{Z}_{t}^{\epsilon} = \frac{2}{\gamma}\log(\wt{X}_{\wt{\sigma}^{\epsilon}(t)}^{\epsilon})$ for $t \geq 0$.  Also, we have that $\wt{T}_{\epsilon} = c^{-2}T_{\epsilon c}$ and hence $\wt{X}_{t}^{\epsilon} = c^{-1} X_{c^{2}t}^{\epsilon c}$ and $\wt{\sigma}^{\epsilon}(t) = c^{-2}\sigma^{\epsilon c}(t)$ for all $t \geq 0$. Moreover if $\sigma^{\epsilon} = \inf \{t \geq 0 : Z_{t}^{\epsilon} \geq \frac{2}{\gamma} \log(c) \}$ (resp.\ $\wt{\sigma}^{\epsilon} = \inf\{t \geq 0 : \wt{Z}_{t}^{\epsilon} \geq 0 \}$), then $Z_{\sigma^{\epsilon}+ t}^{\epsilon}$ for $t \geq - \sigma^{\epsilon}$ (resp.\ $\wt{Z}_{\wt{\sigma}^{\epsilon} + t}$ for $t \geq - \wt{\sigma}^{\epsilon}$) converges weakly as $\epsilon \to 0$ with respect to the local uniform topology to a process $Z$ (resp.\ $\wt{Z}$) indexed by $\R$. Note also that $\wt{Z}_t = Z_t - \frac{2}{\gamma}\log(c)$ for all $t \in \R$. 

The above observations imply that the total quantum length of a surface associated with a given $(u,e^*)$ is given by $e^*U_{e}$ where the $U_{e}$ are i.i.d. random variables indexed by the excursions of $Y$ away from $0$. Also the law of $U_{e}$ is given by the total quantum length of the random surface $\wt{h}$ sampled as follows:
\begin{enumerate}[(i)]
\item Let $X$ be a sample from $\mu_{\delta}^{1}$. The projection of $\wt{h}$ onto $\CH_{1}(\strip)$ is given by parameterizing $\frac{2}{\gamma}\log(X)$ to have quadratic variation $2dt$.
\item The projection of $\wt{h}$ onto $\CH_{2}(\strip)$ is given by taking an independent sample of the law of the corresponding projection onto $\CH_{2}(\strip)$ of a free boundary $\text{GFF}$ on $\strip$.
\end{enumerate}

Note that $ \alpha = \delta - 3$, where $\alpha$ is as in the statement of the theorem, and if $U$ is sampled from the law of $U_{e}$ then $\E[U^{-\alpha - 1}] = c < \infty$ since $-\alpha -1 = 1 - \frac{2(\rho + 2)}{\kappa} \in (0,\frac{4}{\kappa})$ and so \cite[Lemma~4.20]{dms2014mating} applies. Therefore $\{(u,e^*U_{e}) : (u,e^*) \in \Lambda^* \}$ is a $\text{p.p.p.}$ with intensity measure given by $c(du \times t^{\alpha}dt)$ by applying \cite[Lemma~4.19]{dms2014mating}. We also observe that the above $\text{p.p.p.}$ corresponds to the sequence of jumps of a positive stable subordinator with parameter $-\alpha -1 \in (1,2)$. This completes the proof.
\end{proof}
Now suppose that we have the same setup as in Theorem~\ref{thm:quantum_natural_time_cutting}. We note that the quantum natural time $(\qnt_{u})_{u \geq 0}$ can be expressed as follows. Let $(x_{j},t_{j})_{j \in \N}$ be the $\text{p.p.p.}$ encoding the local time at $0$ $L$ of $Y$ and let $(e_{j})_{j \in \N}$ be the corresponding excursions of $Y$ away from zero.  For all $t \geq 0$ we define $A(t)$ to be the set of $j \in \N$ such that the bubble corresponding to $e_{j}$ has been drawn completely by $\eta$ by capacity time $t$ and $x(t) = \sum_{j \in A(t)}t_{j}$. Then we set 
\begin{align*}
\qnt_{u} = \inf \{t \geq 0 : L_{x(t)} > u \}
\end{align*}
for all $u \geq 0$. 

Now we are ready to prove Theorem~\ref{thm:boundary_length_evolution}.  Recall the definitions of $A_u,  B_u,  X_u$ and $Z_u$ in Section~\ref{sec:introduction}.

\begin{proof}[Proof of Theorem~\ref{thm:boundary_length_evolution}.]
First,  we prove that $(X_u,Z_u)$ is an $\alpha$-stable L\'evy process.  Note that if we run the process for $u$ units of quantum natural time, the unexplored region (quantum surface parameterized by the unbounded component) is again a weight $\rho + 4$ wedge and the curve is an $\SLE_{\kappa}(\rho)$ (Theorem~\ref{thm:quantum_natural_time_cutting}). Hence the increments of $(X,Z)$ are independent and so it is a Levy process.

For all $\epsilon > 0$, let $\CE(\epsilon)$ be the set of excursions of $Y$ with time length at least $\epsilon$. For all $t \geq 0$ let $\eta_{t}(\CE(\epsilon))$ be the number of excursions completed by $Y$ by time $t$ which lie in $\CE(\epsilon)$. Let $\nu_{\delta}$ be the It\^o excursion measure of a $\BES^{\delta}$ process.  It follows from \cite[Lemma~6.18]{dms2014mating} that there exists a constant $C_{\delta} > 0$ depending only on $\delta$ such that $\nu_{\delta}(\CE(\epsilon)) = C_{\delta}\epsilon^{\frac{\delta}{2} - 1}$ for all $\epsilon > 0$.  We also have from \cite[Proposition~19.12]{kallenberg1997foundations} that $L_{t} = \lim_{\epsilon \to 0}\frac{\eta_{t}(\CE(\epsilon))}{\nu_{\delta}(\CE(\epsilon))}$ for all $t \geq 0$ a.s.  Fix $C \in \R$ and let $\wt{Y}$ be the $\BES^{\delta}$ process encoding $h + C$. Then it holds that $\wt{Y}_{t} = e^{\frac{C\gamma}{2}}Y_{e^{-\gamma C}t}$ for all $t \geq 0$. Let $\wt{L},\wt{q},\wt{\eta}_{t}$ be the corresponding quantities for $\wt{Y}$. Then $\wt{\eta}_{t}(\CE(\epsilon)) = \eta_{e^{-\gamma C}t}(\CE(\epsilon e^{-\gamma C}))$ and 
\begin{align*}
\frac{\nu_{\delta}(\CE(\epsilon))}{\nu_{\delta}(\CE(\epsilon e^{-\gamma C}))} = \frac{C_{\delta}\epsilon^{\frac{\delta}{2}-1}}{C_{\delta}\epsilon^{\frac{\delta}{2} - 1}e^{\gamma C (1 - \frac{\delta}{2})}} = e^{\gamma C (\frac{\delta}{2} - 1)}
\end{align*}
 and hence 
\begin{align*} 
 \frac{\wt{\eta}_{t}(\CE(\epsilon))}{\nu_{\delta}(\CE(\epsilon))} = \frac{\eta_{e^{-\gamma C}t}(\CE(\epsilon e^{-\gamma C}))}{\nu_{\delta}(\CE(\epsilon e^{-\gamma C}))}e^{\gamma C(1 - \frac{\delta}{2})}
\end{align*} 
converges to $e^{\gamma C(1 - \frac{\delta}{2})}L_t$ as $\epsilon \to 0$ for all $t \geq 0$ a.s.  Thus $\wt{L}_{t} = e^{\gamma C(1 - \frac{\delta}{2})}L_t$ for all $t \geq 0$ a.s. Moreover $\wt{x}(t) = \sum_{j \in A(t)}\wt{t}_{j} = e^{\gamma C}\sum_{j \in A(t)}t_{j}$ and so $\wt{L}_{\Tilde{x}(t)} = e^{\gamma C(1 - \frac{\delta}{2})}L_{e^{-\gamma C}\wt{x}(t)} = e^{\gamma C(1 - \frac{\delta}{2})}L_{x(t)}$ for all $t \geq 0$ a.s. Therefore we have that $\wt{q}_{u} = \qnt_{e^{\gamma C(\frac{\delta}{2} - 1)}u}$ for all $u \geq 0$ a.s. 

Let $(\wt{X},\wt{Z})$ be the process describing the change in the boundary length relative to time zero corresponding to the field $h + C$. Since the quantum length scales by $e^{\frac{\gamma C}{2}}$ when we add $C$ to the field, we obtain that 
\begin{align*}
(\wt{X}_{u},\wt{Z}_{u}) = \left(e^{\frac{\gamma C}{2}}X_{e^{\gamma C(\frac{\delta}{2} - 1)}u},e^{\frac{\gamma C}{2}}Z_{e^{\gamma C(\frac{\delta}{2} - 1)}u}\right)
\end{align*}
for all $u \geq 0$.  Set $\mu = e^{\gamma C(\frac{\delta}{2} - 1)}$. Then $\mu^{-\frac{1}{\alpha}} = e^{\frac{\gamma C}{2}}$ and so $(\wt{X}_{u},\wt{Z}_{u}) = (\mu^{-\frac{1}{\alpha}}X_{\mu u},\mu^{-\frac{1}{\alpha}}Z_{\mu u})$
for all $u \geq 0$. The claim of the theorem then follows since $(h,\eta)$ and $(h + C,\eta)$ have the same law for all $C \in \R$ (\cite[Proposition~4.8]{dms2014mating}) and $(\wt{X},\wt{Z})$ is determined by $(h + C,\eta)$.

We note that $(X,Z)$ makes a jump at time $u$ if and only if the right continuous inverse $T$ of $L$ makes a jump at time $u$. Note also that there is a correspondence between the discontinuities of $T$ and the excursions $(e_{j})_{j \in \N}$ made by $Y$ away from zero. Let $(h_{j})_{j \in \N}$ be the corresponding sequence of quantum surfaces. We also observe that when $(X,Z)$ makes a jump, both of $X$ and $Z$ make a jump with positive and negative sign respectively. The jump sizes of $X$ and $-Z$ are equal to the quantum lengths of $\R$ and $\R \times \{\pi\}$ respectively under the corresponding surface $(\strip,h_{j},-\infty,+\infty)$, where the opening (resp.\ closing) point of the pocket made by $\eta$ which corresponds to the excursion $e_j$ is mapped to $-\infty$ (resp.\ $+\infty$).

We are now ready to describe the law of the boundary length evolution in the case where $\rho = \kappa - 4$ and $\kappa \in (\frac{4}{3},2) $.
Since $\rho = \kappa - 4$, we have that the corresponding surfaces have the law of i.i.d. quantum disks. By the resampling property characterizing the two marked points of a unit boundary length quantum disk (\cite[Proposition~A.8]{dms2014mating}) and the Poissonian structure of the sequence of their quantum boundary lengths (Theorem~\ref{thm:law_of_quantum_lengths}), we obtain that the sequence of jumps of $(X,-Z)$ can be sampled as follows:
\begin{enumerate}[(i)]
\item Firstly, we sample a $\text{p.p.p.}$ $\Lambda = \{(s_{j},t_{j})\}_{j \in \N}$ on $\R_+ \times \R_+$ with intensity measure $c(du \times t^{\alpha}dt)$ where $\alpha$ and $c$ are  as in Theorem~\ref{thm:law_of_quantum_lengths}. In our case, we have that $\alpha = -\frac{4}{\kappa}$.
\item Then we sample independently a sequence of i.i.d. random variables with the uniform law on $[0,1]$, $\{U_{j}\}_{j \in \N}$ and we consider 
\begin{align*}
\Lambda^* = \{(t_{j}u_{j} , (1 - u_{j})t_{j}) : (s_{j},t_{j}) \in \Lambda \}_{j \in \N}
\end{align*}
\end{enumerate}
Note that $\wt{\Lambda} = \{(s_{j},t_{j},u_{j})\}_{j \in \N}$ is a $\text{p.p.p.}$ on $\R_+ \times \R_+ \times [0,1]$ with intensity measure given by $\nu = c(du \times t^{\alpha}dt) \times \mu$, where  $\mu$ is the law of a uniform random variable on $[0,1]$. Consider the function $F \colon \R_+ \times \R_+ \times [0,1] \to  \R_+ \times \R_+$ with $F(s,t,u) = (tu,t(1 - u))$. Then $\Lambda^*$ is a $\text{p.p.p.}$ on $ \R_+ \times  \R_+$ with intensity measure given by $\nu_{*}F$. Set $\wt{\nu}(dx,dy) = c(x + y)^{\alpha - 1}dxdy$ to be defined on $\R_+ \times \R_+$. It is easy then to check that $\nu_{*}F([x_{1},x_{2}] \times [y_{1},y_{2}]) = \wt{\nu}([x_{1},x_{2}] \times [y_{1},y_{2}])$ for all $0 < x_{1} < x_{2}$ and $0 < y_{1} < y_{2}$. Therefore we have that $\nu_{*}F = \wt{\nu}$ and so we have described the law of $(X,-Z)$ since it is determined by the law of the jumps.  This completes the proof of the theorem.
\end{proof}

Now that we have completed the proof of Theorem~\ref{thm:boundary_length_evolution},  we are ready to prove Corollary~\ref{cor:dim_of_boundary_intersection}.

\begin{proof}[Proof of Corollary~\ref{cor:dim_of_boundary_intersection}.]
Suppose that we have the setup of Theorems~\ref{thm:quantum_natural_time_cutting} and~\ref{thm:boundary_length_evolution} where the $\SLE_{\kappa}(\rho)$ process $\eta$ is drawn on top of an independent quantum wedge $\CW = (\h,h,0,\infty)$ of weight $\rho + 4$ and we parameterize $\eta$ by quantum natural time with respect to $\CW$.  Let $(X,Z)$ be the pair encoding the change of the quantum lengths of the left and right outer boundaries as in Theorem~\ref{thm:boundary_length_evolution}.  Then Theorem~\ref{thm:boundary_length_evolution} implies that $(X,-Z)$ has the law of an $\alpha$-stable L{\'e}vy process where $\alpha$ satisfies~\eqref{eqn:alpha_value}.  We parameterize $\R_+$ according to quantum length with respect to $h$ and set $I_t = \inf_{s \in [0,t]} Z_s$,  $S_t = \sup_{s \in [0,t]}(-Z_s)$ for all $t \geq 0$.  Let $L = (L_t)$ be the local time of $Z-I$ at $0$ which is the same as the local time of $S+Z$ at $0$.  Let also $L^{-1}$ be the right-continuous inverse of $L$.  Then,  it follows from \cite[Chapter~VIII]{bertoin1996levy} that $L^{-1}$ is a stable subordinator with index $1 - \frac{1}{\alpha}$.  Note that the ladder height process is defined by $H_t = S_{L_t^{-1}}$ for all $t \geq 0$ (see \cite[Chapter~VI]{bertoin1996levy}).  Moreover \cite[Lemma~1,  Chapter~VIII]{bertoin1996levy} implies that $H$ is a L{\'e}vy stable process with index $\alpha-1 \in (0,1)$ and hence it is an $(\alpha-1)$-stable subordinator.  We note that $Z$ achieves a record minimum at time $t$,  i.e.,  $Z_t = I_t$,  if and only if $x = \eta(t) \in \R_+$ and then $x = -Z_t$.  It follows that $\{\nu_h([0,x]) : x \in \eta \cap \R_+\} = \overline{\{H_t : t \geq 0\}}$.  We also note that $H$ has Laplace exponent $\Phi$ given by $\Phi(\lambda) = \lambda^{\alpha-1}$ for all $\lambda > 0$.  Therefore,  combining with \cite[Theorem~15,  Chapter~III]{bertoin1996levy},  we obtain that 
\begin{equation}\label{eqn:hausdorff_dim_equality}
\text{dim}_{\mathcal{H}}(\{H_t : t \geq 0\}) = \text{dim}_{\mathcal{H}}(\{\nu_h([0,x]) : x \in \eta \cap \R_+\}) = \alpha-1 = -\frac{2(\rho+2)}{\kappa}\,\,\text{a.s.}
\end{equation}
where $\text{dim}_{\mathcal{H}}$ denotes Hausdorff dimension.  We note that it follows from \cite[Theorem~4.1]{rhodes2008kpz} that if $\wt{h}$ is a free boundary $\text{GFF}$ on $\h$ and $K \subseteq [0,\infty)$ is a deterministic set,  then if $\wh{K} = \{\nu_{\wt{h}}([0,x]) : x \in K\}$,  it holds that
\begin{equation}\label{eqn:kpz_formula_equation}
\text{dim}_{\mathcal{H}}(K) = \left(1+\frac{\gamma^2}{4}\right) \text{dim}_{\mathcal{H}}(\wh{K}) -\frac{\gamma^2}{4}\text{dim}_{\mathcal{H}}(\wh{K})^2\,\,\text{a.s.}
\end{equation}
Note that if we consider $\CW$ with the circle average embedding,  then the law of the restriction of $h$ to any subdomain of $\h$ which is bounded away from $0,\infty$ and $\h \cap \partial \D$ is mutually absolutely continuous with respect to the law of the corresponding restriction of a free boundary $\text{GFF}$ on $\h$ normalized to have average zero on $\h \cap \partial \D$.  Since $\eta$ is independent of $h$,  we obtain that~\eqref{eqn:kpz_formula_equation} holds a.s.\  when $K$ is replaced by $\eta \cap \R_+$ and $\wt{h}$ is replaced by $h$.  Then,  \eqref{eqn:dim_of_intersection_with_R_+} follows by combining with~\eqref{eqn:hausdorff_dim_equality}.  This completes the proof of the corollary.
\end{proof}

\bibliographystyle{abbrv}
\bibliography{biblio}

\end{document}